\documentclass[AIF,Unicode,manuscript]{cedram}

\usepackage{amsmath,amsfonts, amsthm,color,amssymb,mathrsfs,pgfplots,enumitem}   
\usepackage[numbers,comma,square,sort&compress]{natbib}
\usepackage{pdfcomment}
\usepackage{bm}
\usepackage{scalerel}    %%%%  used first for scaling \bullet  
\usepackage{stmaryrd}   %%%% for integer interval notation 
\usepackage{mathabx}   %% made \widecheck work  
\usepackage{hyphenat}
\usepackage[utf8]{inputenc}
\usepackage{caption}
\usepackage{subcaption}
\usepackage{a4wide}
\usepackage{accents}
\usepackage{tikz}
\usetikzlibrary{arrows.meta,decorations.markings}
\usepackage{graphicx}
\usepackage{chngcntr}
\usepackage{float} %needed to fix the figures

\usepackage{rotating,centernot,cancel}    %% to manipulate symbols (Timo 2019-03-25) 
\usepackage{geometry}
\numberwithin{figure}{section}
\numberwithin{equation}{section}

\RequirePackage{color}
\definecolor{my-link}{rgb}{0.5,0.0,0.0}
\definecolor{my-blue}{rgb}{0.0,0.0,0.6}
\definecolor{my-red}{rgb}{0.5,0.0,0.0}
\definecolor{my-green}{rgb}{0.0,0.5,0.0}
\definecolor{nicos-red}{rgb}{0.75,0.0,0.0}
\definecolor{really-light-gray}{gray}{0.8}
\definecolor{darkgreen}{rgb}{0.0,0.5,0.0}
\definecolor{darkblue}{rgb}{0.0,0.0,0.3}
\definecolor{nicosred}{rgb}{0.65,0.1,0.1}
\definecolor{light-gray}{gray}{0.7}
\definecolor{sussexg}{rgb}{0,0.7,0.7}
\definecolor{sussexp}{rgb}{0.4,0,0.4}
\definecolor{nicos-red}{rgb}{0.75,0.0,0.0}
\definecolor{carmine}{rgb}{0.59, 0.0, 0.09}
\definecolor{dartmouthgreen}{rgb}{0.05, 0.5, 0.06}

\newcommand{\bbE}{\mathbb{E}}

\newcommand{\bbN}{\mathbb{N}}

\newcommand{\bbP}{\mathbb{P}}

\newcommand{\bbR}{\mathbb{R}}

\newcommand{\bbZ}{\mathbb{Z}}

\newcommand{\sF}{\mathcal{F}}
\newcommand{\sG}{\mathcal{G}}
\newcommand{\sH}{\mathcal{H}}

\newcommand{\sU}{\mathcal{U}}

\def\w{\omega}
\def\zero{\mathbf{0}}
\def\ek{\mathbf{e}}
\font \mymathbb = bbold10 at 11pt
\newcommand{\one}{\mbox{\mymathbb{1}}}

\def\ri{\mathrm{ri\,}}
\providecommand{\abs}[1]{\vert#1\vert}
\newcommand{\fl}[1]{\lfloor{#1}\rfloor} 
\newcommand{\flint}[1]{[\![{#1}]\!]} 
\def\xip{\overline{\xi}}
\def\pxi{\overline{p}}
\def\expnt{a_0}
\def\expntt{a_1}
\def\fluctexp{\nu}

\renewcommand{\tilde}{\widetilde}
\renewcommand{\hat}{\widehat}

\DeclareMathOperator{\Geom}{Geom}
\DeclareMathOperator{\Ber}{Ber}

\DeclareMathOperator{\Prob}{\bbP}
\renewcommand{\P}{\Prob}
\DeclareMathOperator{\E}{\bbE}
\def\e{\varepsilon}

\newcommand{\N}{\mathbb{N}}
\newcommand{\RR}{\mathbb{R}}
\newcommand{\NN}{\mathbb{N}}
\newcommand{\ZZ}{\mathbb{Z}}

\newcommand\mb[1]{\mathbf{#1}}
\newcommand\mc[1]{\mathcal{#1}}

\newcommand\eqd{\mathrel{\stackrel{\makebox[0pt]{\mbox{\normalfont\tiny d}}}{=}}}

\makeatletter
\newcommand*\bigcdot{\mathpalette\bigcdot@{.5}}
\newcommand*\bigcdot@[2]{\mathbin{\vcenter{\hbox{\scalebox{#2}{$\m@th#1\bullet$}}}}}
\makeatother

\newcommand\bbullet{{\raisebox{0.5pt}{\scaleobj{0.6}{\bullet}}}} %% looks good for paths x_\bbullet 
\newcommand\gpp{\gamma} %The shape function without pp subscript 
\DeclareMathOperator{\Exit}{Exit}

\newcommand\geod[1]{\pi_{#1}} %A generic geodesic without a root or direction. The argument is the index.

\newcommand\ximn{\zeta}  % normalizes the point (m,n) by its L1 norm to be in the U-set.  So \zeta(m,n) = 1/(m+n) (m, n)
\newtheorem{theorem}{Theorem}[section]
\newtheorem{lemma}[theorem]{Lemma}
\newtheorem{propo}[theorem]{Proposition}
\newtheorem{cor}[theorem]{Corollary}
\newtheorem{definition}[theorem]{Definition}
\newtheorem{assumption}[theorem]{Assumption}
\numberwithin{equation}{section}

\theoremstyle{remark}
\newtheorem{remark}[theorem]{Remark}

\title[No bi-infinite geodesics in Geometric LPP]{Non-existence of non-trivial bi-infinite geodesics\\ in Geometric Last Passage Percolation}

\author{\firstname{Sean} \lastname{Groathouse}}
\address{University of Utah\\  Mathematics Department\\ 155 S 1400 E\\   Salt Lake City, UT 84112\\ USA}
\email{sean.groathouse@utah.edu}
\author{\firstname{Christopher} \lastname{Janjigian}}
\address{Purdue University\\  Mathematics Department\\  150 N. University Street\\   West Lafayette, IN 47907\\ USA}
\email{cjanjigi@purdue.edu}
\thanks{C.\ Janjigian was partially supported by National Science Foundation grant DMS-2125961}

\author{\firstname{Firas} \lastname{Rassoul-Agha}}
\address{University of Utah\\  Mathematics Department\\ 155 S 1400 E\\   Salt Lake City, UT 84112\\ USA}
\email{firas@math.utah.edu}
\thanks{F.\ Rassoul-Agha and S.\ Groathouse were partially supported by NSF grants DMS-1407574 and DMS-1811090}

\thanks{Last revised: November 30, 2021}   %day submitted to arxiv


\begin{abstract} 
We show non-existence of non-trivial bi-infinite geodesics in the solvable last-passage percolation model with i.i.d.\ geometric weights. This gives the first example of a model with discrete weights where non-existence of non-trivial bi-infinite geodesics has been proven. Our proofs rely on the structure of the increment-stationary versions of the model, following the approach recently introduced by Bal\'azs, Busani, and Sepp\"al\"ainen. %\cite{Bal-Bus-Sep-20}. 
Most of our results work for a general weights distribution and we identify the two properties of the stationary distributions which would need to be shown in order to generalize the main result to a non-solvable setting.
\end{abstract}

\begin{document}
\maketitle

\section{Introduction}
This paper considers directed last-passage percolation (LPP), which is a prototypical example of a lattice interface growth model in 1+1 dimensions. Such lattice growth models have played a central role in the development of modern probability over the last fifty years, with 1+1 dimensional LPP rising in importance over recent decades as a member of the Kardar-Parisi-Zhang (KPZ) universality class. See the recent surveys \cite{Auf-Dam-Han-17,Cor-12,Cor-16,Hal-Tak-15,Qua-Spo-15,Qua-12}.

Last-passage percolation, along with closely related models like first-passage percolation, directed polymers, and certain stochastic Hamilton-Jacobi equations, have interpretations as a kind of directed analogue of a metric. For this point of view, see for example the discussion in \cite{Dau-Ort-Vir-19-,Dau-Vir-21-} and also \cite{Bak-Kha-18}. This connection is exact in the case of first-passage percolation, which genuinely describes a random metric on the lattice. In these interpretations, it is often possible to interpret the solution in terms of random paths, which are variously called geodesics, random polymers, or characteristics, among others. The structure of these random paths has been a major focus of research in the field. 

This project considers a particular subset of questions related to bi-infinite geodesics, which are bi-infinite paths with the property that the restriction of the path to any finite subpath is a geodesic between its endpoints. The study of such paths traces back at least to a question Furstenburg posed to Kesten on first-passage percolation \cite[p.~258]{kes-86-stflour}, where the existence of such paths is equivalent to the existence of non-trivial ground states in the ferromagnetic Ising model with random impurities \cite[p.~105]{Auf-Dam-Han-17}. It is generally believed that in models of the type we consider, non-trivial bi-infinite geodesics should not exist, for reasons we will discuss shortly. Much of the mathematical progress toward proving this conjecture traces back to the seminal ICM note of Newman \cite{New-95}, which instigated a fruitful line of research on the structure of semi-infinite and bi-infinite geodesics in first-passage percolation \cite{How-New-97,How-New-01,Lic-New-96}. These ideas motivated subsequent work on first \cite{Ahl-Hof-16-,Hof-05,Hof-08,Dam-Han-14,Dam-Han-17} and last passage percolation \cite{Cat-Pim-12,Cat-Pim-13,Geo-Ras-Sep-17-ptrf-1,Geo-Ras-Sep-17-ptrf-2,Jan-Ras-Sep-21-}, as well as related models \cite{Bak-Cat-Kha-14,Bak-Li-18,Bak-Li-19, Jan-Ras-20-aop}.

One of the main predictions of the KPZ class concerns the structure of fluctuations of analogues of geodesics, and in particular their characteristic $2/3$ scaling exponent in 1+1 dimensions. At an ACM workshop in 2015 \cite[Section 4.5.1]{Auf-Dam-Han-17}, Newman gave a heuristic argument that in dimensions for which this transversal fluctuation exponent is greater than $1/2$, non-trivial bi-infinite geodesics should not exist. Two different implementations of this heuristic were recently carried out in the last-passage percolation model with i.i.d.~exponential weights by Basu, Hoffman, and Sly \cite{Bas-Hof-Sly-21-} and Bal\'azs, Busani, and Sepp\"al\"ainen \cite{Bal-Bus-Sep-20}. The former implementation uses integrable methods heavily, while the latter relies on the structure of the increment-stationary distributions for the model. Both approaches rely in essential ways on the exact solvability of the exponential last-passage percolation model. A general version of Newman's argument under strong conditions on the passage time was recently implemented by Alexander in \cite{Ale-20-}. Perhaps the strongest unconditional result in this direction is the recent \cite{Bri-Dam-Han-20-}.

The present paper abstracts the approach of \cite{Bal-Bus-Sep-20}. We consider a novel implementation of the argument to the last-passage percolation model with geometric weights, giving an example of a model with discrete weights for which non-trivial bi-infinite geodesics do not exist. More broadly, we re-cast the approach of \cite{Bal-Bus-Sep-20} without reference to particular weight distributions and identify two properties of increment-stationary distributions, recorded as Assumptions \ref{exitPointTailBound} and \ref{ass:RW} below, which would need to be proven in order to realize this program for non-integrable models. After introducing each of these assumptions, we discuss the types of hypotheses on the last-passage percolation model which would need to be proven in order to verify these conditions. It is noteworthy that it is known from \cite{Geo-Ras-Sep-17-ptrf-1, Jan-Ras-20-aop} that the increment stationary models we discuss in these assumptions have been shown to exist generally.

Our main result, Theorem \ref{main-thm}, shows that under our abstract hypotheses, non-trivial bi-infinite geodesics do not exist.  We verify our conditions in the geometric model using exact solvability. Along the way, we also prove some novel results about geometric last-passage percolation in order to verify our hypotheses. In particular, we prove a new, sharp bound for exit times in increment-stationary geometric last-passage percolation following a strategy recently introduced in \cite{Emr-Jan-Sep-20-,Emr-Jan-Sep-21-}. This is recorded as Theorem \ref{geometricExitPointTailBound} below.

\section{Setting and the main result}
%Fix a parameter $r \in (0,1)$ and 
Let $\Omega = \bbR^{\bbZ^2}$ and equip it with the product topology and its Borel $\sigma$-algebra $\sF$. A generic element in $\Omega$ is denoted by $\w$ and is sometimes referred to as an \textit{environment}. 
Let $(\w_x)_{x\in\ZZ^2}$ be the coordinates of $\w$. $\w_x$ is referred to as the \textit{weight} at $x$. 
We assume the following throughout the paper: we are given a probability measure $\P$ on $(\Omega,\sF)$ such that
\begin{align}\label{stand:assumption}
(\w_x)_{x\in\ZZ^2}\text{ are i.i.d.\ under }\P,\quad
\exists \epsilon>0:\ \E[|\w_0|^{2+\epsilon}]<\infty,\quad\text{and}\quad
{\mathbb V}\mathrm{ar}(\w_0)>0.
\end{align}

Denote by $T=\{T_x : x \in \bbZ^2\}$ the natural group of shift operators on $\Omega$, which satisfy 
$(T_y\w)_x = \w_{x+y}$ for $x,y \in \bbZ^2$. 
Given sites $x,y \in \bbZ^2$ with  $x\leq y$ (coordinatewise), an \textit{up-right path} from $x$ to $y$ is a sequence of lattice vertices with increments in the set $\{\ek_1,\ek_2\}$, the canonical basis of $\bbR^2$.
The collection of up-right paths from $x$ to $y$ is denoted by $\Pi_x^y$. The \textit{passage time} (or the weight) of an up-right path $\pi\in\Pi_x^y$ is the sum of the weights of the vertices of the path: %, excluding the endpoint: 
$\sum_{v\in\pi}\w_v$. 
For $x\le y$ in $\ZZ^2$, the (bulk) \textit{last-passage time} from $x$ to $y$ is defined to be
\begin{align}
    G_{x,y}(\w) &= \max_{\pi \in \Pi_x^y} \sum_{v \in \pi} \w_{v}.\label{eq:LPPdef}
\end{align}
%when $x\le y$ and $G_{x,y}(\w)=-\infty$ otherwise. 
In particular, $G_{x,x}(\w)=\w_x$.
%Although immaterial as far as our results are concerned, 
%the weight at the endpoint is omitted because that makes some of our formulas take a nicer form. 
As is customary in probability theory we often omit the $\w$ from the argument of $G_{x,y}$.

A path $\pi\in\Pi_x^y$ 
which realizes the maximum in \eqref{eq:LPPdef} is called a \textit{geodesic}. This terminology is by analogy with the related model of first-passage percolation where $G_{x,y}$ defines a random pseudo-metric on $\bbZ^2$.
Geodesics are unique when $\P(\w_0\le t)$ is continuous in $t$, but when this distribution function is not continuous, multiple geodesics can exit. 

Our main interest in the present paper is in the structure of \textit{bi-infinite geodesics}, which we now define.
\begin{definition}
A bi-infinite up-right path $\pi_{-\infty:\infty} = (\pi_n)_{n\in\bbZ}$ is said to be a \textit{bi-infinite geodesic} if for every $m<n$ in $\bbZ$, the segment $\pi_{m:n}$ is a geodesic between $\pi_m$ and $\pi_n$.
\end{definition}

For each $x \in \bbZ^2$ and $k \in \{1,2\}$, the path $x+\bbZ \ek_k=(x + j \ek_{k} : j \in \bbZ)$ is a \textit{trivial} bi-infinite geodesic. This is because there is only one up-right path between any two sites on such a path. Bi-infinite geodesics which are not of this form are said to be \textit{non-trivial}. Our main theorem says that non-trivial bi-infinite geodesics do not exist 
when the weights are geometric random variables.

\begin{theorem}\label{thm:noBiInfinites}  Assume $\w_0$ is a $\Geom(r)$ random variable for some $r\in(0,1)$. Then with $\P$-probability one there are no non-trivial bi-infinite geodesics.
\end{theorem}

%\mathnote{Quick description of FPP: related model, min instead of max in \ref{eq:LPPdef} and, more importantly, paths have increments in $\{\pm\ek_1,\pm\ek_2\}$. Relation of ground states in random field Ising to bi-infinites in FPP. Site Michael's book.}

%\mathnote{History of the problem}

%\mathnote{Mention Newman's heuristic and Alexander's arguments}

%\mathnote{The strategy: We prove that given sufficiently decaying tail bounds on how long the geodesic sticks to the boundary, in the stationary model, we can control the fluctuations of geodesics in the i.i.d.\ model and use that to show that the fluctuations are so large that there cannot be bi-infinites. We then show that these tail bounds are satisfied in the case of geometric weights.  This strategy follows Timo etal (exponential weights and then positive temperature version). Though we do the tail bound differently, right? We follow the paper with Elnur for that?}

Our main result, Theorem \ref{thm:noBiInfinites}, follows from Theorem \ref{main-thm}, which applies to a general weight distribution. 
 It requires two assumptions, which are then verified (in the appendix) to hold when $\w_0$ is geometrically distributed. 
  These are the only two places where solvability is used. We include the following comments on our use of solvability:
  
a) The independence property in Theorem \ref{thm:coupledGeomLPP}\eqref{coupledGeomLPPc} and the explicit knowledge of marginal distributions in Theorem \ref{thm:coupledGeomLPP}\eqref{coupledGeomLPPd} are used in the proof of the tail bound in Theorem \ref{geometricExitPointTailBound}, which verifies our Assumption \ref{exitPointTailBound}. These methods seem unlikely to generalize, as they rely on a certain structure of Radon-Nikodym derivatives of the marginal distributions which is satisfied for solvable poylmer and percolation models, but not general distributions. See \cite{Emr-Jan-Sep-20-,Emr-Jan-Sep-21-,Emr-Jan-Xie-21-}. Using these methods, the bound we prove is sharp, with cubic exponential decay. This is stronger than is necessary for the rest of our arguments: Assumption \ref{exitPointTailBound} only asks for a polynomial bound with exponent strictly greater than two.
%\mathnote{Maybe explain that it is reasonable to expect that the technology is not too far from proving this kind of decay in general? Do we believe that?}
%See Remark \ref{rk:tail}.  
%\mathnote{If possible, explain where the crucial step is where the geometric assumption is used in the proof of the theorem and what may be the way to go to general weights.}

b) The independence property in  Theorem \ref{thm:coupledGeomLPP}\eqref{coupledGeomLPPb} is used when verifying Assumption \ref{ass:RW}. This is an assumption concerning certain random walks which, in a general setting, would be built out of using the Busemann process constructed in \cite{Geo-Ras-Sep-17-ptrf-1,Jan-Ras-20-aop}. In the setting with geometric weights, this independence allows us to turn the probability of an intersection in \eqref{bd:RW} to a product of probabilities. Moreover, it is used to deduce that the random walks in \eqref{Sdef} have independent increments. Our key random walk estimate, Lemma \ref{lem:rwBound}, assumes that the random walk increments are independent for this reason. For a general weight distribution, we expect that the increments of the associated random walks in question are mixing, but not independent. A version of Lemma \ref{lem:rwBound} can be expected to hold for such random variables, subject to some extra moment hypotheses.  %\mathnote{Maybe say something about how one would go about it, i.e.\ replacing the Berry-Esseen estimate}
%Throughout the paper, we emphasize the specific places where we use exact solvability of the model and identify where weaker estimates could be sufficient. \firasnote{If they are not many, maybe we can point them out right here. It seems that the geometric assumption is only needed for Theorem \ref{geometricExitPointTailBound}.}
%\reversemarginpar\note{F: Agree with a) and b)? C: Rewrote these a bit, please review.}

%\mathnote{Emphasize that general weights is not very far behind and point out some of the progress in the general weights case, including standard FPP.} 
\medskip

{\bf Organization of the paper:} Section \ref{sec:bdry} introduces boundary models. In Section \ref{sec:geo_fluct} 
we derive geodesic fluctuation bounds under Assumption \ref{exitPointTailBound}.
Section \ref{sec:nonexist} has the proof of the nonexistence of bi-infinte geodesics, under Assumptions \ref{exitPointTailBound} and \ref{ass:RW}. Appendix \ref{Busemann general} recalls results that provide the boundary weights for the boundary models needed for the proofs. Sections \ref{sec:bdry}-\ref{sec:nonexist} and Appendix \ref{Busemann general} are for general weights and can be read independently. The rest of the appendixes deal with the case of geometric weights and can each be read independently.
Theorem \ref{geometricExitPointTailBound} in Section \ref{sec:pfExit}
verifies that Assumption \ref{exitPointTailBound} holds in the case of geometric weights.
 Lemma \ref{lm:assRWholds} uses the extra independence structure in Theorem \ref{thm:coupledGeomLPP} and the random walk estimates in Lemma \ref{lem:rwBound} to verify that Assumption \ref{ass:RW} holds in the case of geometric weights.

\subsection{Notation}  
$\bbN$ denotes the natural numbers $\{1,2,\dotsc\}$, $\bbZ$ is the set of integers $\{0,\pm1,\pm2,\dotsc,\}$, and $\bbR$ is the set of real numbers. For $a\in\bbR$, $\bbR_{\ge a}=[a,\infty)$,  $\bbR_{>a}=(a,\infty)$, $\bbZ_{\ge a}=[a,\infty)\cap\bbZ$, and $\bbZ_{>a}=(a,\infty)\cap\bbZ$.
$\bbR_{\le a}$, $\bbR_{<a}$, $\bbZ_{\le a}$, and $\bbZ_{<a}$ are defined analogously.
For $a,b \in \bbR$ with $a\le b$ we write $\flint{a,b}$ to denote the integers that are in $[a,b]$ and we abbreviate $\flint{n}=\flint{1,n}$.
For points $u, v \in \bbR^2$, $u \leq v$ and $v\geq u$ mean $u_1 \leq v_1$ and $u_2 \leq v_2$. For such $u$ and $v$, let $[u,v]=\{x\in\RR^2:u\le x\le v\}$ 
%\reversemarginpar\note{C: This is not well defined if $u$ and $v$ are co-linear. For example, if $u = \zero$ and $v = \ek_1$, there are two such rectangles}\note{F: Better now?} 
and $\flint{u,v}=\{x\in\bbZ^2:u\le x\le v\}$. 
%is the set of vertices in $\bbZ^2$ that are in $[u,v]$.

We denote the canonical basis vectors of $\bbR^2$ by $\ek_1=(1,0)$ and $\ek_2=(0,1)$. Set $\zero=(0,0)$. An up-right path $\geod{m:n}= (\geod{m},\geod{m+1},\dots,\geod{n})$ is a collection of vertices $\geod{i} \in \bbZ^2$ which satisfies $\geod{i} - \geod{i-1} \in \{\ek_1,\ek_2\}$ for $i \in \flint{m+1,n}$.
%\note{Do we actually use this parametrization? C: No, see lower on this page :) Removed it.} 
For $x\le y$, the set of up-right paths which start at $x$ and end at $y$ is denoted by $\Pi_x^y$.

Let $\sU = [\ek_2,\ek_1] = \{t \ek_1 + (1-t) \ek_2 : 0 \leq t \leq 1\}$. Its relative interior is denoted by $\ri\sU=(\ek_2,\ek_1)=\{t \ek_1 + (1-t) \ek_2 : 0<t<1\}$. We will use the notation
$a\vee b=\max(a,b)$ and $a\wedge b=\min(a,b)$.
	
For $r \in (0,1),$ a $\Geom(r)$ random variable $X$ satisfies $\P(X = n) = r^n (1-r)$ for $n\in\bbZ_{\ge0}$.  For $p\in[0,1]$, a $\Ber(p)$ random variable $X$ satisfies $\P(X=1)=1-\P(X=0)=p$.
%Geo(r) here is the usual Geo(1-r)-1  so the mean is 1/(1-r)-1=r/(1-r)
%and the variance is r/(1-r)^2
	
\section{Models with boundary}\label{sec:bdry}
The main player in the proof of Theorem \ref{thm:noBiInfinites} is a coupling of the bulk passage times and a collection of passage times in models with boundary conditions. Given weights $\w\in\Omega$ and numbers $\{I_x,J_x:x\in\bbZ^2\}$, referred to as \textit{boundary weights},  
the \textit{boundary passage time} $G^{\rm{SW}}_{x,y}(\w,I,J)$ from $x$ to $y$, with $x\le y$, 
is the maximum weight of up-right paths from $x$ to $y$, where each path collects $0$ weight at the site $x$, $I$ weights at each vertex on the horizontal boundary $x+\bbZ_{\ge0} \ek_1$, $J$ weights at each vertex on the vertical boundary $x+\bbZ_{\ge0} \ek_2$, and bulk weights $\w$ at each vertex in the bulk $x + \NN^2$. 
See Figure \ref{fig:statmodel}. 
Rigorously, for $x=(x_1,x_2)\in\ZZ^2$ and $k\in\NN$
we set $G^{\rm{SW}}_{x,x}=0$,
\begin{align}\label{GSW0}
G^{\rm{SW}}_{x,x+k\ek_1}=\sum_{i=1}^k I_{x+i\ek_1},
\quad\text{and}\quad
G^{\rm{SW}}_{x,x+k\ek_2}=\sum_{i=1}^k J_{x+i\ek_2}.
\end{align}
Then for $y\in x+\N^2$ we let
\begin{align}\label{GSW}
	G^{\rm{SW}}_{x, y} =\max _{1 \leq k \leq y_1-x_1}\Bigl\{ \sum_{i=1}^{k} I_{x+i \ek_1} + G_{x+k \ek_1 + \ek_2, y} \Bigr\} \bigvee \max_{1 \leq \ell \leq y_2-x_2} \Bigl\{ \sum_{j=1}^{\ell} J_{x+j \ek_2}+G_{x+\ek_1+\ell \ek_2, y}\Bigr\}.
\end{align}
%with the convention that $\max_{\varnothing}=0$. In particular, $G^{\rm{SW}}_{x, x} = 0$.
Note that $G^{\rm{SW}}_{x,y}$ is a function of $\{I_{x+i\ek_1},I_{x+j\ek_2},\w_z:i,j\in\bbN,z\in x+\bbN^2\}$.
Hence the superscript SW which stands for southwest as this is where the $\w$-weights are switched to $I$ and $J$. 

\begin{figure}[ht]
\begin{center}
\begin{tikzpicture}[scale = 1, baseline=(current bounding box.north)]
\foreach \x in {-1,0,1,2,3,4}
	\foreach \y in {-1,0,1,2,3,4}
		\filldraw (\x,\y) circle (1.5pt);
\draw[opacity = .75] (-1,-1) -- (2,-1) -- (2,1) -- (2,3) -- (4,3) -- (4,4);
\draw (2,-1) circle [radius=.25];
\node[] at (-1.5,-1.5) {$x$};
\node[] at (0,-1.5) {$I_{x+\ek_1}$};
\node[] at (2,-1.5) {$\dots$};
\node[] at (4,-1.5) {$I_{x+5\ek_1}$};
\node[] at (-2,0) {$J_{x+\ek_2}$};
\node[] at (-2,2) {$\vdots$};
\node[] at (-2,4) {$J_{x+5\ek_2}$};
\draw[thin, dashed] (-1.25,4.25) -- (-1.25,-1.25) -- (4.25,-1.25) -- (4.25, -.75) -- (-.75,-.75) -- (-.75,4.25) -- (-1.25,4.25);
\end{tikzpicture}
\end{center}
\caption{An illustration of paths in the model with boundary conditions. The boundary is contained between the dashed lines, the geodesic is solid, and the exit point of the geodesic from the boundary is circled.}\label{fig:statmodel}
\end{figure}
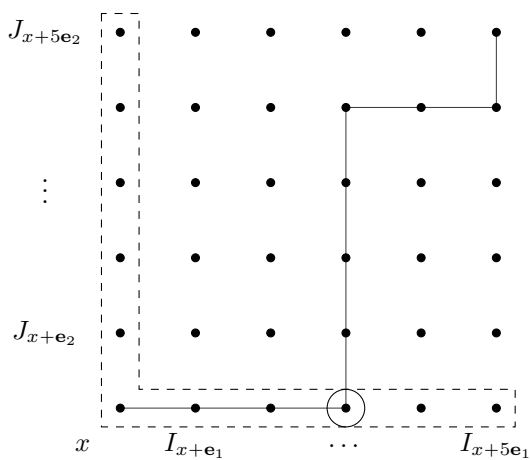

As in the bulk model, a geodesic in the model with boundary conditions is an up-right path that achieves the maximum in (\ref{GSW0}-\ref{GSW}). Recall that geodesics are not necessarily unique if the weights do not have a continuous distribution.

 Each geodesic path must exit the boundary at some point. See the circled vertex in Figure \ref{fig:statmodel} for an illustration of such an exit point. 
We denote by $\Exit_{x,y}^{\rm{SW}}(\w,I,J)$ the set of locations of exit points of the geodesics from $x$ to $y$, relative to the starting point $x$. That is, with the convention that we index exit points from the vertical boundary with negative numbers,
	\begin{align*}
	\Exit_{x,y}^{\rm{SW}} 
	&= \Bigl\{k\in\flint{1,y_1-x_1}:\sum_{i=1}^{k} I_{x+i \ek_1} + G_{x+k \ek_1 + \ek_2, y} = G_{x,y}^{\rm{SW}} \Bigr\}\\
	&\qquad\qquad\qquad\qquad\bigcup 
	\Bigl\{ -\ell:\ell\in\flint{1,y_2-x_2} \text{ and } \sum_{j=1}^{\ell} J_{x+j\ek_2} + G_{x+\ell \ek_2+\ek_1,y} = G_{x,y}^{\rm{SW}} \Bigr\}.
	\end{align*}
The furthest exit point in the $\ek_1$ direction is then given by 
    \[Z_{x,y}^{\rm{SW}, \ek_1} = \max \Exit_{x,y}^{\rm{SW}}\] 
and the furthest exit point in the $\ek_2$ direction is given by 
    \[Z_{x,y}^{{\rm{SW}}, \ek_2} = \min \Exit_{x,y}^{\rm{SW}}.\] 
Note that if $I_{x+k\ek_1}\le \bar I_{x+k\ek_2}$ for integers $1\le k\le(y-x)\cdot \ek_1$, then $Z_{x,y}^{{\rm{SW}},\ek_1}(\w,I,J)\le Z_{x,y}^{{\rm{SW}},\ek_1}(\w,\bar I,J)$, with a similar statement if the $J$-weights are increased.
%say z is some point between x and Z(I,J), A is the sum of the I's from x to z, B is the sum of the I's from z+\ek_1 to Z(I,J), C is the passage time from z+\ek_2 to y, and D the passage time from Z(I,J)+\ek_2 to y.  Then A+C \le A+B+D and so C\le B+D and increasing the weights makes B larger but C and D are not changed. So with the new weights we have \bar A+C\le \bar A+\bar B+D and so z\le Z(\bar I,J). This proves that Z(I,J)\le Z(\bar I,J).
 
The boundary weights that we will use in this paper are the random variables
$\{I_x^\xi,J_y^\xi:x,y\in\bbZ^2\}$, supplied by Theorem \ref{thm:coupledGenericLPP} for each fixed $\xi\in\ri\sU$. 
%\reversemarginpar
%\note{Not talking about the stochastic processes here. Just a family of random variables, since the theorem in this section does not need the process.}
We then use the notation
    \begin{align}\label{Gxi}
    G_{x,y}^{\xi}(\w)=G_{x,y}^{\rm{SW}}(\w,I^\xi(\w),J^\xi(\w)).
    \end{align}
$\Exit_{x,y}^\xi$ and $Z_{x,y}^{\xi,\ek_k}$ are defined similarly.

When the starting point is the origin $\zero$ we will omit it from the index and abbreviate quantities of the form $A_{\zero,x}$ by writing $A_x$ or $A(x)$.
We will also sometimes write $A(m,n)=A_{m\ek_1+n\ek_2}$.
%\smallskip

%The above particular choice of boundary weights satisfies: \mathnote{should document this in Theorem \ref{thm:coupledGenericLPP}}
%   \begin{align}\label{GIJ}
%   G_{x,y+\ek_1}^\xi-G_{x,y}^\xi=I_{y+\ek_1}^\xi=G_{y,y+\ek_1}^\xi\quad\text{and}\quad G_{x,y+\ek_2}^\xi-G_{x,y}^\xi=J_{y+\ek_2}^\xi=G_{y,y+\ek_2}^\xi.
%   \end{align}
%As a consequence, the gradients on the left-hand sides do not depend on $x$, 
%as long as $x\le y$. 
%
%Another way to express \eqref{GIJ} is as follows. 
The significance of the particular choice of boundary weights is that while the bulk passage times are subadditive:
    \begin{align}\label{Gsubadd}
    G_{x,y}-\w_y+G_{y,z}-\w_z\le G_{x,z}-\w_z,\quad\forall x\le y\le z,
    \end{align}
the boundary passage times are additive:     \begin{align}\label{Gadd}
    G^\xi_{x,y}+G^\xi_{y,z}=G^\xi_{x,z}.
    \end{align}
In particular, for any $x=(m,n)\in\bbZ_{\ge0}^2$
    \[\E[G_{\zero,x}^\xi]=\E[G_{\zero,\ek_1}^\xi]m+\E[G_{\zero,\ek_2}^\xi]n=
    \E[I_{\ek_1}^\xi]m+\E[J_{\ek_2}^\xi]n.\]
The above leads to a variational characterization of the limiting shape of the bulk model. Indeed, \cite[Theorem 2.3]{Mar-04} and a standard coarse-graining argument (see, for example, \cite{Jan-Nur-Ras-20-}) imply that if $\E[\abs{\w_0}^{2+\e}]<\infty$ for some $\e>0$, 
then for $\P$-almost every $\w$, 
    \begin{align}\label{shapeThm}
    \varlimsup_{n\to\infty}n^{-1}\max_{\substack{x\in\bbZ_{\ge0}^2\\\abs{x}_1=n}}
    \abs{G_{\zero,x}-\gpp(x)}=0,
    \end{align}
where 
    \begin{align}\label{shape}
    \gpp(x)=\gpp(x_1,x_2)=\inf_{\xi\in\ri\sU}\bigl\{\E[I_{\ek_1}^\xi]x_1+\E[J_{\ek_2}^{\xi}]x_2\bigr\}\quad\text{for }x\in\bbR_{\ge0}^2.
    \end{align}
This expression for $\gpp$ is an immediate consequence of the construction in \cite[Theorem 4.7]{Jan-Ras-20-aop}, which defines our $I^\xi$ and $J^\xi$ (see (4.3) and Lemma 4.12 there), and the variational characterization of a homogeneous concave function in terms of its superdifferential. 

Another property of the $(I^\xi,J^\xi)$ process is that it is stationary and, as a consequence, 
\begin{align}\label{Gxi-stat}
\{G_{z+x,z+y}^{\xi}:y\ge x\text{ in }\bbZ^2\}\quad\text{has the same distribution as}
\quad\{G_{x,y}^\xi:y\ge x\text{ in }\bbZ^2\}\quad\forall z\in\bbZ^2.
\end{align}
This explains why the last-passage percolation model  with these boundary weights is called the \textit{stationary model} or the \textit{stationary LPP}.%\smallskip

The significance of specializing to the case where the weights are geometric random variables, i.e.\  $\w_0\sim\Geom(r)$ for some $r\in(0,1)$, is that because of the memoryless property of the geometric distribution, many explicit computations are possible. 
For this reason, this case is said to be \textit{solvable}. 
For example, for any $x\le y$ in $\bbZ^2$ and any $\xi\in\ri\sU$, 
\begin{align}\label{IJw-indep}
\bigl\{\w_z,G_{x,y+(n+1)\ek_i}^\xi - G_{x,y+n\ek_i}^\xi :n\in\bbZ_{\ge0},i\in\{1,2\},z-x\in\NN^2\bigr\}\quad
\text{are independent}
\end{align}
and, marginally, $G_{x,y+\ek_1}^\xi - G_{x,y}^\xi \sim \Geom(p)$ and $G_{x,y+\ek_2}^\xi- G_{x,y}^\xi \sim \Geom(r/p)$, with $p=\pxi(\xi)$ given by
        \begin{align}\label{xi->p}
        \pxi(\xi)=\pxi(\xi_1,\xi_2) = \frac{r(\xi_1+\xi_2) + (r+1)\sqrt{r \xi_1\xi_2}}{\xi_1 + r\xi_2 + 2\sqrt{r\xi_1\xi_2}}\,\in(r,1)\quad\text{for $\xi\in\bbR_{>0}^2$.}
        \end{align}
These are some of the properties contained in Theorem \ref{thm:coupledGeomLPP}.

For a fixed $r\in(0,1)$, \eqref{xi->p} gives a bijection from $\ri\sU$ to $(r,1)$ with the inverse function given  by 
    \begin{align}\label{p->xi}
    \xip(p)= \left(\frac{r(1-p)^2}{p^2(r+1)-4pr+r(r+1)}\,,\,\frac{(p-r)^2}{p^2(r+1)-4pr+r(r+1)}\right)\in\ri\sU.
    \end{align}
Switching from $\xi$ to $p$ in the variational formula \eqref{shape} and then using the explicit distributions allows to solve \eqref{shape} explictly and get
\begin{align}
\gpp(x) = \gpp(x_1,x_2)=\inf_{p \in (r, 1)} M^p(x) = \frac{r}{1-r}(x_1+x_2) + \frac{2\sqrt{r}}{1-r} \sqrt{x_1 x_2}\,, \label{eq:shape}
\end{align}
where
\begin{align}
M^p(x)=M^p(x_1,x_2) %= \E[G^p_{\zero,x}] 
=  \frac{px_1}{1-p} + \frac{\frac{r}{p} x_2}{1-\frac{r}{p}} = \frac{px_1}{1-p} + \frac{rx_2}{p-r}\,. \label{eq:statshape}
\end{align}

%the boundary weights 
%$\{I_{x+i \ek_1}^\xi : 1 \leq i < (y-x)\cdot \ek_1\}$ and $\{J_{x+j \ek_2}^\xi :  1 \leq j < (y-x) \cdot \ek_2\}$ %are chosen to each be marginally i.i.d., mutually independent and independent of $\{\w_u : x+\ek_1+\ek_2 \leq u \leq y\}$, and if $I_{x+i \ek_1}^z\sim \Geom(z)$ and $J_{x+j}^z\sim \Geom(r/z)$, then if $(y_k : 1 \leq k \leq n)$ is any sequence of vertices with $x \leq y_k$ and $y_k \in \{\ek_1,-\ek_2\}$ for all $k \in \{0,\dots,n\}$, the collection $\{G_{x,x_k}^z - G_{x,x_{k-1}}^z : 1 \leq k \leq n\}$ are independent and, marginally, for each $y \geq x$, $G_{x,y+\ek_1}^z - G_{x,y}^z \sim \Geom(z)$ and $G_{x,y+\ek_2}^z- G_{x,y}^z \sim \Geom(r/z)$. The proofs of these properties, as well a description of our coupling of these processes as functions of are contained in Appendix \ref{app:Buse}.

\section{Geodesic Fluctuation Bounds}\label{sec:geo_fluct}
Theorem \ref{thm:coupledGenericLPP} produces random variables
$\{I_x^\xi,J_y^\xi:x,y\in\bbZ^2, \xi\in\ri\sU\}$ and the passage times that use these variables as boundary weights, which we denote by $G_{x,y}^{\xi}(\w)=G_{x,y}^{\rm{SW}}(\w,I^\xi(\w),J^\xi(\w))$.

In this section, we give bounds on the size of the fluctuations of the point-to-point geodesics under the following assumption on the tails of 
$Z^{\xi,\ek_k}_{0,x}$, when $\abs{x}_1$ is large and $x/\abs{x}_1$ is close to $\xi$.
For $\delta\in(0,1)$, define the cone 
    \[S_\delta = \{x \in \RR_{>0}^2 : x\cdot \ek_1 \geq \delta x \cdot \ek_2 \text{ and } x \cdot \ek_2 \geq \delta x\cdot \ek_1  \}.\] 
\begin{assumption}\label{exitPointTailBound}
%\normalmarginpar

	There exist a $\fluctexp>2$ and a $\delta_0\in(0,1)$ such that for any $\delta \in (0,\delta_0)$ and $\kappa \geq 0$, there exist positive finite constants $C_0(\delta)$, $N_0(\delta,\kappa)$, and $s_0(\delta,\kappa)$ such that 
    \begin{align}\label{poly-tail}
    \P\bigl\{\abs{Z^{\xi, \ek_1}(m,n)}\vee\abs{Z^{\xi, \ek_2}(m,n)} \geq s (m + n)^{2/3}\bigr\} 
    \leq C_0 s^{-\nu},
    \end{align}
		for all $(m, n) \in S_\delta \cap \ZZ_{\geq N_0}^2$, $s \geq s_0$, and $\xi\in\ri\sU$ such that $\xi_1 \in (\delta, 1-\delta)$ and $\abs{\xi_1 - \frac{m}{m+n}} \leq \kappa(m+n)^{-1/3}$.
\end{assumption}\smallskip

	By Theorem \ref{geometricExitPointTailBound}, this assumption is satisfied for any $\nu>2$ when $\w_0$ is geometrically distributed. This assumption is verified in the case of exponential weights in \cite[Corollary 4.3]{Bal-Bus-Sep-20}, with a sharp bound appearing in \cite[Corollary 3.2]{Emr-Jan-Sep-21-}. 
	%\mathnote{Maybe add some explanation as to why one would expect this to hold for general weights} 
	%\mathnote{maybe we can point to some spots in Timo's paper that show that the assumption is also satisfied in the exponential case??}
	%Again here $Z^{z, \ek_1}(m,n)$ is the signed exit point of the rightmost geodesic, and $Z^{z, \ek_2}(m,n)$ is the signed exit point of the upmost geodesic.  
	
	%The following tail bound is sufficient for the proof of Theorem \ref{thm:noBiInfinitesAwayFromAxes}.  T

We begin with some preliminary observations about the structure of last-passage percolation. Given points $x\le y$ in $\bbZ^2$ and weights $\w$, define the boundary weights 
	\begin{align}\label{G-IJ}
	\begin{split}
	&I_y^{[x]}(\w)=G_{x, y}(\w)-G_{x,y-\ek_1}(\w),\quad\text{when $x\le y-\ek_1$, and}\\ 
	&J_y^{[x]}(\w)=G_{x,y}(\w)-G_{x, y-\ek_2}(\w),\quad\text{when $x\le y-\ek_2$}.
	\end{split}
	\end{align}
Then for $z\in y+\bbZ^2_{\ge0}$, let
    \[G_{y,z}^{[x]}(\w)=G^{\rm{SW}}_{y,z}(\w,I^{[x]}(\w),J^{[x]}(\w))
%Let $Z_{x, y}^{\ek_2}$ be the signed exit points of the upmost geodesic of $G_{u, p}$.
\quad\text{and}\quad 
    Z^{[x],\ek_2}_{y,z}(\w)=Z^{\rm{SW},\ek_2}_{y,z}(\w,I^{[x]}(\w),J^{[x]}(\w)).\]
%to be the signed exit point of the topmost geodesic of $G_{y, z}^{[x]}$.  

The following is immediate from the definitions.  See, for example, \cite[Lemma A.1]{Sep-18}. 
%\note{C: Added a reference. Change if you intended a different one.}

\begin{lemma}\label{lem:pushForwardLPP} Let $x \leq y \leq z$ in $\bbZ^2$. Fix a configuration of weights $\w\in\Omega$. Then $G_{x, z}(\w) = G_{x, y}(\w) + G_{y, z}^{[x]}(\w)$.  Furthermore, if an upright path is a geodesic of $G_{x, z}(\w)$, then its restriction to $y + \ZZ_{\geq 0}^2$ is part of a geodesic of $G_{y, z}^{[x]}(\w)$. 
Likewise, if an upright path is part of a geodesic of $G_{y,z}^{[x]}(\w)$, then it can be extended to a geodesic of $G_{x, z}(\w)$. 
%\note{C:The part that extends is the part of the path which is not on the boundary. Essentially, the exit edge from the $[x]$ boundary lies on a geodesic in the original environment.} \firasnote{Similar question: does the extension have to pass through $z$?}
\end{lemma}

The next lemma is a direct consequence of the one above.
%This is an equivalent statement of the lemma below:
%Let $m$ be a positive integer. 
%Let $x\ge y$ be in $\ZZ^2$.  Take $i\in\{1,2\}$.
%Fix a configuration of weights $\w\in\Omega$.
%Then, for $\ell\in\flint1,m}$, $Z_{x, y}^{{\rm{SW}},\ek_i}(\w) = m$ if and only  if
%$Z_{x+\ell \ek_1, y}^{{\rm{SW}},\ek_i}(\w) = m-\ell$. Similarly, 
%$Z_{x, y}^{{\rm{SW}},\ek_i}(\w) = -m$ if and only if 
%$Z_{x+\ell \ek_2, y}^{{\rm{SW}},\ek_i}(\w) = -m+\ell$. 

\begin{lemma}\label{lem:exitPointShift}
	Let $\ell,m$ be positive integers. 
	Let $x\ge y$ be in $\ZZ^2$. 
	Take $i \in \{1, 2\}$.
	Fix a configuration of weights $\w\in\Omega$.
	Then $Z_{x, y}^{{\rm{SW}},\ek_i}(\w) = \ell + m$ if and only if $Z_{x+ \ell \ek_1, y}^{[x],\ek_i}(\w) = m$. Similarly, $Z_{x, y}^{{\rm{SW}},\ek_i}(\w) = -\ell - m$ if and only if $Z_{x+ \ell \ek_2, y}^{[x],\ek_i}(\w) = -m$.
\end{lemma}

The above definitions and lemmas are deterministic statements and work for every fixed choice of the environment $\w$. Therefore, by considering passage times with boundary weights $I^\xi$ and $J^\xi$ and recalling \eqref{Gadd}, we see that both lemmas hold if we replace $G_{x,z}$, $G_{x,y}$, $G_{y,z}^{[x]}$, $Z_{x,y}^{{\rm{SW}},\ek_i}$, and $Z^{[x],\ek_i}_{x+\ell \ek_i,y}$, $i\in\{1,2\}$, with, respectively, $G^{\xi}_{x,z}$,  $G^{\xi}_{x,y}$, $G^{\xi}_{y,z}$, $Z_{x,y}^{\xi,\ek_i}$, and $Z^{\xi,\ek_i}_{x+\ell \ek_i,y}$, $i\in\{1,2\}$. 

%	\begin{align*}
%	I_{y+i \ek_{1}}^{[x],\xi}(\w)=G^\xi_{x, y+i \ek_{1}}(\w)-G^\xi_{x,y+(i-1) \ek_{1}}(\w) \quad \text{and}\quad J_{y+j \ek_{2}}^{[x],\xi}(\w)=G^\xi_{x,y+j \ek_{2}}(\w)-G^\xi_{x, y+(j-1) \ek_{2}}(\w),\quad i,j\in\bbN.
%	\end{align*}
%Then \eqref{Gadd} implies that 
%\[I_{y+i \ek_{1}}^{[x],\xi}=I_{y+i \ek_{1}}^{\xi}\quad\text{and}\quad J_{y+j \ek_{2}}^{[x],\xi}=J_{y+j \ek_{2}}^{\xi}.\]
%Similarly, 
%    \[G_{y,z}^{[x],\xi}(\w)=G^{\rm{SW}}_{y,z}(\w,I^{[x],\xi}(\w),J^{[x],\xi}(\w))=G_{y,z}^{\xi}(\w)
%Let $Z_{x, y}^{\ek_2}$ be the signed exit points of the upmost geodesic of $G_{u, p}$.
%\quad\text{and}\quad 
%    Z^{\xi,[x],\ek_2}_{y,z}(\w)=Z^{\rm{SW},\ek_2}_{y,z}(\w,I^{[x],\xi}(\w),J^{[x],\xi}(\w))=Z^{\xi,\ek_2}_{y,z}(\w).\]
%to be the signed exit point of the topmost geodesic of $G_{y, z}^{[x]}$.  

	\begin{cor}\label{cor:exitPointsOffCharDirection} Suppose Assumption \ref{exitPointTailBound} holds. Then for any $\delta \in (0,\delta_0)$, $A>0$, and $\kappa \geq 0$ there exist positive finite constants $C_1(\delta,\delta_0,\fluctexp,A)$, $N_1(\delta,\delta_0,\kappa)\ge1$, and $s_1(\delta,\delta_0,\kappa)$ such that 
	\begin{equation}
	\P\{Z^{\xi, \ek_2}(m,n- \fl{ s(m+n)^{2/3} }) < 0\} \leq C_1 s^{-\fluctexp} \label{Cor4-3-L}
	\end{equation}
	and 
	\begin{equation}
	\P\{Z^{\xi, \ek_1}(m,n+ \fl{ s(m+n)^{2/3} }) > 0\} \leq C_1 s^{-\fluctexp} \label{Cor4-3-U}
	\end{equation}
	for all $(m,n) \in S_{\delta} \cap \ZZ^2_{\geq N_1}$, $s \geq s_1$,  $\xi\in\ri\sU$ such that $\xi_1 \in (\delta, 1-\delta)$ and $\abs{\xi_1 - \frac{m}{m+n}} \leq \kappa(m+n)^{-1/3}$, and with $n - \fl{ s(m+n)^{2/3} } \geq 1$, in the case of \eqref{Cor4-3-L}, and $s\le A(m+n)^{1/3}$, in the case of \eqref{Cor4-3-U}.
	\end{cor}
	
	\begin{proof} 
	Fix $\delta$ and $\kappa$ as in the claim.
	Recall the constants $N_0$ and $s_0$ in Assumption \ref{exitPointTailBound}.
	Take $(m,n) \in S_{\delta} \cap \ZZ^2_{\geq N_0}$ and $s \geq\max(2s_0,2^{1/3}N_0^{-2/3})$ such that $n - \fl{ s(m+n)^{2/3} } \geq 1$. Take $\xi\in\ri\sU$ such that $\xi_1 \in (\delta, 1-\delta)$ and $\abs{\xi_1 - \frac{m}{m+n}} \leq \kappa(m+n)^{-1/3}$. 
	Apply shift-invariance, Lemma \ref{lem:exitPointShift}, and  \eqref{poly-tail} to obtain
	\begin{align*}
	\P\bigl\{Z^{\xi, \ek_2}_{(0,0), (m, n - \fl{ s(m+n)^{2/3} } )} < 0\bigr\}
	&= \P\bigl\{Z^{\xi, \ek_2}_{(0, \fl{ s(m+n)^{2/3} }), (m, n )} < 0\bigr\} 
	%= \P\bigl\{Z^{[(0,0)],\xi, \ek_2}_{(0, \fl{ s(m+n)^{2/3} }), (m, n )} < 0\bigr\} 
	\\
	&= \P\bigl\{Z^{\xi, \ek_2}_{(0,0), (m,n)} < -\fl{ s(m+n)^{2/3} } \bigr\} 
	\leq \P\bigl\{Z^{\xi, \ek_2}_{(0,0), (m,n)} < -s(m+n)^{2/3}/2 \bigr\}   \\
	&\leq 2^{\fluctexp}C_0s^{-\fluctexp}.
	\end{align*}
	%In the last inequality we use $s(m+n)^{2/3}\ge s(2N_0)^{2/3}\ge2$ and the round down of a number $a\ge2$ is bounded below by $a-1\ge a/2$.
	
	For \eqref{Cor4-3-U}, 
	let $\bar N_0=N_0(\delta/2,\kappa+1)$ and 
	\[\bar s_0=\max\bigl(s_0(\delta/2,\kappa+1),2(1+\delta)\delta^{-1}\bar N_0^{-2/3},(1+\delta)^{-1}\delta\bar N_0^{1/3}\bigr).\] 
	Take $(m,n) \in S_{\delta} \cap \ZZ^2_{\ge\bar N_0}$ and $s\ge\bar s_0$. Let $d = \fl{ s(m+n)^{2/3} }$, $\tilde{n} = n + d$ and $\tilde{m} = m + \left\lfloor  \frac{dm}{n} \right\rfloor$. 
	Then
	\[ \frac{\delta}2\le \frac{\delta n+\delta s(m+n)^{2/3}-1-\delta}{n+s(m+n)^{2/3}}\le\frac{\tilde{m}}{\tilde{n}}\le\frac1\delta\le\frac2\delta\,.\]
	%On the left: \ge \delta-\frac{1+\delta}{s(2\bar N_0)^{2/3}\ge\delta/2
	%On the right: m\le n/\delta and \lfloor dm/n\rfloor\le dm/n\le d/\delta
	Take $\xi\in\ri\sU$ such that $\xi_1\in(\delta, 1-\delta)$ and $\abs{\xi_1 - \frac{m}{m+n}} \leq \kappa(m+n)^{-1/3}$. Then	
    \[\bigl|\xi_1-\frac{\tilde m}{\tilde m+\tilde n}\bigr|\le
	\bigl|\xi_1-\frac{m}{m+n}\bigr|+\frac{md-n\lfloor dm/n\rfloor}{(m+n)^2}\le\kappa(m+n)^{-1/3}+(m+n)^{-1}\le(\kappa+1)(m+n)^{-1/3}.\]
	Furthermore, if we take $s\le A(m+n)^{1/3}$, then we get
	\[\frac{\tilde m+\tilde n}{m+n}\le1+\frac{(\delta^{-1}+1)d}{m+n}\le1+(\delta^{-1}+1)s(m+n)^{-1/3}\le1+(\delta^{-1}+1)A.\]
%	\[\fl{dm/n}+d\le(\delta^{-1}+1)d\le(1+\delta)\delta^{-1} s(m+n)^{2/3}\le (1+\delta)\delta^{-1}\bar N_0^{-1/3}s(m+n),\]
%	which implies
%	\[\tilde m+\tilde n-m-n=\fl{dm/n}+d\le\frac{(1+\delta)\delta^{-1}\bar N_0^{-1/3}s}{1+(1+\delta)\delta^{-1}\bar N_0^{-1/3}s}(m+\fl{dm/n}+n+d),\]
%    leading to
%        \[m+n\ge\frac{\tilde m+\tilde n}{1+(1+\delta)\delta^{-1}\bar N_0^{-1/3}s}\ge\frac12(1+\delta)^{-1}\delta\bar N_0^{1/3}s^{-1}(\tilde m+\tilde n).\]
	Now, similar to the above computation, we have
	\begin{align*}
	\P\bigl\{Z^{\xi, \ek_1}(m,n+\lfloor s(m+n)^{2/3}\rfloor) > 0 \bigr\} 
	%&=\P\bigl\{Z^{\xi, \ek_1}_{(\lfloor  \frac{dm}{n} \rfloor,0),(\tilde{m},\tilde{n})} > 0 \bigr\} \\
	&=\P\Bigl\{Z^{\xi, \ek_1}( \tilde{m} , \tilde{n}) >  \bigl\lfloor\frac{d m}{n}\bigr\rfloor \Bigr\} \\
	&\leq \P\bigl\{Z^{\xi, \ek_1}( \tilde{m} , \tilde{n}) >  \tfrac12\delta s(m+n)^{2/3} \bigr\} 		%\delta\bar s_0(2\bar N_0)^{2/3}\ge2^{5/3}(1+\delta)\ge2
\\
	&\leq \P\bigl\{Z^{\xi, \ek_1}( \tilde{m} , \tilde{n}) >  \tfrac12\delta\bigl(1+(\delta^{-1}+1)A\bigr)^{-2/3} s(\tilde m+\tilde n)^{2/3} \bigr\} \\
	&\leq (2/\delta)^{\fluctexp}\bigl(1+(\delta^{-1}+1)A\bigr)^{2\fluctexp/3}C_0(\delta/2) s^{-\fluctexp}.\qedhere
	\end{align*}
\end{proof}

	For $(m, n) \in \NN^2$, $\alpha \in (0,1)$, and $s > 0$, define 
	\[
	\mc{C}_{\alpha, s}^{(m,n)} = \flint{(\fl{\alpha m},\fl{\alpha n}) - s(m+n)^{2/3}\ek_2, (\fl{\alpha m }, \fl{\alpha n}) + s(m+n)^{2/3}\ek_2}.
	\]
	$\mc{C}_{\alpha, s}^{(m,n)}$ is the symmetric vertical line segment centered at $ (\fl{ \alpha m }, \fl{ \alpha n })$ with length $2s (m+n)^{2/3}$.
	For $x\in\bbR_+^2\setminus\{\zero\}$, let
	\[\ximn(x) = \frac{x}{\abs{x}_1}\in \sU.\]  
	Let $\pi^{(m,n),\ek_1}$ and $\pi^{(m,n),\ek_2}$  denote, respectively, the rightmost  and the upmost geodesics of $G(m,n)$.

	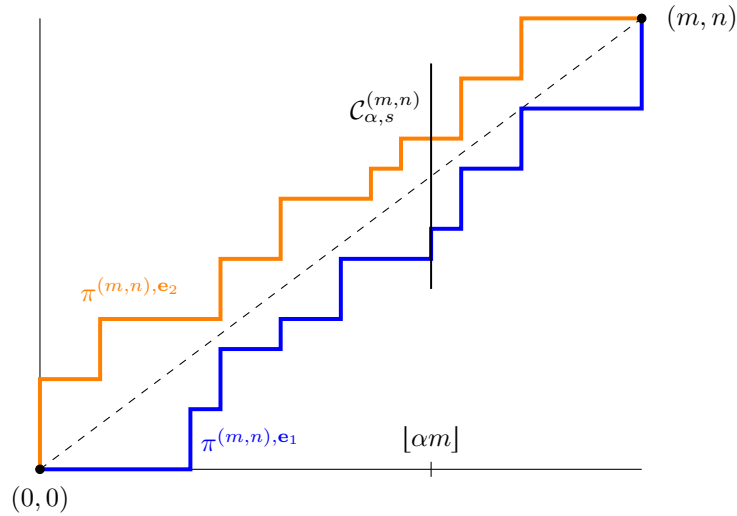
\begin{figure}[ht]
		\begin{center}
			\begin{tikzpicture}[scale = 0.4, baseline=(current bounding box.north)]
				\draw (0,15) -- (0,0) -- (20,0);
				\draw[dashed] (0,0) -- (20,15);
				
				\draw[line width=1.5pt, orange] (0,0) -- (0,3) -- (2,3) -- (2,5) -- (6,5) -- (6,7) -- (8,7) -- (8,9) -- (11,9) -- (11,10) -- (12,10) -- (12,11) -- (14,11) -- (14,13) -- (16,13) -- (16,15) -- (20,15);
				\node[orange] at (3,6) { $\pi^{(m,n), \ek_2}$};
			
				\draw[line width=1.5pt, blue] (0,0) -- (5,0) -- (5,2) -- (6,2) -- (6,4) -- (8,4) -- (8,5) -- (10,5) -- (10,7) -- (13,7) -- (13,8) -- (14,8) -- (14,10) -- (16,10) -- (16,12) -- (20,12) -- (20,15);
				\node[blue] at (7,1) { $\pi^{(m,n), \ek_1}$};
				
				\draw (13,-0.25) -- (13,0.25);
				\node at (13,1) {$\fl{\alpha m}$};
				
				\draw[line width=0.75pt] (13, 6) -- (13, 13.5);
				\node at (11.5,12) { $\mc{C}_{\alpha, s}^{(m,n)}$};
				
				\filldraw (0,0) circle (4pt);
				\node[] at (0,-1) { $(0,0)$};
				
				\filldraw (20,15) circle (4pt);
				\node[] at (22,15) {$(m,n)$};
				
			\end{tikzpicture}
		\end{center}
		\caption{An illustration of the high probability event in Lemma \ref{lem:geodesicsIntersectVertical}.  The upmost and rightmost geodesics from $(0,0)$ to $(m,n)$ will intersect the vertical line segment $\mc{C}_{\alpha, s}^{(m,n)}$. }\label{fig:geodesicsIntersectVertical}
	\end{figure}

	\begin{lemma}\label{lem:geodesicsIntersectVertical} 
	Suppose Assumption \ref{exitPointTailBound} holds.
	For $0 < \delta < \delta_0$ and $0 < \e < \frac{1}{2}$, there exist finite positive constants $C_2(\delta, \delta_0, \fluctexp, \e)$, $N_2(\delta, \delta_0)\ge1$, and $s_2(\delta, \delta_0)$ such that the following holds:  for all $(m,n) \in S_\delta \cap \ZZ_{\geq N_2}^2$, $\alpha \in (\e, 1-\e)$, and $s \in \bigl[s_2, \frac{\e \delta}{3(\delta+1)}(m+n)^{1/3} \bigr] $,
	\begin{align}\label{main-sec3}
	\P\Bigl\{(\pi^{(m,n), \ek_1} \cap \mc{C}_{\alpha, s}^{(m,n)} = \varnothing) \cup (\pi^{(m,n), \ek_2} \cap \mc{C}_{\alpha, s}^{(m,n)} = \varnothing)\Bigr\} \leq C_2 s^{-\fluctexp}. 
	\end{align}
	\end{lemma}
	
	\begin{proof}  Fix $\delta, \varepsilon$ as in the claim.  
	Take 
	    \begin{align}\label{N2}
	    \begin{split}
	    N_2
	    &=\max\Bigl\{\e^{-1}N_1(\delta/3,\delta_0,1),
	    (3-\delta)\delta^{-1}\e^{-1}\bigl(N_1(\delta/3,\delta_0,1)+1\bigr),3(3-\delta)\delta^{-2}\e^{-1},\\
	    &\qquad\qquad
	    2(1-\delta/3)\e^{-1},(2-2\delta/3)^{3/2}(1+3/\delta)^{1/2}\e^{-1},\frac12\bigl(2\delta^{-1}(3-\delta)\bigr)^{3/2},(1-\e)^{-1},\\
	    &\qquad\qquad
	    \frac65\e^{-1}N_0(\delta/3,1),%2\delta\e^{-1}/3,
	    \e^{-1},\frac12(3\e^{-1})^{3/2}\bigl(1-\frac{\e\delta}{3(\delta+1)}\bigr)^{-3}\Bigr\}.
	    \end{split}
	    \end{align}
	Take 
	    \[s_2=\max\bigl\{2,8\delta^{-2/3}s_1(\delta/3,\delta_0,1),2s_0(\delta/3,1)\bigr\}.\]
	Take $(m,n) \in S_{\delta} \cap \ZZ_{\geq N_2}^2$ and 
	\[s \in \bigl[s_2, \frac{\e \delta}{3(\delta+1)}(m+n)^{1/3} \bigr].\]  
    This ensures that $n-s(m+n)^{2/3}\ge (1-\e/3)n>0$.
    %\[\alpha n - 2s(m+n)^{2/3}\ge\alpha n-\frac{2\e\delta}{3(\delta+1)}\cdot(m+n)
    %\ge\e n\bigl(1-\frac{2\delta}{3(\delta+1)}\cdot(1\delta+1)\bigr)
    %\ge \e n/3>0.
    %\]
	Let $\xi^\star = \ximn(m, n-s(m+n)^{2/3})$ and $\xi_\star = \ximn(m, n+s(m+n)^{2/3})$.  
	Let $\pi^\star$ denote the upmost geodesic of $G^{\xi^\star}(m,n)$ and $\pi_\star$ be the rightmost geodesic of $G^{\xi_\star}(m,n)$.  
	
	Theorem \ref{thm:coupledGenericLPP} gives a coupling of the weights 
	$\{\w_x,I_x^{\xi^\star},J_x^{\xi^\star},I_x^{\xi_\star},J_x^{\xi_\star}:x\in\bbZ^2\}$ that is stationary and such that almost surely, for all $x\in\bbZ^2$,
	\begin{equation}
	\omega_x \leq I_x^{\xi^\star} \leq I_x^{\xi_\star}\quad\text{and}\quad\omega_x \leq J_x^{\xi_\star} \leq J_x^{\xi^\star}.\label{Lem4-4-monotonicity}
	\end{equation}
	
	Take $\alpha \in (\e, 1-\e)$ and define 
	$\tilde{o} = (\fl{ \alpha m }, \fl{ \alpha (n-s(m+n)^{2/3}) })$.
	%which is $\xi^\star$-directed.  Define the rectangle $R = \flint{ \tilde{o}, (m,n)}$. Let $G_{\tilde{o}, (m,n)}^{[\zero],\xi^\star}$ be the stationary LPP process on $R$ with boundary weights given for $i, j \geq 1$ by 
	%\begin{align*}
	%I_{\tilde{o}+i \ek_{1}}^{[\zero],\xi^{\star}}=G_{\zero, \tilde{o}+i \ek_{1}}^{\xi^{\star}}-G_{\zero, \tilde{o}+(i-1) \ek_{1}}^{\xi^{\star}} \quad \text { and } \quad J_{\tilde{o}+j \ek_{2}}^{[\zero],\xi^{\star}}=G_{\zero, \tilde{o}+j \ek_{2}}^{\xi^{\star}}-G_{\zero, \tilde{o}+(j-1) \ek_{2}}^{\xi^{\star}}.
	%\end{align*}
	By Lemma \ref{lem:pushForwardLPP}, the point where $\pi^\star$ crosses the southwest boundary of the rectangle $[\tilde{o}, (m,n)]$ is the same as the exit point of the upmost geodesic of $G_{\tilde{o}, (m,n)}^{\xi^\star}$ from the same boundary. 
	Furthermore, we clearly have
	    	    \[\fl{\alpha n}+2s(m+n)^{2/3}\ge\fl{\alpha(n-s(m+n)^{2/3})}+2s(m+n)^{2/3},\]
    and the fact that $m+n\ge1$ and $s\ge1$ implies
        $(2-\alpha)s(m+n)^{2/3}\ge1+\e\ge1$,
    which implies
	    \[\fl{\alpha n}-2s(m+n)^{2/3}\le\fl{\alpha(n-s(m+n)^{2/3})}.\]
	 %lhs \le \alpha n-2s(m+n)^{2/3}\le\alpha n-\alpha s(m+n)^{2/3}-1\le rhs
	Therefore, 
	\begin{align*}
	\bigl\{\pi^\star \cap\, \mc{C}_{\alpha, 2s}^{(m,n)} = \varnothing\bigr\} 
	\subset \bigl\{ \pi^\star \cap\, [\tilde{o}, \tilde{o} + 2s(m+n)^{2/3} \ek_2] = \varnothing \bigr\} 
	\subset \bigl\{Z_{\tilde{o}, (m,n)}^{\xi^\star, \ek_2} \notin [-2s(m+n)^{2/3}, -1 ]\bigr\}.
	\end{align*}
	Consequently,
	\begin{align}
	\P(\pi^\star \cap \mc{C}_{\alpha, 2s}^{(m,n)} = \varnothing) &\leq \P(Z_{\tilde{o}, (m,n)}^{\xi^\star, \ek_2} \notin \flint{ -2s(m+n)^{2/3}, -1 }) \notag\\
	&= \P(Z_{\tilde{o}, (m,n)}^{\xi^\star, \ek_2} > 0) + \P(Z_{\tilde{o}, (m,n)}^{\xi^\star, \ek_2} < -2s(m+n)^{2/3}). \label{probGeodInterLine}
	\end{align}
	
	To bound the first of these two probabilities, let 
	    \[\tilde m=m-\fl{\alpha m},\quad\tilde n=\fl{\xi_2^\star\tilde m/\xi_1^\star},\quad\text{and}\quad \tilde s=\frac{n-\fl{\alpha(n-s(n+m)^{2/3})}-\tilde n}{(\tilde m+\tilde n)^{2/3}}\,.\]
	We next check that we can apply Corollary \ref{cor:exitPointsOffCharDirection} with these parameters. 
	
	Note that
	    \[\xi_1^\star=\frac{m}{m+n-s(m+n)^{2/3}}\ge\frac{m}{m+n}\ge\frac{\delta}{\delta+1}%\ge\frac{\delta}2
	    \ge\frac\delta3\]
and	    \[\xi_1^\star%=\frac{m}{m+n-s(m+n)^{2/3}}
	    \le\frac{m}{m+n-\frac{\e\delta}{3(\delta+1)}\cdot(m+n)}
	    =\frac{m}{m+n}\cdot\frac{\delta+1}{\delta+1-\e\delta/3}
	    \le \frac1{1+5\delta/6}%=1-\frac{5\delta}{6+5\delta}
	    %\le 1-\frac{5\delta}{11}
	    \le1-\frac\delta3\,.\]	  
	Next, we use the choice of $N_2$ in \eqref{N2} repeatedly.
	First, we have
        \[\tilde m\ge (1-\alpha)m\ge\e N_2\ge N_1(\delta/3,\delta_0,1)\]
        and
        \[\tilde n\ge\xi_2^\star\tilde m/\xi_1^\star-1\ge\frac{\delta/3}{1-\delta/3}\cdot\e N_2-1\ge N_1(\delta/3,\delta_0,1).\]
    We also have
        \[\frac{\tilde m}{\tilde n}\ge\frac{\tilde m}{\xi_2^\star\tilde m/\xi_1^\star}
        \ge\frac{\delta/3}{1-\delta/3}\ge\frac\delta3\]
    and 
        \[\frac{\tilde m}{\tilde n}\le\frac{\tilde m}{\xi_2^\star\tilde m/\xi_1^\star-1}
        \le\frac1{\frac{\delta}{3-\delta}-\frac1{\e N_2}}\le\frac3\delta\,.\]
    In other words, $(\tilde m,\tilde n)\in S_{\delta/3}$.
    Furthermore, 
		\begin{align*}
		\Bigl|\xi_1^\star-\frac {\tilde m}{\tilde m+\tilde n}\Bigr|
		%&=(\tilde m+\tilde n)^{-1}\abs{\tilde n\xi_1^\star-\tilde m\xi_2^\star}
		&=\Bigl|\tilde n-\frac{\tilde m\xi_2^\star}{\xi_1^\star}\Bigr|\,(\tilde m+\tilde n)^{-1}\xi_1^\star
		\le(\tilde m+\tilde n)^{-1}.
		\end{align*}
%An alternative way to do this bound:
%        \begin{align*}
%        \Bigl|\xi_1^\star-\frac{\tilde m}{\tilde m+\tilde n}\Bigr|
%        %&=\Bigl|\xi_1^\star-\frac{\tilde m}{\tilde m+\fl{\xi_2^\star\tilde m/\xi_1^\star}}\Bigr|\\
%        &=\frac{\bigl|\tilde m(\xi_1^\star-1)+\xi_1^\star\fl{\xi_2^\star\tilde m/\xi_1^\star}\bigr|}{\tilde m+\fl{\xi_2^\star\tilde m/\xi_1^\star}}
%        \le\frac1{\tilde m/\xi_1^\star-1}
%        %\xi_2^\star\tilde m-\xi_2^\star\le\xi_1^\star\floor{\xi_2^\star\tilde m/\xi_1^\star}\le\xi_2^\star\tilde m
%        %so the absolute value in the numerator is bounded by \xi_2^\star\le1
%        \le\frac1{\tilde m/(1-\delta/3)-1}\\
%        &\le\frac{2(1-\delta/3)}{\tilde m}
%        %used 1\le \e N_2/(2(1-\delta/3))\le\tilde m/(2(1-\delta/3))
%        \le\frac{2(1-\delta/3)(\tilde m+\tilde n)^{1/3}}{\tilde m}\cdot(\tilde m+\tilde n)^{-1/3}\\
%        %&\le\frac{2(1-\delta/3)(1+3/\delta)^{1/3}}{\tilde m^{2/3}}\cdot(\tilde m+\tilde n)^{-1/3}\\
%        &\le\frac{2(1-\delta/3)(1+3/\delta)^{1/3}}{(\e N_2)^{2/3}}\cdot(\tilde m+\tilde n)^{-1/3}\le (\tilde m+\tilde n)^{-1/3}.
%        \end{align*}
    Lastly, we derive bounds on $\tilde s$. An upper bound is given by
        \begin{align*}
            \tilde s
            &=\frac{n-\fl{\alpha(n-s(n+m)^{2/3})}-\tilde n}{(\tilde m+\tilde n)^{2/3}}
            \le\frac{n}{\bigr(\tilde m+\xi_2^\star \tilde m/\xi_1^\star-1\bigr)^{2/3}}
            %\le\frac{m/\delta}{\bigl(\tilde m/\xi_1^\star-1\bigr)^{2/3}}
            \le\frac{m/\delta}{\bigl(\tilde m/(1-\delta/3)-1\bigr)^{2/3}}\\
            &\le\frac{\e^{-1}(1-\alpha)m/\delta}{\bigl(\tilde m/(1-\delta/3)-1\bigr)^{2/3}}
            \le\bigl(2(1-\delta/3)\bigr)^{2/3}\delta^{-1}\e^{-1}{\tilde m}^{1/3}\le A(\tilde m+\tilde n)^{1/3},
        \end{align*}
    where $A$ is the constant in front of ${\tilde m}^{1/3}$ in the middle of the second line. We similarly have the lower bound
        \begin{align*}
            \tilde s
            &\ge\frac{n-\alpha n+\alpha s(n+m)^{2/3}-\xi_2^\star(m-\alpha m+1)/\xi_1^\star}{\bigl((m-\alpha m+1)/\xi_1^\star\bigr)^{2/3}}
            =\frac{(1-\alpha)(n-\xi_2^\star m/\xi_1^\star)+\alpha s(n+m)^{2/3}-\xi_2^\star/\xi_1^\star}{\bigl((m-\alpha m+1)/\xi_1^\star\bigr)^{2/3}}\\
            %&=\frac{(1-\alpha)s(m+n)^{2/3}+\alpha s(n+m)^{2/3}-\xi_2^\star/\xi_1^\star}{\bigl((m-\alpha m+1)/\xi_1^\star\bigr)^{2/3}}
            &=\frac{s(m+n)^{2/3}-\xi_2^\star/\xi_1^\star}{\bigl((m-\alpha m+1)/\xi_1^\star\bigr)^{2/3}}
            \ge\frac{s(m+n)^{2/3}-(3-\delta)/\delta}{\bigl(3\delta^{-1}((1-\e)m+1)\bigr)^{2/3}}
            \ge\frac{s(m+n)^{2/3}}{2\bigl(6\delta^{-1}(1-\e)m\bigr)^{2/3}}
            \ge \delta^{2/3}s/8\ge s_1(\delta/3,\delta_0,1).
        \end{align*}
    In the last inequality we used $s\ge s_2$ and the choice of $s_2$.
    
    Now, apply Corollary \ref{cor:exitPointsOffCharDirection} to get
	\begin{align*}
	\P(Z_{\tilde{o}, (m,n)}^{\xi^\star, \ek_2} > 0) 
	&=\P(Z^{\xi^\star,\ek_2}(\tilde m,\tilde n+\fl{\tilde s(\tilde m+\tilde n)^{2/3}})>0)\leq C_1(\delta/3,\delta_0,\e,A){\tilde s}^{-\fluctexp}\\
	&\le C_1(\delta/3,\delta_0,\e,A)(\delta^{2/3}/8)^{-\fluctexp} s^{-\fluctexp}\,.
	\end{align*}
	%note that \tilde s(\tilde m+\tilde n)^{2/3} equals
	%n-\fl{...}-\tilde n and so it is an integer
	
	%OLD TEXT:
	%for some constant $C = C(\delta, \delta_0, \fluctexp, \e) > 0$.
    %since starting at $\tilde{o}$ shifts the total length from $(m,n)$ to the form $(1-\alpha) (m,n)$, so that the coefficient in front is not $s$ anymore, but rather with $\e s$).  
	
	To bound the second probability in \eqref{probGeodInterLine} start by using Lemma \ref{lem:exitPointShift} to write
	%shift-invariance, and Assumption \ref{exitPointTailBound}.  When we take $N_2, s_2$ large enough, we get
	\begin{align*}
	\P\bigl(Z_{\tilde{o}, (m,n)}^{ \xi^\star, \ek_2} < -2s(m+n)^{2/3}\bigr) &\leq \P\bigl(Z_{\tilde{o}, (m,n)}^{ \xi^\star, \ek_2} < -2 \fl{ s(m+n)^{2/3} }\bigr) \\
	&= \P\bigl(Z_{\tilde{o} + \fl{ s(m+n)^{2/3} } \ek_2, (m,n)}^{ \xi^\star, \ek_2} < -\fl{ s(m+n)^{2/3} }\bigr) \\
	%&= \P\bigl(Z_{\tilde{o}, (m,n-\fl{ s(m+n)^{2/3}})}^{\xi^\star, \ek_2} < -\fl{ s(m+n)^{2/3}} \bigr) \\
	&= \P\bigl(Z^{\xi^\star, \ek_2}(m',n') < -\fl{ s(m+n)^{2/3}} \bigr), %\\
	%&\leq \P(Z_{\tilde{o}, (m,n-\fl{ s(m+n)^{2/3}})}^{\xi^\star, \ek_2} < -\fl{ s } (m+n)^{2/3} )\\
	%&\leq C_1 \fl{ s }^{-\fluctexp} \leq C_2 s^{-\fluctexp}.
	\end{align*}
where
	    \[m'=m-\fl{\alpha m}\quad\text{and}\quad n'=n-\fl{\alpha(n-s(n+m)^{2/3})}-\fl{s(m+n)^{2/3}}.\]
    Now we check that we can use Assumption \ref{exitPointTailBound}. 
    %\mathnote{Once we check this, we can comment it all out and say: Similarly to the above, one can check that \eqref{poly-tail} can be used, with $\delta/3$ in place of $\delta$ and with $\kappa=1$ again. Or we could keep these bounds, since they show explicitly why we chose the $N_2$ the way we did. Otherwise, if someone does the bounds differently, they will get different conditions on $N_2$ and will be mystified by ours. We could also comment out the last line of conditions on $N_2$ and just say add ``, provided one makes $N_2$ larger, if necessary.''} 
    We have
	    \[m'\ge\e N_2\ge N_0(\delta/3,1)\]
	   and
	    \[n'\ge(1-\alpha)\bigl(n-s(m+n)^{2/3}\bigr)
	    %\ge(1-\alpha)(1-\frac{\e\delta}{3(\delta+1)}(m/n+1))n
	    \ge(1-\alpha)(1-\e/3)n\ge5(1-\alpha)n/6\ge5\e N_2/6\ge N_0(\delta/3,1).\]
	 Also,
	    \[\frac{\delta}3\le\frac{\delta}{1+1/3+\e^{-1}/N_2}
	    %\le\frac{m}{n+\frac{\delta}{3(\delta+1)}(m+n)+\e^{-1}}
	    \le\frac{(1-\alpha)m}{(1-\alpha)n+\frac{\e\delta}{3(\delta+1)}(n+m)+1}\le\frac{m'}{n'}\le\frac{(1-\alpha)m+1}{5(1-\alpha)n/6}\le\frac6{5\delta}+\frac6{5\e N_2}\le\frac3{\delta}\,. \]
	We already checked that $\xi_1^\star\in(\delta/3,1-\delta/3)$, and we have
	    \begin{align*}
	    \Bigl|\xi_1^\star-\frac{m'}{m'+n'}\Bigr|
	    %&= \Bigl|\frac{s(m+n)^{2/3}(m-\fl{\alpha m})-m\fl{s(n+m)^{2/3}}+n\fl{\alpha m}-m\fl{\alpha(n-s(n+m)^{2/3})}}{(m+n-s(n+m)^{2/3})(m'+n')}
	    &\le \frac{\bigl| s(m+n)^{2/3}(m-\fl{\alpha m})-m\fl{s(n+m)^{2/3}}+n\fl{\alpha m}-m\fl{\alpha(n-s(n+m)^{2/3})} \bigr|}{(1-\alpha)\bigl(1-\frac{\e\delta}{3(\delta+1)}\bigr)^2(m+n)^2}\\
%thing under abs value is between 0 and s(m+n)^{2/3}+2m
        &\le\frac{s(m+n)^{2/3}+2m}{(1-\alpha)\bigl(1-\frac{\e\delta}{3(\delta+1)}\bigr)^2(m+n)^2}
        \le\frac{\frac{\e\delta}{3(\delta+1)}+2}{(1-\alpha)\bigl(1-\frac{\e\delta}{3(\delta+1)}\bigr)^2(m+n)}\\
        &\le\frac{3}{\e\bigl(1-\frac{\e\delta}{3(\delta+1)}\bigr)^2(2N_2)^{2/3}}\cdot(m+n)^{-1/3}\le(m'+n')^{-1/3}.
	    \end{align*}
	We can now use \eqref{poly-tail} to write    
	\begin{align*}
	    \P\bigl(Z^{ \xi^\star, \ek_2}(m',n') < -\fl{s(m+n)^{2/3}}\bigr)
	    &\le\P\bigl(Z^{ \xi^\star, \ek_2}(m',n') < -\fl{s}(m'+n')^{2/3}\bigr) \\
	    &\le C_0(\delta/3)\fl{s}^{-\fluctexp}\le 2^{\fluctexp}C_0 s^{-\fluctexp}.
	\end{align*}

	Using the above bounds in \eqref{probGeodInterLine}, there exists a finite constant $C(\delta, \delta_0, \fluctexp, \e) > 0$ such that
	\begin{align}\label{foofoo1}
	\P(\pi^\star \cap \mc{C}_{\alpha, 2s}^{(m,n)} = \varnothing) \leq Cs^{-\fluctexp}.
	\end{align}
	
    Next, write
    \begin{align*}
        \P\bigl\{Z^{\xi^\star, \ek_1}(m,n) > 0\bigr\} 
        &= \P\Bigl\{Z^{\xi^\star, \ek_1}\bigl(m,n-\fl{s(m+n)^{2/3}}+\fl{s(m+n)^{2/3}}\bigr) > 0\Bigr\}\\
        &= \P\Bigl\{Z^{\xi^\star, \ek_1}(m'',n''+\fl{s''(m''+n'')^{2/3}}\bigr) > 0\Bigr\},
    \end{align*}
    where 
        \[m''=m,\quad n''=n-\fl{s(m+n)^{2/3}},\quad\text{and}\quad
        s''=\frac{s(m+n)^{2/3}}{(m''+n'')^{2/3}}\,.\]
	Similarly to the above, we can check that the conditions of Corollary \ref{cor:exitPointsOffCharDirection} are satisfied, with $\delta/3$ in place of $\delta$ and with $\kappa=1$, provided we choose $N_2$ large enough. 
	Therefore, \eqref{Cor4-3-U} gives the upper bound 
	    \begin{align}\label{foofoo2}
        \P\bigl\{Z^{\xi^\star, \ek_1}(m,n) > 0\bigr\} \le Cs^{-\fluctexp},
        \end{align}
where $C$ is a (possibly different larger) finite positive constant depending only on $\delta$, $\delta_0$, $\fluctexp$, and $\e$.
    %m''\ge N_2\ge \e^{-1}N_1(\delta/3,\delta_0,1)\ge N_1
    %n''\ge n-\frac{\e\delta}{3(\delta+1)}(m+n)\ge n(1-\e/3)\ge 5N_2/6
    %So need: N_2\ge 6N_1(\delta/3,\delta_0,1)/5.
    %m''/n''\ge m/n\ge\delta
    %m''/n''\le m/[(1-\e\delta/(3(\delta+1))(1+m/n))n]\le 1/[(1-\e/3)\delta]\le6/(5\delta)\le2/\delta
    %so (m'',n'')\in S_{\delta/2}\subset S_{\delta/3}
    %\xi^\star is in S_{\delta/3}
    %|\xi^\star-\frac{m''}{m''+n''}|= 
    %m |s(m+n)^{2/3}-\fl{s(m+n)^{2/3}}|
    %/(m+n-s(m+n)^{2/3})(m+n-\fl{s(m+n)^{2/3}})
    %\le m/[(1-\e\delta/(3(\delta+1)))^2 (m+n)^2]
    %\le 1/[(1-\e\delta/(3(\delta+1)))^2 (m+n)]
    %\le 1/[(1-\e\delta/(3(\delta+1)))^2 (2N_2)^{2/3}](m''+n'')^{-1/3}
    %\le (m+n)^{-1/3}  provided N_2 is large
    %s''\ge \frac{s}\ge s-1\ge s/2\ge s_0(\delta/3,1)  
    %need the lower bound by s/2 to replace (s'')^{-\fluctexp} by s^{-\fluctexp}
    %LASTLY:
    %s''\le [\e\delta/(3(1+\delta))](m+n)^{1/3}/[1-\e\delta/(3(1+\delta))]^{2/3}
    %\le [\e\delta/(3(1+\delta))](m''+(1-\e/3)^{-1}n'')^{1/3}/[1-\e\delta/(3(1+\delta))]^{2/3}
    %\le [\e\delta/(3(1+\delta))](1-\e/3)^{-1/3}(m''+n'')^{1/3}/[1-\e\delta/(3(1+\delta))]^{2/3}=A''(m''+n'')^{1/3}
    %So it's all good.
    
	An identical reasoning gives the bounds
	\begin{align}\label{foofoo3}
	\P(\pi_\star \cap \mc{C}_{\alpha, 2s}^{(m,n)} = \varnothing) \leq Cs^{-\fluctexp}
	\quad\text{and}\quad
    \P\bigl\{Z^{\xi_\star, \ek_2}(m,n) < 0\bigr\} 
        %&= \P\Bigl\{Z^{\xi_\star, \ek_2}\bigl(m,n+\fl{s(m+n)^{2/3})}-\fl{s(m+n)^{2/3}}\bigr) > 0\Bigr\}\\
        \le Cs^{-\fluctexp}.
    \end{align}
%	for some (possibly larger) finite 
%	constant $C = C(\delta, \delta_0, \fluctexp, \e) > 0$.  

Next, we argue that $Z^{\xi^\star, \ek_2}(m,n) < 0$ implies that $\pi^{(m,n), \ek_2}$ never goes strictly above $\pi^\star$. 
To argue by contradiction, suppose
there existed a positive integer $k$ and $x \in \ZZ_{\geq 0}^2$ such that $\pi^\star_k = \pi^{(m,n), \ek_2}_k = x$,  $\pi^\star_{k+1} = x+\ek_1$, and $\pi^{(m,n), \ek_2}_{k+1} = x+\ek_2$.  Since  $Z^{\xi^\star, \ek_2}(m,n) < 0$, the upmost geodesic $\pi^\star$ goes from $\zero$ to $\ek_2$ and therefore $k \geq 1$ and 
$x+\ek_1$ lies in the bulk $\NN^2$. Consequently, 
$\pi^\star_{k+1:m+n}$ is a geodesic for $G_{x+\ek_1,(m,n)}$.  
Since $\pi^{(m,n), \ek_2}_{k:m+n}$ is the upmost geodesic of $G_{x,(m,n)}$, it must be that the passage time of $\pi^{(m,n), \ek_2}_{k+1:m+n}$ is at least as large as the passage time of $\pi^\star_{k+1:m+n}$ and the former path never goes strictly below the latter one. Now, the bounds in \eqref{Lem4-4-monotonicity} say that the edge weights on the boundary $\NN \ek_2$ are at least as large as the bulk weights there. Therefore, the passage time of $\pi^\star_{k+1:m+n}$ (which only uses bulk weights) is no larger than the passage time of $\pi^{(m,n), \ek_2}_{k+1:m+n}$, even when the latter uses boundary weights on $\NN \ek_2$ (which is possible if $x$ is on that boundary). But this means that replacing $\pi^\star_{k+1:m+n}$ by $\pi^{(m,n), \ek_2}_{k+1:m+n}$ in $\pi^\star$ gives a geodesic for $G^{\xi^\star}(m,n)$ that at some point goes strictly above $\pi^\star$.
%the passage time is at least as large as that of \pi^\star, but it cannot be strictly larger  since $\pi^\star$ is itself a geodesic. so the new path is also a geodesic.
This contradicts the definition of $\pi^\star$ as the upmost geodesic.  Consequently, $\pi^{(m,n), \ek_2}$ can never go strictly above $\pi^\star$.

Similarly, if $Z^{\xi_\star, \ek_1}(m,n) > 0$, then $\pi^{(m,n), \ek_1}$ never goes strictly right of $\pi_\star$. Consequently, if we have both $Z^{\xi^\star, \ek_2}(m,n) < 0$ and 
$Z^{\xi_\star, \ek_1}(m,n) > 0$, then
all the geodesics of $G(m,n)$ are sandwiched between $\pi^\star$ and $\pi_\star$. If, furthermore, $\pi^\star$ and $\pi_\star$ both intersect $\mc{C}_{\alpha, 2s}^{(m,n)}$, then both $\pi^{(m,n),\ek_k}$, $k\in\{1,2\}$, are forced to intersect it as well. We have thus shown that 
    \begin{align*}
    &\Bigl\{(\pi^{(m,n), \ek_1} \cap \mc{C}_{\alpha, s}^{(m,n)} = \varnothing) \cup (\pi^{(m,n), \ek_2} \cap \mc{C}_{\alpha, s}^{(m,n)} = \varnothing)\Bigr\}\\ 
    &\qquad\subset\Bigl\{Z^{\xi^\star, \ek_1}(m,n)>0\Bigr\}
    \cup\Bigl\{Z^{\xi_\star, \ek_2}(m,n)<0\Bigr\}
    \cup\Bigl\{\pi^\star \cap \mc{C}_{\alpha, 2s}^{(m,n)} = \varnothing\Bigr\}
    \cup\Bigl\{\pi_\star \cap \mc{C}_{\alpha, 2s}^{(m,n)} =\varnothing\Bigr\}.
    \end{align*}
This, together with (\ref{foofoo1}-\ref{foofoo3}) complete the proof of the lemma.
	\end{proof}

%%%%%%%%%%%%%%%%%%%%%%%%%%%%%%%%%%%%%%%%%%%%%
%%%%%%%%%%%%%%%%%%%%%%%%%%%%%%%%%%%%%%%%%%%%%
%%%%%%%%%%%%%%%%%%%%%%%%%%%%%%%%%%%%%%%%%%%%%

\section{Non-existence of bi-infinite geodesics}\label{sec:nonexist}

We begin by proving non-existence of non-trivial axis-directed bi-infinite geodesics, which is essentially an immediate consequence of the uniqueness of axis-directed semi-infinite geodesics.

\begin{lemma} \label{lem:noAxisDirectBiInfinites} With probability one, for each $x \in \ZZ_{\geq 0}^2$ and  $\ell \in \{1, 2\}$, the only semi-infinite geodesic starting at $x$ satisfying $\varliminf_{k \rightarrow \infty} k^{-1} x_k \cdot \ek_{3-\ell} = 0$ is the trivial geodesic $\{x + k\ek_\ell \}_{k=0}^\infty$.  
\end{lemma}
	\begin{proof}
		The proof of this result is essentially the same as that of \cite[Lemma A.6]{Jan-Ras-Sep-21-}, where there is an additional assumption that the weight distribution is continuous. We include the proof for completeness. It suffices to prove the result for semi-infinite geodesics starting at the origin which are $\ek_1$-directed.  Fix a strictly decreasing sequence of directions $\xi_i \in \ri\sU$ such that $\xi_i \searrow \ek_1$ and $\gpp$ is differentiable at $\xi_i$ for each $i$.  
		By Lemma 4.1(b) in \cite{Geo-Ras-Sep-17-ptrf-2}, each $B^{\xi_i}$ given by 
		Theorem \ref{thm:coupledGenericLPP}		
		produces
		%starting from the origin, following the  minimal increments of $B^{\xi_i}$ and moving in the $\ek_2$ direction in case of ties, 
		an upmost semi-infinite geodesic $x_{0: \infty}^i$ that starts at the origin and follows the minimal increments of $B^{\xi_i}$, taking an $\ek_2$ increment in case of a tie.  By \cite[Theorem 4.3]{Geo-Ras-Sep-17-ptrf-2}, for each $i$, the limit points of $x^i_n/n$, as $n\to\infty$, lie on the same (possibly degenerate) linear segment of $\gpp$ that contains $\xi_i$.
		
		Due to the uniqueness of the upmost geodesic between any pair of points $x\le y$, 
		%we have that for $i < j$, 
		%if the $x_{0: \infty}^{\zero, \xi_j, \ek_2}$ geodesic ever went above the $x_{0: \infty}^{\zero, \xi_i, \ek_2}$ geodesic, a loop would be created since $\xi_i > \xi_j$.  But then $x_{0: \infty}^{\zero, \xi_i, \ek_2}$ would have taken the upper portion of this loop, as it is the upmost geodesic generated by $B^{\xi_i}$.  So 
		%$x_{0: \infty}^{\zero, \xi_j, \ek_2}$ must stay weakly to the right of $x_{0: \infty}^{\zero, \xi_i, \ek_2}$.  
		%By the same reasoning, 
		any $\ek_1$-directed semi-infinite geodesics starting from the origin must stay weakly to the right of all of the geodesics $x_{0: \infty}^i$.  

		The result now follows if we show that for any $m\in\ZZ_{\ge0}$ and any $i$ large enough,  $x^i_{0:m}=\flint{\zero,m \ek_1}$. We prove this by induction. This claim is trivial for $m=0$. Suppose the claim is true for some $m\in\ZZ_{\ge0}$. Lemma 5.1 in \cite{Geo-Ras-Sep-17-ptrf-2} says that $B^{\xi_i}(m\ek_1,m\ek_1+\ek_2)\to\infty$ as $i\to\infty$.
		This implies that for $i$ large enough $B^{\xi_i}(m\ek_1,m\ek_1+\ek_2)>\w_{m\ek_1}=B^{\xi_i}(m\ek_1,(m+1)\ek_1)$, which implies that $x^i_{m+1}=(m+1)\ek_1$.
		%		
%
%		Fix a positive integer $N$.  The Busemann functions $B^{\xi_i}$ converge as $i \rightarrow \infty$. The vertical increments diverge to infinity, and the horizontal increments converge to the $\omega$ weights.  Take a decreasing rational sequence $\ek_j \searrow 0$.  For each $j$, there is a positive integer $m_j$ such that for all $i \geq m_j$, the smallest vertical increment of $ B^{\xi_i}$ is larger than the largest horizontal increment of $B^{\xi_i}$ within the box $\flint{0, N}$ with probability at least $1-\ek_j$.  Therefore, within the box  $\flint{0, N }$ every $\ek_1$-directed semi-infinite geodesic starting from the origin almost surely takes its first $N$ steps in the $\ek_1$ direction.  This holds for all  $N$, so almost surely the only $\ek_1$-directed semi-infinite geodesic starting at the origin is $\{(0,0) + k\ek_1 \}_{k=0}^\infty$.  
	\end{proof}
	
Next, we turn to interior-directed bi-infinite geodesics.   
Recall that the passage times $G_{x,y}^{\xi}$ that use the boundary weights 
$\{I_{x+k\ek_1}^\xi,J_{x+k\ek_2}^\xi:k\in\N\}$ on the southwest boundary of $x+\ZZ^2_+$
give a stationary LPP process satisfying \eqref{Gadd}. We will now need to consider the stationary LPP process that corresponds to putting appropriate weights on the northeast boundary. 
To this end, define the reflected weights $\hat\w=(\hat\w_x)_{x\in\ZZ^2}$ with $\hat\w_x=\w_{-x}$. Define the boundary weights 
%\mathnote{Check this gives us what we want! Namely, no minus sign mistakes!!  S:  Yes I think this is right with the signs.  The negative with the reflected weights and the negative in $-x$ mean we will get the correct weights around $x$.}
    \begin{align}\label{IJhat}
    \hat I^\xi_x(\w)=I^\xi_{-x}(\hat\w)\quad\text{and}\quad\hat J^\xi_x(\w)=J^\xi_{-x}(\hat\w).
    \end{align}

    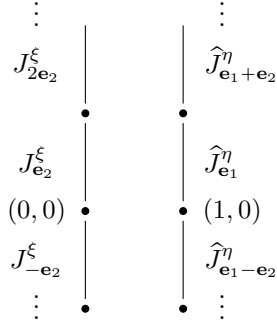
\begin{figure}[ht]
		\begin{center}
			\begin{tikzpicture}[scale = 1.3, baseline=(current bounding box.north)]	
				\foreach \y in {-1, 0, 1}
				{	
					\filldraw (0,\y) circle (1pt);
					\filldraw (1,\y) circle (1pt);
					
					\draw (0,\y+0.1) -- (0,\y+0.9);
					\draw (1,\y+0.1) -- (1,\y+0.9);		
				}
			
				\node[] at (-0.5,0) { $(0,0)$};
				\node[] at (1.5,0) { $(1,0)$};
				
				\node[] at (-0.5,-0.5) {$J^{\xi}_{-\ek_2}$};
				\node[] at (1.6,-0.5) {$\hat{J}^{\eta}_{\ek_1-\ek_2}$};
				
				\node[] at (-0.5,0.5) {$J^{\xi}_{\ek_2}$};
				\node[] at (1.4,0.5) {$\hat{J}^{\eta}_{\ek_1}$};
				
				\node[] at (-0.5,1.5) {$J^{\xi}_{2\ek_2}$};
				\node[] at (1.6,1.5) {$\hat{J}^{\eta}_{\ek_1 + \ek_2}$};
				
				\node[] at (-0.5, 2.1) {$\vdots$};
				\node[] at (1.4, 2.1) {$\vdots$};
				
				\node[] at (-0.5, -0.9) {$\vdots$};
				\node[] at (1.4, -0.9) {$\vdots$};
			\end{tikzpicture}
		\end{center}
		\caption{The edges involved in $S_n^{\xi, \eta}$. }\label{fig:rwEdges}
	\end{figure}

	Given $\xi,\eta\in\ri\sU$ let 
    \begin{align}
	S_n^{\xi,\eta}=\begin{cases}
	\sum_{j=1}^n(J_{j\ek_2}^{\xi}-\widehat{J}_{\ek_1+(j-1)\ek_2}^{\eta}),&n\in\ZZ_{\ge1},\\[4pt]
	0,&n=0,\\[3pt]
	-\sum_{j=n+1}^0(J_{j\ek_2}^{\xi}-\widehat{J}_{\ek_1+(j-1)\ek_2}^{\eta}),&n\in\ZZ_{\le-1}.
	\end{cases}\label{Sdef}
	\end{align}
	
	Given $\xi^\star,\xi_\star,\eta^\star,\eta_\star\in\ri\sU$, we can use Theorem \ref{thm:coupledGenericLPP} to couple the  weights 
	\begin{align}\label{bigcoupling}
	\{\w_x,I^{\xi^\star}_x,J^{\xi^\star}_x,I^{\xi_\star}_x,J^{\xi_\star}_x,
	I^{\eta^\star}_x,J^{\eta^\star}_x,I^{\eta_\star}_x,J^{\eta_\star}_x:
	x\in\ZZ^2\}.
	\end{align}
	This produces a coupling of 
	$\{\w_x,J^{\xi^\star}_x,J^{\xi_\star}_x,\hat J^{\eta^\star}_x,\hat J^{\eta_\star}_x:
	x\in\ZZ^2\}$
	and of 
	$\{\w_x,I^{\xi^\star}_x,I^{\xi_\star}_x,\hat I^{\eta^\star}_x,\hat I^{\eta_\star}_x:
	x\in\ZZ^2\}$.
%\note{S: Later in section 4, we will take $d_2 = N^{b(\fluctexp)}$ and $s = N^{c(\fluctexp)}$.  These exponents must satisfy
%\begin{align*}
%&\frac{1}{3\fluctexp} < b(\fluctexp) \\
%&< c(\fluctexp) < \frac{1}{6}
%\end{align*}
%Then $a(\fluctexp) = \frac{2}{3} - 2c(\fluctexp)$.}

\begin{assumption}\label{ass:RW}
There exist an $\expnt\in(1/3,2/3)$ and a $\delta_0\in(0,1)$ such that for any $\delta \in (0,\delta_0)$,
there exist positive finite constants $C_3(\delta)$ and $N_3(\delta)$ such that for all $N\ge N_3$, and all $\eta_\star,\eta^\star,\xi_\star,\xi^\star\in\ri\sU$ with $\ek_1$-coordinate in $(\delta,1-\delta)$ and such that
			\begin{align}\label{xistar-etastar}
			 -N^{-\expnt/2}
			\leq \xi_{\star}\cdot \ek_1 - \eta^\star\cdot \ek_1 < 0\quad\text{and}\quad
            -N^{-\expnt/2} 
			\leq \eta_{\star}\cdot \ek_1 - \xi^\star\cdot \ek_1<0,
			\end{align}
			we have
\begin{align}\label{bd:RW}
\P\Bigl\{\sup_{0<k\le N^{2/3}} S_k^{\xi_\star,\eta^\star}\le0\quad\text{and}\sup_{-N^{2/3}+1 \le k<0}S_k^{\xi^\star,\eta_\star}\le0\Bigr\}\le C_3N^{-\expnt}.
\end{align}
\end{assumption}

By Lemma \ref{lm:assRWholds}, this assumption is satisfied for any $\expnt\in(1/3,2/3)$ when $\w_0$ is geometrically distributed. This assumption is verified in the exponential model in \cite[Lemma C.1]{Bal-Bus-Sep-20} with $a_0 = 2/5.$
%\mathnote{Maybe add some explanation as to why one would expect this to hold for general weights
%Namely, each walk should be a very mixing process and one should be able to use strong coupling and get the same estimates as for classical random walk. This is a CLT type scaling, right?  The question is though: 
\normalmarginpar
%\note{when weights are general, why would we expect the probability of the intersection to be bounded by a constant times the product of probabilities?  Is there some kind of negative correlation? Or some other reason? Or maybe one can look at the increments in two intervals that are separated, so that some decorrelation can kick in. Could there be some other event $U^{u,v}$ that would be included inside the random walk event [with this gap] and that would still give us the no bi-infinites conclusion? Can we see something like that in our proof?} 
%\mathnote{would be good to point to some spot(s) in Timo's paper that show that this assumption is also satisfied for exponential weights?  S:  In their paper, Lemma \cite[C.1]{Bal-Bus-Sep-20} is where the random walk bound is done where the steps are the difference of independent exponentials: $\text{Exp}(\alpha) - \text{Exp}(\beta)$.  The display between (5.21) and (5.22) is where they apply it to the random walks with $a_0 = 2/5$.}

Theorem \ref{thm:noBiInfinites} now follows from 
Theorem \ref{geometricExitPointTailBound}, Lemma \ref{lm:assRWholds}, and the following, more general result.

\begin{theorem}\label{main-thm}
Suppose Assumptions \ref{exitPointTailBound} and \ref{ass:RW} hold with 
    \begin{align}\label{a-nu}
    \expnt<\frac{2(\fluctexp-1)}{3\fluctexp}\,.
    \end{align}
Then with $\P$-probability one there are no non-trivial bi-infinite geodesics.    
\end{theorem}

\begin{remark}%\reversemarginpar\note{Check this remark}
By exchanging the roles of the two axes, one sees that the above theorem also holds if Assumption \ref{ass:RW} holds with $I^\bbullet_{j\ek_1}$-increments instead of $J^\bbullet_{j\ek_2}$. We expect that if the assumption holds with one set of increments, then it holds with the other set as well.
\end{remark}%\medskip
%\firasnote{Add somewhere a remark about the symmetric condition with $\ek_1$-increments and how one of the two is enough, but most likely both hold} 

Given $\delta\in(0,1)$ and a positive integer $N$, define the southwest boundary,
\begin{equation}
\partial^{N, \delta} = ( \{-N\} \times \flint{ -N, -\delta N } ) \cup ( \flint{ -N, -\delta N } \times \{-N\}),
\end{equation}
and the northeast boundary,
\begin{equation}\label{nebdry}
\hat{\partial}^{N, \delta} = ( \{N\} \times \flint{ \delta N, N } ) \cup ( \flint{ \delta N, N } \times \{N\}).
\end{equation}
%Define the event 
%\begin{equation}
%W_{N, \delta} = \bigl\{ \exists u \in \partial^{N, \delta}, v \in \hat{\partial}^{N, \delta} \text{ such that at least one geodesic of } G_{u,v} \text{ goes through the origin}  \bigr\}.
%\end{equation}

By Lemma \ref{lem:noAxisDirectBiInfinites} a nontrivial bi-infinite geodesic must eventually take an $\ek_1$ step. 
Then by the shift-invariance of $\bbP$, to prove Theorem \ref{main-thm} it suffices to show that almost surely there are no nontrivial bi-infinite geodesics that take the edge $(\zero,\ek_1)$.  Thus, this theorem follows from Lemma \ref{lem:noAxisDirectBiInfinites} and the next result. 
%\firasnote{Maybe we should state the general theorem that says there are no bi-infinites, under the two assumptions.}

	For $u\le v$ in $\ZZ^2$ define the event
	\begin{align}\label{Uuv}
	U^{u, v}=\left\{\text {at least one geodesic of } G_{u,v} \text { goes through both $\zero$ and $\ek_1$}\right\}. 
	\end{align}

% See the note at the beginning of Section Geodesic Fluctuation Bounds for an explanation of how this exponent might be changed
\begin{theorem}  \label{thm:noBiInfinitesAwayFromAxes} Suppose Assumptions \ref{exitPointTailBound} and \ref{ass:RW} hold with \eqref{a-nu} satisfied. Let
    \begin{align}\label{a:def}
    \expntt=\min\Bigl(\expnt,\bigl(\frac13-\frac\expnt2\bigr)\fluctexp\Bigr)\in(1/3,2/3).
    %\expntt(\fluctexp) = \frac{2(2+\fluctexp)}{3(6+\fluctexp)}\textcolor{red}{=\frac13+\frac{\fluctexp-2}{3(6+\fluctexp)}}.
    \end{align}
%$b(\fluctexp) = \frac{\fluctexp-2}{3(6+\fluctexp)}$.  
For each $\delta \in (0,\delta_0)$ there exist positive finite constants $N_4(\delta, \delta_0, \fluctexp,\expnt)$ and  $C_4(\delta, \delta_0, \fluctexp,\expnt)$ such that for all $N\ge N_4$
\begin{align*}
	\P\Bigl(\bigcup_{u \in \partial^{N,\delta}, v \in \hat{\partial}^{N,\delta}} U^{u, v}\Bigr) \leq C_4 N^{-(\expntt- 1/3)}.
\end{align*}
\end{theorem}  
	
%\reversemarginpar\note{agree with all this??}
The reason behind the relation \eqref{a-nu} is that if $\nu$ is close to 2, then this affects the bound \eqref{poly-tail} and, as a consequence, we do not have good control over the geodesic fluctuations in \eqref{main-sec3}. Then, when using \eqref{bd:RW} in the argument against the existence of bi-infinite geodesics, we need to allow for a larger interval in \eqref{xistar-etastar}, which means  using a smaller $\expnt$.
That said, it should be the case that if the i.i.d.~environment has finite exponential moments, then Assumption \ref{exitPointTailBound} holds for all $\fluctexp>2$ and Assumption \ref{ass:RW} holds for $\expnt\in(1/3,2/3)$.% \mathnote{agree with all this?  S: Looks good to me :)}

	The rest of the section builds up towards the proof of the above theorem.
	%We will fix the parameter $\delta\in(0,\delta_0)$ and may sometimes omit it from the superscripts and subscripts.  
	Define the vertical segment 
	\begin{align}\label{mcI}
	\mc{I} = \{0\} \times \flint{ -N^{2/3}, N^{2/3}  }.
	\end{align}
	%which is indexed by $I =  \flint{ -N^{\frac{2}{3}}, N^{\frac{2}{3}}  }$.  
	For $N \geq 1$, $1 \leq s \leq \frac{\delta}{4}N^{1/3}$, and $o \in  \ZZ_{>0}^2\cup\ZZ_{<0}^2$, define the directions 
	\begin{align}\label{defXimn}
	\begin{split}
	&\ximn(o) = \frac{o}{o_1+o_2}\,, \quad
	\ximn_{\star}(o) = \ximn(o) + (-sN^{-1/3}, sN^{-1/3})\,,\\ %\quad 
	&\text{and} \quad
	\ximn^{\star}(o) = \ximn(o) + (sN^{-1/3}, -sN^{-1/3})\,.
	\end{split}
	\end{align}

	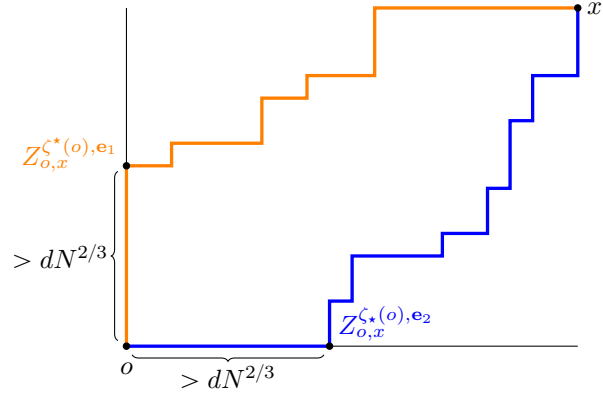
\begin{figure}[ht]
	\begin{center}
	\begin{tikzpicture}[scale = 0.3, baseline=(current bounding box.north)]
		\draw (0,15) -- (0,0) -- (20,0);
		
		\draw[line width=1.25pt, orange] (0,0) -- (0,8) -- (2,8) -- (2,9) -- (6,9) -- (6,11) -- (8,11) -- (8,12) -- (11,12) -- (11,15) -- (20,15);
		\filldraw (0,8) circle (4pt);
		\node[orange] at (-2.5,8.5) { $Z_{o,x}^{\zeta^{\star}(o), \ek_1}$};
		\draw[decorate,decoration={brace,raise=0.75ex}] (0,0.25) -- (0,7.75) 
		node[midway,left=0.75ex]{ $> dN^{2/3}$} ;
		
		\draw[line width=1.25pt, blue] (0,0) -- (9,0) -- (9,2) -- (10,2) -- (10,4) -- (14,4) -- (14,5) -- (16,5) -- (16,7) -- (17,7) -- (17,10) -- (18,10) -- (18,12) -- (20,12) -- (20,15);
		\filldraw (9,0) circle (4pt);
		\node[blue] at (11.5,1) { $Z_{o,x}^{\zeta_{\star}(o), \ek_2}$};
		\draw[decorate,decoration={brace,mirror,raise=0.75ex}] (0.25,0) -- (8.75,0) 
		node[midway,below=0.75ex]{ $> dN^{2/3}$} ;  
		
		\filldraw (0,0) circle (4pt);
		\node[] at (0,-1) { $o$};
		
		\filldraw (20,15) circle (4pt);
		\node[] at (20.75,15) { $x$};
		
	\end{tikzpicture}
	\end{center}
	\caption{An illustration of the high probability event in Lemma \ref{lem:geodesicsInNGrid}.  The upmost geodesic of $G_{o,x}^{\zeta_\star(o)}$ %shown in blue 
	exits at least $dN^{2/3}$ to the right of $o$.  The rightmost geodesic of $G_{o,x}^{\zeta^\star(o)}$ %shown in orange 
	exits at least $dN^{2/3}$ above $o$. }\label{fig:geodesicsInNGrid}
	\end{figure}

	\begin{lemma} \label{lem:geodesicsInNGrid}  
	Suppose Assumption \ref{exitPointTailBound} holds. For any $\delta\in(0,\delta_0)$
	there exist finite positive constants $C_5(\delta, \delta_0, \fluctexp)$, 
	$N_5(\delta, \delta_0)\ge8\delta^{-3}$, and $s_5(\delta, \delta_0)$ such that for all $N \geq N_5$, if %\note{the $\delta/64$ is to leave room for $8d\le s\le\delta N^{1/3}/4$!!}
		\begin{equation}
		1 \leq d \leq \frac{\delta}{64} N^{1/3} \qquad \text{ and } \qquad  \max(s_5,8d)\leq s \leq \frac{\delta}{4} N^{1/3},   \label{dsCondition}
		\end{equation}
	then for all $x \in \mc{I}$, and $o \in \partial^{N,\delta}$, 
	\begin{align}
	\P(Z_{o, x}^{\ximn_\star(o), \ek_2} &\leq
	dN^{2/3}) \leq C_5 s^{-\fluctexp}  \label{gridGeodesicRight} \\
	\P(Z_{o, x}^{\ximn^\star(o), \ek_1} &\geq -dN^{2/3}) \leq C_5 s^{-\fluctexp} \label{gridGeodesicUp}.
	\end{align}
	\end{lemma}

	\begin{proof}
	    The condition that $N\ge8/\delta^3$ guarantees that $-\delta N+1\le -N^{2/3}$,
	    %N\ge8\delta^3 implies both N\ge1  and \delta N\ge 2N^{2/3}
	    %and so -\delta N+1\le -\delta N+N^{2/3}\le -N^{2/3} 
	    which implies that 
	    $-\fl{\delta N}\le -\fl{N^{2/3}}$ 
	    and hence $\partial^{N,\delta}$ is entirely below $\mc{I}$. 
		We prove \eqref{gridGeodesicRight} and the second bound follows analogously. 
		Let $o = -(aN, bN)$ where $a \vee b = 1$ and $a \wedge b \in [\delta, 1]$.  Abbreviate $\xi_\star = \ximn_\star(o)$.  
		The upmost geodesic from $o$ to $\fl{ N^{2/3} } \ek_2 $ must stay above any geodesic from $o$ to $x \in \mc{I}$.  This, Lemma \ref{lem:exitPointShift}, and shift-invariance give
		\begin{equation}
		\begin{aligned}
		&\P\Bigl\{Z_{o, x}^{\xi_\star, \ek_2} \le dN^{2/3}\Bigr\} 
		\leq \P\Bigl\{Z_{o, \fl{ N^{2/3}} \ek_2 }^{\xi_\star, \ek_2}  \leq \fl{dN^{2/3}}\Bigr\} 
		= \P\Bigl\{Z_{o + \fl{ dN^{2/3} } \ek_1, \fl{ N^{2/3}} \ek_2 }^{\xi_\star, \ek_2} < 0 \Bigr\} \\
		&\qquad= \P\Bigl\{Z_{o, \fl{ N^{2/3}} \ek_2 - \fl{ dN^{2/3} } \ek_1 }^{\xi_\star, \ek_2} < 0\Bigr\} 
		= \P\Bigl\{Z^{\xi_\star, \ek_2}(aN-\fl{ dN^{2/3}}, bN+\fl{ N^{2/3}}) < 0\Bigr\}.
		\end{aligned} \label{gridGeodesicBoundSetup}
		\end{equation}
		Next, we check that we can apply Corollary \ref{cor:exitPointsOffCharDirection} with
        \[m = aN - \fl{ dN^{2/3} },\quad n=\Big\lfloor\frac{m\xi_\star\cdot \ek_2}{\xi_\star\cdot \ek_1}\Big\rfloor=\Big\lfloor\frac{b+(a+b)sN^{-1/3}}{a-(a+b)sN^{-1/3}}\cdot m\Big\rfloor,\quad\text{and}
        \quad s'=\frac{n-bN-\fl{N^{2/3}}}{(m+n)^{2/3}}\,,\]
        %\note{details can be commented out. but maybe better leave them?}
        if we take $N$ large enough and $s$ as in \eqref{dsCondition}.
		Here are the details. Take 
		    \[N_5=\max\bigl(8\delta^{-3}+1,64(N_1(\delta/4,\delta_0,1)+1)/(63\delta^2)\bigr)\]
		    %=\max\bigl(64(N_1(\delta/4,\delta_0,1)+1)/(63\delta^2),128/(63\delta^2)\bigr)\]
		    and
		    \[s_5=\max\bigl(4,8/(3\delta),2^{11/3}s_1(\delta/4,\delta_0,1)/(3\delta),4/(1+\delta),2^{10/3}s_1(\delta/4,\delta_0,1)/\delta^{2/3}
		    \bigr).\]
		Take $N\ge N_5$ and $d$ and $s$ as in \eqref{dsCondition}.
		Then $m\ge 63\delta N/64\ge N_1$, $n\ge bm/a-1\ge63\delta^2 N/64-1\ge N_1$, and
           \[\frac\delta4
            %\le\frac{\delta-\delta/2}{1+\delta/2}
            \le\frac{a-(a+b)\delta/4}{b+(a+b)\delta/4}\le\frac{m}{n}\le\frac{m}{bm/a-1}\le\frac{a}{b-64a/(63\delta N_5)}
            %\le\frac1{\delta-\delta/2}
            \le\frac2\delta\,.\]
		Also
		    \begin{align*}
		   (n+m)^{1/3} \ge s'
		    %&\ge\frac{\frac{b+(a+b)sN^{-1/3}}{a-(a+b)sN^{-1/3}}\cdot(aN-dN^{2/3})-1-bN-N^{2/3}}{(aN)^{2/3}\Bigl(1+\frac{b+(a+b)\delta/4}{a-(a+b)\delta/4}\Bigr)^{2/3}}\\
		    %&=\frac{\frac{abN+a(a+b)sN^{2/3}-bdN^{2/3}-d(a+b)sN^{1/3}-abN+b(a+b)sN^{2/3}}{a-(a+b)sN^{-1/3}}-1-N^{2/3}}{(aN)^{2/3}\Bigl(1+\frac{b+(a+b)\delta/4}{a-(a+b)\delta/4}\Bigr)^{2/3}}\\
		    %&\ge\frac{(a+b)^2sN^{2/3}-bdN^{2/3}-d(a+b)\delta N^{2/3}/4-2N^{2/3}}{(aN)^{2/3}\Bigl(1+\frac{b+(a+b)\delta/4}{a-(a+b)\delta/4}\Bigr)^{2/3}}\\
		    &\ge\frac{(a+b)^2s-bd-d(a+b)\delta/4-2}{a^{2/3}\Bigl(1+\frac{b+(a+b)\delta/4}{a-(a+b)\delta/4}\Bigr)^{2/3}}
		    \ge\frac{s/2-2d}{\bigl(1+\frac{2+\delta}{\delta}\bigr)^{2/3}}
            \ge\frac{\delta^{2/3}s}{2^{10/3}}
		    \ge s_1\,.
            \end{align*}
		And
		\begin{align}\label{xistar in cone}
		\frac{\delta}4
		%\le\frac{\delta(3-\delta)}{4(1+1/\delta)}
		\le\frac1{1+1/\delta}-\frac\delta4
		\le\frac1{1+{o_2/o_1}}-\frac\delta4\le\xi_\star\cdot \ek_1\le\frac1{1+{o_2/o_1}}\le\frac1{1+\delta}
		%=1-\frac{\delta}{1+\delta}
		\le 1-\frac\delta2\,.
		\end{align}
		And lastly, 
		\begin{align*}
		\Bigl|\xi_\star\cdot \ek_1-\frac m{m+n}\Bigr|
		%&=(m+n)^{-1}\Bigl|n\xi_\star\cdot \ek_1-m\xi_\star\cdot \ek_2\Bigr|
		&=\Bigl|n-\frac{m\xi_\star\cdot \ek_2}{\xi_\star\cdot \ek_1}\Bigr|\,(m+n)^{-1}\xi_\star\cdot \ek_1
		\le(m+n)^{-1}.
		\end{align*}

		Thus, \eqref{Cor4-3-L} gives 
		\begin{align*}
		\P\Bigl\{Z_{o, x}^{\xi_\star, \ek_2} \leq dN^{2/3} \Bigr\} \leq \P\Bigl\{ Z^{\xi_\star, \ek_2}(m, n-s'(m+n)^{2/3})<0 \Bigr\} 
		\leq C s^{-\fluctexp}
		\end{align*}
		for some positive finite constant $C(\delta, \delta_0, \fluctexp)$.
	\end{proof}
	
	To control coarse graining on the scale $N^{2/3}$, we use the parameters $d_1$ for the southwest boundary and $d_2$ for the northeast boundary of the square $\flint{-N, N }^2$.  Let $d=(d_1,d_2)$. For $o \in \partial^{N,\delta}$ define
	\[
	\mc{I}_{o,d} = \Bigl\{u \in \partial^{N,\delta} : |u-o|_1 \leq \frac{d_1 N^{2/3}-1}{2}\Bigr\}.
	\]
	Because $\mc{I}_{o,d}$ is a connected portion of the boundary of a square, it contains a unique point $o_c$ such that $o_c \leq u$ coordinate-wise for each point $u \in \mc{I}_{o,d}$.   	
	
	For $s \leq \frac{\delta}{4} N^{1/3}$, define the directions 
	\begin{align}\label{xistars}
	\xi_{\star} = \ximn_{\star}(o_c)\quad\text{and}\quad\xi^\star = \ximn^{\star}(o_c)
	\end{align}
	as in \eqref{defXimn}.  
	%Then $\xi_\star, \xi^\star \in \mc{U} \cap S_{\frac{d}{2}}$.  
	%We consider the increment-stationary LPP processes $G_{o_c, \bigcdot}^{\xi_\star}$ and $G_{o_c, \bigcdot}^{\xi^\star}$ coupled as in Theorem \ref{thm:coupledGeomLPP}.  For $i, j \geq 1$,
	%\[
	%\omega_{o_c + i\ek_1} \leq I_{o_c + i\ek_1}^{\xi^\star} \leq I_{o_c + i\ek_1}^{\xi_\star} \quad \text{ and } \quad \omega_{o_c + j\ek_2} \leq J_{o_c + j\ek_2}^{\xi_\star} \leq J_{o_c + j\ek_2}^{\xi^\star}.
	%\]
	Use Theorem \ref{thm:coupledGenericLPP} to couple the weights 
	$\{\w_x,I_x^{\xi^\star},J_x^{\xi^\star},I_x^{\xi_\star},J_x^{\xi_\star}:x\in\bbZ^2\}$ so that \eqref{Lem4-4-monotonicity} holds almost surely and for all $x\in\bbZ^2$.
	%
	%
	%We consider $N \geq \delta^{-3}$ so that $-\delta N \leq -N^{2/3}$ and $\partial^N$ is entirely below $\mc{I}$. 
	%Recall the increments defined in \eqref{G-IJ} and 
	%for $u \in \mc{I}_{o,d}$ and $j\in I$ abbreviate
	%\[
	%J_j^u = J^{[u]}_{j\ek_2},%=G_{u, j\ek_2} - G_{u, (j-1)\ek_2} 
	%\quad J_j^{\xi_\star} = J^{\xi_\star}_{j\ek_2},
	%\quad \text{and} \quad J_j^{\xi^\star} = J^{\xi^\star}_{j\ek_2}.\]
    %
	Define the event
	\begin{align}
	A_{o, d}=\Bigl\{Z_{o_{c},-\fl{N^{2 / 3}} \ek_2}^{\xi^{\star}, \ek_1}<-d_{1} N^{2/3}\quad\text{and}\quad
	Z_{o_{c},\fl{N^{2 / 3}}\ek_{2}}^{\xi_{\star}, \ek_2}>d_{1} N^{2/3}\Bigr\}.
	\end{align}
	Recall the boundary weights defined in \eqref{G-IJ}.

	\begin{lemma}\label{lem:AodInequalities}	
		Suppose Assumption \ref{exitPointTailBound} holds.
		%There exist finite positive constants $C_6(\delta, \delta_0, \fluctexp)$ and $N_6(\delta, \delta_0)$ such that 
		Then for any $\delta\in(0,\delta_0)$,
	 $N \geq N_5(\delta,\delta_0)$,  $o \in \partial^{N,\delta}$, and $(d_1, s)$ satisfying \eqref{dsCondition},
		\begin{equation}
		\P(A_{o,d}^c) \leq 2C_5(\delta,\delta_0,\fluctexp) s^{-\fluctexp}. \label{boundAodComp}
		\end{equation}
		
		On the event $A_{o,d}$, the following inequalities hold for all $x\in\mc I$ and $u \in \mc{I}_{o,d}$:
		\begin{equation}
		J_{x+\ek_2}^{\xi_\star} \leq J_{x+\ek_2}^{[u]} \leq J_{x+\ek_2}^{\xi^\star}. \label{monotonicityBetweenLPP}
		\end{equation}
	\end{lemma}
	\begin{proof}
		Lemma \ref{lem:geodesicsInNGrid} implies \eqref{boundAodComp}.  
		We prove the second inequality of \eqref{monotonicityBetweenLPP}  and the first inequality follows similarly.  Let $\tilde{G}_{x, y}$ be the LPP process on the quadrant $o_c + \ZZ_{\geq 0}^2$ with weights $\tilde{\omega}_{o_c} = 0$, $\tilde{\omega}_{o_c + j\ek_2} = J_{o_c + j\ek_2}^{\xi^\star}$ for each $j \geq 1$, and $\tilde{\omega}_{o_c + x} = \omega_{o_c + x}$ whenever $x \cdot \ek_1 > 0$.  
		
		First consider the case that $u = o_c + j\ek_2$ for some $j \geq 0$.  On the event $A_{o,d}$, we have that the rightmost geodesic of $G^{\xi^\star}_{o_c,-\fl{N^{2/3}}\ek_2}$
		exits the boundary above $o_c+d_1 N^{2/3} \ek_2$.  Therefore, for any $x \in (-\fl{N^{2/3}}+\ZZ_{\ge0})\ek_2$, every geodesic of $G^{\xi^\star}_{o_c,x}$ must exit the boundary above $o_c+d_1 N^{2/3} \ek_2$, i.e., 
		\[
		Z_{o_c, x}^{\xi^\star, \ek_1} < -d_1 N^{2/3}.
		\]
		Thus, every geodesic of $G_{o_c, x}^{\xi^\star}$ includes $u$ and $u+\ek_2$.  
		%Explanation:
		%$\mc I=\flint{\ce{o-(d_1N^{2/3}-1)/2},\fl{o+(d_1N^{2/3}-1)/2}}$
		%and $\fl{o+(d_1N^{2/3}-1)/2}}+1-\ce{o-(d_1N^{2/3}-1)/2}\le d_1N^{2/3}$.
		Since the weights used by $\tilde{G}$ and $G^{\xi^\star}$ are the same away from the horizontal boundary, and on that boundary the weights used by the former are smaller than the ones used by the latter, we get that 
		\[
		G_{o_c, x+\ek_2}^{\xi^\star} - G_{o_c, x}^{\xi^\star} = \tilde{G}_{u, x+\ek_2} - \tilde{G}_{u, x}.
		\]
		%In fact the passage times themselves are equal, not only their difference.
		 
		 By \cite[Lemma B.1]{Bal-Bus-Sep-20}, $\tilde{G}_{u, x+\ek_2} - \tilde{G}_{u, x} \geq G_{u, x+\ek_2} - G_{u, x}$.  Putting these together gives $J_{x+\ek_2}^{\xi^\star} \geq J_{x+\ek_2}^{[u]}$.  
		 
		 If instead $u = o_c + i\ek_1$ for some $i\geq 1$, then as before $G_{o_c, x}^{\xi^\star} = \tilde{G}_{o_c, x}$.  Furthermore, since $\tilde{G}_{u,x}$ does not use the vertical weights above $o_c$, then $\tilde{G}_{u,x} = G_{u,x}$.  By \cite[Lemma B.2]{Bal-Bus-Sep-20}, 
		 \begin{align*}
		 G_{o_c +\ek_{1}, x+\ek_{2}} - G_{o_c +\ek_{1}, x} \leq G_{o_c, x+\ek_{2}}-G_{o_c, x}.
		 \end{align*}
		 Inductively, this gives $G_{o_c, x+\ek_{2}}-G_{o_c, x} \geq G_{u, x+\ek_2} - G_{u, x}$ and applying \cite[Lemma B.1]{Bal-Bus-Sep-20} for the first inequality we get
		 \[
		 G^{\xi^\star}_{o_c, x+\ek_2} - G^{\xi^\star}_{o_c, x} = \tilde{G}_{o_c, x+\ek_2} - \tilde{G}_{o_c, x} \geq  G_{o_c, x+\ek_{2}}-G_{o_c, x} \geq G_{u, x+\ek_2} - G_{u, x}.\qedhere
		 \]
	\end{proof}

	Now we do an analogous construction for the stationary process with a northeast boundary. 
	Recall the northeast boundary \eqref{nebdry}. We continue to drop the $\delta$ from the notation. Recall also the weights \eqref{IJhat}. 
	For $x=(x_1,x_2)$ and $y=(y_1,y_2)$ in $\bbZ^2$ 
set $\hat G^\xi_{x, y} = 0$ if  $x\not\le y$,
while if $x\le y$ then let
\begin{align}\label{GNE}
	\hat G^\xi_{y, x} =\max _{1 \leq k \leq y_1-x_1}\Bigl\{ G_{x, y-k\ek_1-\ek_2} +  \sum_{i=1}^{k} \hat I^\xi_{y-i \ek_1} +  \Bigr\} \bigvee \max_{1 \leq \ell \leq y_2-x_2} \Bigl\{ G_{x,y-\ek_1-\ell \ek_2} + \sum_{j=1}^{\ell} \hat J^\xi_{y-j \ek_2}\Bigr\},
\end{align}
with the convention that $\max_{\varnothing}=0$. In particular, $\hat G^\xi_{x, x} = 0$.
Then
    \[\hat G^\xi_{y,x}(\w)=G^\xi_{-x,-y}(\hat\w).\]
The additivity \eqref{Gadd} becomes
    \begin{align}\label{Ghatadd}
    \hat G^\xi_{z,y}+\hat G^\xi_{y,x}=\hat G^\xi_{z,x},
    \end{align}
for $x\le y\le z$ in $\ZZ^2$.
The quantities $\hat\Exit_{y,x}^\xi$ and $\hat Z_{y,x}^{\xi,-\ek_k}$ are defined analogously to $\Exit_{y,x}^\xi$ and $Z_{y,x}^{\xi,\ek_k}$. Precisely, 
	\begin{align*}
	\hat\Exit_{y,x}^\xi 
	&= \Bigl\{ k \in\flint{1,y_1-x_1} : \sum_{i=1}^k \hat{I}_{y-i\ek_1}^\xi + G_{x,y-k\ek_1-\ek_2} = \hat{G}_{y,x}^\xi \Bigr\}\\
	&\qquad\qquad\qquad\qquad\bigcup 
	\Bigl\{ -\ell : \ell \in\flint{1,y_2-x_2} \text{ and } \sum_{j=1}^{\ell} \hat{J}_{y-j\ek_2}^\xi + G_{x,y-\ell \ek_2-\ek_1} = \hat{G}_{y,x}^\xi \Bigr\},
	\end{align*}
$\hat{Z}_{y,x}^{\xi, -\ek_1} = \max\hat\Exit_{y,x}^\xi$, and $\hat{Z}_{y,x}^{\xi, -\ek_2} = \min\hat\Exit_{y,x}^\xi$.%\medskip

	For $\hat{o} \in \hat{\partial}^{N,\delta}$ let
	\begin{align*}
	\widehat{\mathcal{I}}_{\widehat{o}, d}=\Bigl\{v \in \widehat{\partial}^{N,\delta}:|v-\widehat{o}|_{1} \leq \frac{d_{2} N^{2 / 3}-1}2\Bigr\}
	\end{align*}
	and let $\hat{o}_c$ be the unique point of $\hat{\mc{I}}_{\hat{o},d}$ such that $\hat{o}_c \geq v$ for each point $v \in \hat{\mc{I}}_{\hat{o},d}$.   	
	For $1 \leq s \leq \frac{\delta}{4}N^{1/3}$ define 
	\begin{align}\label{etastars}
    \eta_{\star}=\ximn_{\star}\left(\widehat{o}_{c}\right)\quad\text{and}\quad \eta^{\star}=\ximn^{\star}\left(\widehat{o}_{c}\right)
    \end{align}
    as in \eqref{defXimn}.
	Couple the weights $\{\w_x,I^{\eta^\star}_x,J^{\eta^\star}_x,I^{\eta_\star}_x,J^{\eta_\star}_x\}$
	using Theorem \ref{thm:coupledGenericLPP}. This produces a coupling of
	$\{\w_x,\hat I^{\eta^\star}_x,\hat J^{\eta^\star}_x,\hat I^{\eta_\star}_x,\hat J^{\eta_\star}_x\}$
    such that
	\begin{align*}
	\w_x \leq \hat{I}_x^{\eta^{\star}} \leq \hat{I}_x^{\eta{\star}} \quad\text{and}\quad \w_x \leq \hat{J}_x^{\eta_\star} \leq \hat{J}_x^{\eta^{\star}}\,,
	\end{align*}
	the analogue of \eqref{Lem4-4-monotonicity}, holds almost surely and for all $x\in\ZZ^2$.
	
	Define the increment variables analogously to \eqref{G-IJ}:
	\begin{align*}\label{Ghat-IJ}
	\begin{split}
	&\hat I_x^{[y]}=G_{x, y}-G_{x+\ek_1,y},\quad\text{when $x+\ek_1\le y$, and}\\ 
	&\hat J_x^{[y]}=G_{x,y}-G_{x+\ek_2, y},\quad\text{when $x+\ek_2\le y$}.
	\end{split}
	\end{align*}

%	Define increment variables on the vertical edges $\{(x+\ek_1, x+\ek_1+\ek_2) : x \in \mc{I}\}$ shifted by $\ek_1$ from $\mc{I}$.  For $v \in \hat{\mc{I}}_{\hat{o},d}, j \in I$, and $\xi \in \{\xi_\star, \xi^\star \}$, let
%	\[
%	\widehat{J}_{j}^{v}=\widehat{G}_{v, \ek_{1}+(j-1) \ek_{2}}-\widehat{G}_{v, \ek_{1}+j \ek_{2}} \text{ and } \widehat{J}_{j}^{\xi}=\widehat{G}_{\widehat{o}_{c}, \ek_{1}+(j-1) \ek_{2}}^{\xi}-\widehat{G}_{\widehat{o}_{c}, \ek_{1}+j \ek_{2}}^{\xi}
%	\]
	Define the event
	\begin{align}
	B_{\widehat{o}, d}=\left\{\widehat{Z}_{\widehat{o}_{c}, \lfloor N^{2 / 3} \rfloor \ek_{2}+\ek_{1}}^{\eta^{\star}, -\ek_1}<-d_{2} N^{2/3}, \quad \widehat{Z}_{\widehat{o}_{c},-\lfloor N^{2 / 3} \rfloor \ek_2 + \ek_1}^{\eta_{\star}, -\ek_2}>d_{2} N^{2/3}\right\}
	\end{align}
	The next result follows from Lemma \ref{lem:AodInequalities}.
	
	\begin{lemma}\label{lem:BodInequalities}
	Suppose Assumption \ref{exitPointTailBound} holds.
	Then for any $\delta\in(0,\delta_0)$,
	$N \geq N_5(\delta,\delta_0)$, $\hat{o} \in \hat{\partial}^{N,\delta}$, and $(d_2, s)$ satisfying \eqref{dsCondition},
		\begin{equation}
		\P(B_{\hat{o},d}^c) \leq 2C_5(\delta,\delta_0,\fluctexp) s^{-\fluctexp}. \label{boundBodComp}
		\end{equation}
		
		On the event $B_{\hat{o},d}$, the following inequalities hold for all $x \in \mc I$ and $v \in \hat{\mc{I}}_{\hat{o},d}$:
		\begin{equation}
		\hat{J}_{x+\ek_1+\ek_2}^{\eta_\star} \leq \hat{J}_{x+\ek_1+\ek_2}^{[v]} \leq \hat{J}_{x+\ek_1+\ek_2}^{\eta^\star}. \label{monotonicityBetweenRevLPP}
		\end{equation}
	\end{lemma}
	
	Let $o \in \partial^{N,\delta}$, $\hat{o} \in \hat{\partial}^{N,\delta}$ and consider the LPP process from points $u \in \mc{I}_{o,d}$ to the interval $\mc{I}$ and the reverse LPP process from points $v \in \hat{\mc{I}}_{\hat{o}, d}$ to the shifted interval $\ek_1 + \mc{I}$.  Recall
	\eqref{xistars} and \eqref{etastars}.
	Use Theorem \ref{thm:coupledGenericLPP} again to couple the weights in \eqref{bigcoupling} and thus produce a coupling of the weights \[\{\w_x,J^{\xi^\star}_x,J^{\xi_\star}_x,\hat J^{\eta^\star}_x,\hat J^{\eta_\star}_x:
	x\in\ZZ^2\}.\]
	Recall the random walks $S^{\xi_\star,\eta^\star}$ and $S^{\xi^\star,\eta_\star}$, as defined in \eqref{Sdef}. Define also
	\begin{align*}%\label{Suvdef}
	S_n^{u,v}=\begin{cases}
	\sum_{j=1}^n(J_{j\ek_2}^{[u]}-\widehat{J}_{\ek_1+(j-1)\ek_2}^{[v]}),&n\ge1,\\[4pt]
	0,&n=0,\\[3pt]
	-\sum_{j=n+1}^0(J_{j\ek_2}^{[u]}-\widehat{J}_{\ek_1+(j-1)\ek_2}^{[v]}),&n\le-1.
	\end{cases}
	\end{align*}
%	Define three random walks indexed by the edges $\{((0, j),(1, j)): j \in I\}$ that run along the $y$-axis.  The steps are defined
%	\begin{align*}
%	X_{j}^{u, v}=J_{j}^{u}-\widehat{J}_{j}^{v}, \quad Y_{j}^{\prime}=J_{j}^{\xi^{\star}}-\widehat{J}_{j}^{\eta_{\star}}, \quad \text { and } \quad Y_{j}=J_{j}^{\xi_{\star}}-\widehat{J}_{j}^{\eta^\star},
%	\end{align*}
%	and the corresponding walks are denoted by
%	\begin{align}\label{RWs}
%	S^{u, v}=S(X^{u, v}), \quad S^{\prime}=S(Y^{\prime}), \quad \text { and } \quad S=S(Y).
%	\end{align}
	
The following is immediate from  \eqref{monotonicityBetweenLPP} and \eqref{monotonicityBetweenRevLPP}. 

	\begin{lemma} \label{lem:RWs} 
		On the event $A_{o,d} \cap B_{\hat{o}, d}$, for all $u \in \mc{I}_{o,d}$ and $v \in \hat{\mc{I}}_{\hat{o}, d}$,
		\begin{align}
		\begin{aligned}
		&S^{\xi_\star,\eta^\star}_{n} \leq S_{n}^{u, v} \leq S^{\xi^\star,\eta_\star}_{n} \text{ for } n \in \flint{ 0, N^{2 / 3} } \quad\text{and}\\
		&S^{\xi^\star,\eta_\star}_{n}  \leq S_{n}^{u, v} \leq S^{\xi_\star,\eta^\star}_{n} \text{ for } n \in  \flint{-N^{2 / 3}+1,0 }.
		\end{aligned}\label{RWMonotonicity}
		\end{align}
	\end{lemma}

	Recall the event $U^{u,v}$ defined in \eqref{Uuv}.

	\begin{lemma}\label{lem:boundProbUseEdge}
		Suppose Assumptions \ref{exitPointTailBound} and \ref{ass:RW} hold with \eqref{a-nu} satisfied. 
		For any $\delta\in(0,\delta_0)$
	    there exist finite positive constants $C_6(\delta, \delta_0, \fluctexp,\expnt)$ and  $N_6(\delta, \delta_0)\ge8\delta^{-3}$ such that for all $N \geq N_6$ and $o \in \partial^{N,\delta}$, if $\hat{o}=-o \in \hat{\partial}^{N,\delta}$, 
	    %$s = N^{\frac{4}{3(6+\fluctexp)}}$,  
	    $d_1=1$, %$d_2=N^{\frac{2}{3(2+\fluctexp)}})$
	    $d_2=N^{\frac13-\frac{\expnt}2}/18$, and 
	    $s=8d_2$, then
		\begin{equation}
		\P\Bigl( \bigcup_{u \in \mc{I}_{o,d}, v \in \hat{\mc{I}}_{\hat{o},d}} U_{u,v} \Bigr) \leq C_6 N^{-\expntt},
		\end{equation}
		where $\expntt$ is defined in \eqref{a:def}.
\end{lemma}

	\begin{proof}
		Let $o \in \partial^{N,\delta}$, $\hat o=-o$, $u \in \mc{I}_{o,d}$, and $v \in \mc{\hat{I}}_{\hat{o}, d}$.  
		The walk $S^{u,v}$ determines where the geodesics of $G_{u,v}$ leave the vertical axis,  since
		\begin{align*}
			G_{u, v} &=\max_{u_{2} \leq n \leq v_{2}}\left\{G_{u,(0, n)}+\widehat{G}_{v,(1, n)}\right\} \\
			&=\max_{u_{2} \leq n \leq v_{2}}\left\{\left[G_{u,(0, n)}-G_{u,(0,0)}\right] + G_{u,(0,0)} +\widehat{G}_{v, (1,0)}-\left[\widehat{G}_{v,(1,0)}-\widehat{G}_{v,(1, n)}\right]\right\} \\
			&=\max_{u_{2} \leq n \leq v_{2}}\left\{G_{u,(0,0)}+\widehat{G}_{v,(1,0)}+S_{n}^{u, v}\right\}.
		\end{align*}
		Therefore, a geodesic of $G_{u,v}$ takes the edge $(j\ek_2,\ek_1+j\ek_2)$ if and only if $j\in\flint{u_2,v_2}$ is such that $S^{u,v}_j=\max_{u_2 \leq n \leq v_2} S_n^{u,v}$.  Consequently,
		\begin{align*}
			U^{u,v} &\subset \Bigl\{\sup_{0 < k \leq N^{2/3}}  S_k^{u, v} \leq 0  \Bigr\} \cap \Bigl\{\sup_{-N^{2/3}+1 \leq k < 0}  S_k^{u, v} \leq 0 \Bigr\}.
		\end{align*}
	This and \eqref{RWMonotonicity} imply that on the event $A_{o, d} \cap B_{\hat{o}, d}$ 
		\begin{align*}
			\bigcup_{u \in \mc{I}_{o,d}, v \in \hat{\mc{I}}_{\hat{o},d}} U^{u,v} %&\subset \bigcup_{u \in \mc{I}_{o,d}, v \in \hat{\mc{I}}_{\hat{o},d}} \left\{\sup_{0 < k \leq N^{2/3}}  S_k^{u, v} \leq 0  \right\} \cap \left\{\sup_{-N^{2/3} \leq k < 0}  S_k^{u, v} \leq 0 \right\} \\
			%&\subset  \left( \bigcup_{u \in \mc{I}_{o,d}, v \in \hat{\mc{I}}_{\hat{o},d}} \left\{\sup_{0 < k \leq N^{2/3}}  S_k^{u, v} \leq 0  \right\} \right) \cap \left( \bigcup_{u \in \mc{I}_{o,d}, v \in \hat{\mc{I}}_{\hat{o},d}} \left\{\sup_{-N^{2/3} \leq k < 0}  S_k^{u, v} \leq 0 \right\}  \right) \\
			&\subset \Bigl\{ \sup_{0<k\leq N^{2/3}} S^{\xi_\star,\eta^\star}_k \leq 0 \Bigr\} \cap \Bigl\{ \sup_{-N^{2/3}+1 \leq k < 0} S^{\xi^\star,\eta_\star}_k \leq 0 \Bigr\},
		\end{align*}
		where $\xi^\star,\xi_\star,\eta^\star,\eta_\star$ were defined in \eqref{xistars} and \eqref{etastars}. 
		%and we take %$s=N^{\frac{4}{3(6+\fluctexp)}}$ 
		%$s=...$
		%in \eqref{defXimn} (which is used in these definitions). 
		As a result, we have
		\begin{align}
		\P\Bigl( \bigcup_{u \in \mc{I}_{o,d}, v \in \hat{\mc{I}}_{\hat{o},d}} U^{u,v} \Bigr) 
		\leq  \P\Bigl( \Bigl\{ \sup_{0<k\leq N^{2/3}} S^{\xi_\star,\eta^\star}_k \leq 0 \Bigr\} \cap \Bigl\{ \sup_{-N^{2/3}+1 \leq k < 0} S^{\xi^\star,\eta_\star}_k \leq 0 \Bigr\} \Bigr) + \P(A_{o,d}^c \cup B_{\hat{o},d}^c).\label{boundInTermsOfRWs}
		\end{align}
		
		Take $N\ge N_3(\delta)\vee N_5(\delta,\delta_0)$ and such that $s\ge s_5(\delta,\delta_0)$, $d\le \delta N^{1/3}/64$, and hence $s\le\delta N^{1/3}/4$.  
		Since $o = -\hat{o}$ 
		\[
		|o_c + \hat{o}_c|_1 \leq \abs{o_c-o}_1+\abs{\hat o_c-\hat o}_1\le
		(d_1 N^{2/3} + d_2 N^{2/3})/2 \leq d_2 N^{2/3}
		%=N^{\frac{2}{3} + \frac{2}{3(2+\fluctexp)}}
		\]
        and thus
		\[
		\abs{\zeta(o_c) - \zeta(\hat o_c)}_1
		\le \Bigl|\frac{o_c}{\abs{o_c}_1}+\frac{\hat o_c}{\abs{\hat o_c}_1}\Bigr|_1
		\le 
		%\frac{\abs{o_c+\hat o_c}_1}{\abs{o_c}_1}+\Bigl|\frac{\hat o_c}{\abs{\hat o_c}_1}-\frac{\hat o_c}{\abs{o_c}_1}\Bigr|_1
		%=
		\frac{\abs{o_c+\hat o_c}_1}{\abs{o_c}_1}+\frac{\bigl| \abs{o_c}_1-\abs{\hat o_c}_1\bigr|}{\abs{o_c}_1}
		\le
		\frac{2\abs{o_c+\hat o_c}_1}{\abs{o_c}_1}
		%\le 2 N^{-\frac{1}{3} + \frac{2}{3(2+\fluctexp)}}.
		\le2d_2 N^{-1/3}.
		\]
		Therefore
			\[-N^{-\expnt/2} =
			%-18d_2N^{-1/3}= 
			-2d_2N^{-1/3}-2sN^{-1/3}\leq \xi_{\star}\cdot \ek_1 - \eta^\star\cdot \ek_1 \le 2 d_2N^{-1/3} -2sN^{-1/3}< 0\]
		and similarly
			\[-N^{-\expnt/2}
			\leq \eta_{\star}\cdot \ek_1 - \xi^\star\cdot \ek_1<0.\]
		%\mathnote{for the above we need $a(\fluctexp)=2\fluctexp/3(2+\fluctexp)$}
		Furthermore, the inequalities in \eqref{xistar in cone} verify that the $\ek_1$-coordinates of  $\xi^\star,\xi_\star,\eta^\star,\eta_\star$ are all in $(\delta/4,1-\delta/4)$. 
		We can now apply \eqref{bd:RW}, 
		\eqref{boundAodComp}, and \eqref{boundBodComp}, which together with \eqref{boundInTermsOfRWs} 
		give
		\[
		\P\Bigl( \bigcup_{u \in \mc{I}_{o,d}, v \in \hat{\mc{I}}_{\hat{o},d}} U^{u,v} \Bigr) \leq C_3(\delta) N^{-\expnt} + 4C_5(\delta,\delta_0,
		\fluctexp)s^{-\fluctexp} \leq C_6N^{-\expntt}.\qedhere\]
		%where \[\expntt=\min(\expnt,(\frac13-\frac\expnt2)\fluctexp.\]
		%Here we used that $s^{-\fluctexp} = N^{\frac{4}{3(6+\fluctexp)}(-\fluctexp)}$ and
		%\begin{align*}
		%    -a(\fluctexp) = -\frac{2(2+\fluctexp)}{3(6+\fluctexp)} \geq -\frac{4\fluctexp}{3(6+\fluctexp)} = \frac{4}{3(6+\fluctexp)}(-\fluctexp).
		%\end{align*}
	\end{proof}

	Just as above, for $o \in \partial^{N,\delta}$, let $\hat{o} = -o$ and set
	\begin{align*}
		\hat{\mathcal{F}}_{\hat{o}, d}=\Bigl\{v \in \widehat{\partial}^{N,\delta}:|\widehat{o}-v|_{1}>\frac{d_{2} N^{2/3}-1}2\Bigr\}.
	\end{align*}
	
	\begin{lemma}\label{lem:boundGeodesicsToOuter}  
	Suppose Assumption \ref{exitPointTailBound} holds. For any 
	$a\in(0,2/3)$ and $\delta\in(0,\delta_0)$, 
	there exist positive finite constants $C_7(\delta, \delta_0, \fluctexp)$ and $N_7(\delta, \delta_0,a)\ge8\delta^{-3}$ such that for any $N \geq N_7$ and $o \in \partial^{N,\delta}$,
	if $d_1=1$ and $d_2= N^{\frac13-\frac{a}2}/18$, then
	\begin{align}
		\P\Bigl(\bigcup_{u \in \mathcal{I}_{o, d}, v \in \widehat{\mathcal{F}}_{\widehat{o}, d}} U^{u, v}\Bigr) \leq C_7 N^{-(\frac13-\frac{a}2)\nu}. \label{lemma57event}
	\end{align}
	\end{lemma}
	
	\begin{proof} Define the boundaries
	\begin{align*}
		\begin{aligned}
			&\partial \widehat{\mathcal{F}}_{\widehat{o}, d}=\left\{v \in \widehat{\mathcal{F}}_{\widehat{o}, d}: \exists u \in \widehat{\mathcal{I}}_{\widehat{o}, d} \text { such that } \abs{v-u}_1=1\right\}\quad\text{and}\\
			&\partial \mathcal{I}_{o, d}=\left\{v \in \mathcal{I}_{o, d}: \exists u \in \partial^{N,\delta} \backslash \mathcal{I}_{o, d} \text { such that }\abs{v-u}_1=1\right\}.
		\end{aligned}
	\end{align*}
	Their cardinalities are either 1 or 2, since it may happen that  $\hat{\mc{I}}_{\hat{o},d}$ contains an endpoint such as $(N, \fl{ \delta N })$.  Additionally, $1 \leq |\partial \hat{\mc{F}}_{\hat{o},d}| \leq |\partial \mc{I}_{o,d} | \leq 2$ because $d_1 < d_2$, so $\hat{\mc{I}}_{\hat{o},d}$ would include an endpoint of the boundary whenever $\mc{I}_{o,d}$ does.  
	Label the points in $\partial \mc{I}_{o,d}$ as $h^1$ and $h^2$ and label those of $\partial \widehat{\mathcal{F}}_{\widehat{o}, d}$ as $f^1$ and $f^2$ so that
	\begin{align*}
		\begin{aligned}
			&h_{1}^{1} \geq o_{1} \geq h_{1}^{2}, \quad h_{2}^{1} \leq o_{2} \leq h_{2}^{2},\quad
			f_{1}^{1} \leq \widehat{o}_{1} \leq f_{1}^{2}, 
			\quad\text{and}\quad 
			f_{2}^{1} \geq \widehat{o}_{2} \geq f_{2}^{2}.
		\end{aligned}
	\end{align*}
Traveling clockwise around the boundary of the square $\flint{-N,N }^2$ starting at $(0,N)$, the points that exist come in this order:  $f^1, \hat{o}, f^2, h^1, o, h^2$.  

We will show that if some geodesic from 
$u \in \mathcal{I}_{o, d}$ to $v \in \widehat{\mathcal{F}}_{\widehat{o}, d}$
 uses the edge $(\zero,\ek_1)$ then, for some $i \in \{1,2\}$, $\pi^{u,v,\ek_i}$, the $\ek_i$-most geodesic of $G_{h^i,f^i}$, deviates by at least $\delta d_2 N^{2/3}/16$ from the straight line segment from $h^i$ to $f^i$. To this end, define 	
%	For points $u \in \partial^{N,\delta}, v \in \hat{\partial}^{N,\delta}$ and $i \in \{1, 2\}$ let
	\begin{align*}
		\mathcal{P}_{m}^{u, v, \ek_i}=\pi^{u, v, \ek_i} \cap\left\{x \in \mathbb{Z}^{2}: x_{1}=m\right\}.
	\end{align*}
	This is the intersection of the $\ek_i$-most geodesic of $G_{u,v}$ %$\pi^{u, v, \ek_i}$ 
	with the vertical line $x_1 = m$.  For $t>0$ let
	\begin{align}
		D_{m, t}^{u, v}= \bigcup_{i=1}^2 \Bigl\{\inf _{p=\left(p_{1}, p_{2}\right) \in \mathcal{P}_{m}^{u, v, \ek_i}}\Bigl|u_{2}+\frac{v_{2}-u_{2}}{v_{1}-u_{1}}\left(m-u_{1}\right)-p_{2}\Bigr|>t\Bigr\} \label{Ddef}
	\end{align}
	be the event that at the vertical line $x_1 = m$, some geodesic from $u$ to $v$ deviates from the straight line segment from $u$ to $v$ by more than $t$.

	For $u \in \mathcal{I}_{o, d}$ and $v \in \widehat{\mathcal{F}}_{\widehat{o}, d}$, let $e^u=u-o$ and $e^v = v-\hat{o}$. %Using  $v_{i}-u_{i}=\widehat{o}_{i}+e_{i}^{v}-\left(o_{i}+\ek_{i}^{u}\right)=-2 o_{i}+e_{i}^{v}-e_{i}^{u}$ write the quantity in \eqref{Ddef} as
	%\begin{align}
	%	\begin{aligned}
	%		u_{2}+\frac{v_{2}-u_{2}}{v_{1}-u_{1}}\left(m-u_{1}\right) 
	%		%&=o_{2}-\frac{v_{2}-u_{2}}{v_{1}-u_{1}} o_{1}+e_{2}^{u}+\frac{v_{2}-u_{2}}{v_{1}-u_{1}}\left(m-e_{1}^{u}\right) \\
	%		&=\frac{o_{2} e_{1}^{v}-o_{1} e_{2}^{v}}{v_{1}-u_{1}}-\frac{o_{2} e_{1}^{u}-o_{1} e_{2}^{u}}{v_{1}-u_{1}}+e_{2}^{u}+\frac{v_{2}-u_{2}}{v_{1}-u_{1}}\left(m-e_{1}^{u}\right).
	%	\end{aligned} \label{vertDistCalc}
	%\end{align}
	%%Here the first line uses that $u_2 = o_2 + e_2^u$ and $-u_1 = -o_1-e_1^u$.  The second line combines the first two terms with a common denominator.  Then it uses the calculation above for $v_1 - u_1$ and $v_2 - u_2$.  Lastly, organize the terms into two fractions. 
   Then
   $\hat{\mc{F}}_{\hat{o},d}$ is the union of two disjoint pieces
   \[
	\hat{\mathcal{F}}_{\hat{o}, d}^{1}=\left\{v \in \widehat{\mathcal{F}}_{\hat{o}, d}: e_{1}^{v} \leq 0 \leq e_{2}^{v}\right\} \quad \text { and } \quad \hat{\mathcal{F}}_{\hat{o}, d}^{2}=\left\{v \in \hat{\mathcal{F}}_{\hat{o}, d}: e_{2}^{v} \leq 0 \leq e_{1}^{v}\right\},
	\]
	separated by $\hat{\mc{I}}_{\hat{o},d}$, one of which can be empty.  $\widehat{\mathcal{F}}_{\widehat{o}, d}^{1}$ is to the left and above $\hat{\mc{I}}_{\hat{o},d}$, and if it is not empty, then it is separated from $\hat{\mc{I}}_{\hat{o},d}$ by the point $f^1$. $\widehat{\mathcal{F}}_{\widehat{o}, d}^{2}$ is to the right and below $\hat{\mc{I}}_{\hat{o},d}$, and if it is not empty, then it is separated from $\hat{\mc{I}}_{\hat{o},d}$ by the point $f^2$. 

    Take 
    $N\ge\max\bigl(\sqrt8,(1+\delta)/\delta^2,N_2(\delta,\delta_0)/(2\delta)\bigr)$ and large enough so that \[\frac{2d_1}{\delta}\le\frac{\delta d_2}{16},\quad 
    d_2\ge16s_2(\delta,\delta_0),\quad\text{and}\quad
    \frac{d_2}{16}\le\frac\delta{4(1+\delta)}\cdot\frac{\delta (4N)^{1/3}}{3(1+\delta)}\,.\]
    
	If $u \in \mc{I}_{o,d}$ and $v \in \hat{\mc{F}}^1_{\hat{o},d}$, then 
		\begin{align*}
		\left|e^{u}\right|_{1} \leq \frac{d_{1} N^{2/3}-1}2\,, \quad\left|e^{v}\right|_{1}>\frac{d_{2} N^{2/3}-1}2\ge\frac{d_{2} N^{2/3}}4\,,
		\quad\text{and}\quad e_1^v\le 0\le e_2^v.
		%\quad\left|e_{1}^{v}\right| \vee\left|e_{2}^{v}\right| \leq 2(1-\delta) N\,, \quad \text{and} \quad 
		%e_{1}^{v} e_{2}^{v} \leq 0. %\label{vectorProps}
	\end{align*}
	%The first two are because $u$ is within $\mc{I}_{o,d}$ and $v$ is outside of $\hat{\mc{I}}_{\hat{o},d}$.  
	%The third one is because either $v$ and $\hat o$ are on the same segment of $\hat\partial^{N,\delta}$, in which case their distance is bounded by the length of that side, i.e.\ by $(1-\delta)N$, or they are on different sides and then their distance is bounded by $2(1-\delta)N$.
	%The fourth one is because $e^v$ either has a 0 entry in one coordinate, or if not, it is up-left or down-right.
   Using this, together with  $v_{i}-u_{i}=\widehat{o}_{i}+e_{i}^{v}-\left(o_{i}+e_{i}^{u}\right)=-2 o_{i}+e_{i}^{v}-e_{i}^{u}$, $-N\le o_i\le-\delta N$, and $\delta\le(v_2-u_2)/(v_1-u_1)\le1/\delta$, we get
    %The $e^u$ terms in \eqref{vertDistCalc} are collected together into a single error term.  
	\begin{align}
			u_{2}+\frac{v_{2}-u_{2}}{v_{1}-u_{1}}\left(-u_{1}\right) & = \frac{o_{2} e_{1}^{v}-o_{1} e_{2}^{v}}{v_{1}-u_{1}}-\frac{o_{2} e_{1}^{u}-o_{1} e_{2}^{u}}{v_{1}-u_{1}}+e_{2}^{u}+\frac{v_{2}-u_{2}}{v_{1}-u_{1}}\left(-e_{1}^{u}\right)\notag \\
			& \geq \frac{\delta N\left|e^{v}\right|_{1}}{2 N}-\Bigl(\frac{N}{2 N \delta}+1+\delta^{-1}\Bigr)\left|e^{u}\right|_{1} \notag\\
			& \geq \frac{1}{8} \delta d_{2} N^{2/3}-2 \delta^{-1} d_{1} N^{2/3} \geq \frac{1}{16} \delta d_{2} N^{2/3}.
		    \label{verticalBound1}
	\end{align}
	%The first line is the earlier calculation \eqref{vertDistCalc} with $m = 0$.  The second line, first term, we want to minimize this term since it's positive (since $v \in \hat{\mc{F}}_{\hat{o},d}^1$).  The denominator is maximized by $2N$.  The numerator is minimized by bounding $|o_1|, |o_2|$ by $\delta N$, and then what's left in the numerator is the $L^1$ norm of $e^v$.  The last terms will all be combined.  For the second term, we want to maximize since it's negative.  The numerator is maximized similarly to the first term by bounding the $o$ parts by $N$ and then we have the $L^1$ norm of $e^u$.  The denominator is minimized by $2N \delta$.  The third term, we assume $e_2^u$ is negative to get a lower bound, and maximize it as the $L^1$ norm of $e^u$.  Assume the last term is negative for the lower bound.  Maximize the slope factor with $\delta^{-1}$ and the rest is bounded by the $L^1$ norm of $e^u$.

%Since $e^v_1e^v_2\le0$ we have
%	    \[\frac{d_2 N^{2/3}}{4}\le\frac{o_{2} e_{1}^{v}-o_{1} e_{2}^{v}}{v_{1}-u_{1}}\le \frac{d_2 N^{2/3}}{2\delta}\]
%if $v \in \hat{\mc{F}}^1_{\hat{o},d}$ and 
%	    \[-\frac{d_2 N^{2/3}}{2\delta}\le\frac{o_{2} e_{1}^{v}-o_{1} e_{2}^{v}}{v_{1}-u_{1}}\le -\frac{d_2 N^{2/3}}{4}\]
%if $v \in \hat{\mc{F}}^2_{\hat{o},d}$.  Since $d_2$ is much larger than $d_1$, this term dominates the others on the right-hand side of \eqref{vertDistCalc}.

	Similarly, for $u \in \mc{I}_{o,d}$ and $v \in \hat{\mc{F}}_{\hat{o},d}^2$, we have
	\begin{align}
			u_{2}+\frac{v_{2}-u_{2}}{v_{1}-u_{1}}\left(1-u_{1}\right) 
			&=\frac{o_{2} e_{1}^{v}-o_{1} e_{2}^{v}}{v_{1}-u_{1}}-\frac{o_{2} e_{1}^{u}-o_{1} e_{2}^{u}}{v_{1}-u_{1}}+e_{2}^{u}+\frac{v_{2}-u_{2}}{v_{1}-u_{1}}\left(1-e_{1}^{u}\right)\notag\\
			&\leq-\frac{\delta N\left|e^{v}\right|_{1}}{2 N}+\left(\frac{N}{2 N \delta}+1+\delta^{-1}\right)\left|e^{u}\right|_{1}+\delta^{-1} \notag\\
			%use e_2^v\le0\le e_1^v and the bounds on o_i and slope
			& \leq-\frac{1}{8} \delta d_{2} N^{2/3}+2 \delta^{-1} d_{1} N^{2/3} 
			%use 1\le d_1N^{2/3}, 1\le\delta^{-1}, and 1/2\le1
			\leq-\frac{1}{16} \delta d_{2} N^{2/3}.
		    \label{verticalBound2}
	\end{align}

%COMMENT REGARDING THE CARDINALITY of the two sets constructed at the start of the proof:
%The cardinalities are both either 1 or 2 for every n. If the cardinality is 1, it just means we are considering an interval that is at the edge of the boundary.  For the Lemma 5.11, we consider the case where v is outside the interval around -o. If the cardinalities are 1, this means there is only one side of the interval for v to be on, and therefore only one case to consider for whether v is in \hat{F}^1 or 2. In the two cases at the bottom of page 19, only one would be needed since v is guaranteed to be in one of the two sets 
%
%In the case where the cardinalities are 2, v could be on either side of the interval and we need both cases. But now both pairs exist, and both cases at the bottom of page 19 apply

	Now suppose that for some $u \in \mc{I}_{o,d}$ and $v \in \hat{\mc{F}}_{\hat{o},d}$ some geodesic of $G_{u,v}$ goes through the edge $(\zero,\ek_1)$.  We have these two cases:
	
	(i)  If $v \in \hat{\mc{F}}_{\hat{o},d}^1$, then the rightmost geodesic $\pi^{h^1, f^1, \ek_1}$ stays to the right of all the geodesics from $u$ to $v$. %because both its endpoints are below and to the right of $u$ and $v$. 
	Consequently, this geodesic crosses the axis $\RR\ek_2$ at or below $\zero$.
	Then \eqref{verticalBound1} with $u = h^1$ and $v = f^1$ shows that %at $\ek_1$-coordinate $x_1=0$ the geodesic 
	$\pi^{h^1, f^1, \ek_1}$ avoids the vertical interval of radius $\frac{1}{16} \delta d_2 N^{2/3}$, centered around the point on the line segment from $h^1$ to $f^1$ with $\ek_1$-coordinate $x_1=0$.  
	
	(ii) If $v \in \hat{\mc{F}}_{\hat{o},d}^2$, then the upmost geodesic $\pi^{h^2, f^2, \ek_2}$ stays above all the geodesics from $u$ to $v$ and therefore crosses $\RR\ek_2$ at or above $\zero$.  Then \eqref{verticalBound2} with $u = h^2$ and $v = f^2$ shows that %at $\ek_1$-coordinate $x_1=1$ the geodesic 
	$\pi^{h^2, f^2, \ek_2}$  avoids the vertical interval of radius $\frac{1}{16} \delta d_2 N^{2/3}$, centered around the point on the line segment from $h^2$ to $f^2$ with $\ek_1$-coordinate $x_1=0$.   
	
	We can now apply Lemma \ref{lem:geodesicsIntersectVertical} with $\e=\frac\delta{4(1+\delta)}$ because we took $N$ large enough so that $f^i-h^i\in S_\delta\cap\ZZ_{\ge N_2(\delta,\delta_0)}$, 
	\[s=\frac{d_2N^{2/3}}{16|f^i-h^i|_1^{2/3}}\in
	\Bigl[s_2(\delta,\delta_0), \frac{\e \delta}{3(\delta+1)}\abs{f_i-h_i}_1^{1/3} \Bigr]\]
	and 
	\[\alpha=\begin{cases}
	\frac{-h^1_1}{f^1_1-h^1_1}\in\bigl[\frac\delta{1+\delta},\frac1{1+\delta}\bigr]\subset(\e,1-\e)&\text{in case (i),}\\[6pt]
	\frac{1-h^2_1}{f^2_1-h^2_1}\in\bigl[\frac\delta{1+\delta},\frac1{1+\delta}+\frac1{2\delta N}\bigr]\subset(\e,1-\e)&\text{in case (ii).}
	\end{cases}\]
	%1/(2\delta N)\le(\delta/2)/(1+\delta)
	%and (1+\delta/2)/(1+\delta)=1-\delta/(2+2\delta)<1-\e

	Combining these results, we conclude that
	\begin{align*}
		\P\Bigl(\bigcup_{u \in \mathcal{I}_{o, d}, v \in \widehat{\mathcal{F}}_{\widehat{o}, d}} U^{u, v} \Bigr) &\leq \P\left(D_{0, \delta d_{2} N^{2 / 3}/16}^{h^{1}, f^{1}} \cup D_{1, \delta d_{2} N^{2 / 3}/16}^{h^{2}, f^{2}}\right) \\
		&\leq 
		2C_2(\delta,\delta_0,\fluctexp,\e)s^{-\fluctexp}
		\le C_2\bigl(\frac{d_2}{16\cdot2^{2/3}}\bigr)^{-\fluctexp}
		=C_7 N^{-(\frac13-\frac{a}2)\nu}.
	\end{align*}
	The lemma is proved.
	\end{proof}
	
Using a union bound and Lemmas \ref{lem:boundProbUseEdge} and \ref{lem:boundGeodesicsToOuter} we get the following.

\begin{lemma} \label{lem:boundOnUuv} 	
Suppose Assumptions \ref{exitPointTailBound} and \ref{ass:RW} hold with \eqref{a-nu} satisfied. 
For any $\delta\in(0,\delta_0)$ there exist positive finite constants $C_8(\delta, \delta_0, \fluctexp,\expnt)$ and  $N_8(\delta, \delta_0,\expnt)\ge8\delta^{-3}$ such that for all $N \geq N_8$ and $o \in \partial^{N,\delta}$, if 
$d_1=1$ and $d_2=N^{\frac13-\frac{\expnt}2}/18$, then
\begin{align*}
	\P \Bigl(\bigcup_{u \in \mathcal{I}_{o, d}, v \in \widehat{\partial}^{N,\delta}} U^{u, v}\Bigr) \leq C_8 N^{-\expntt}.
\end{align*}
where $\expntt$ is given in \eqref{a:def}.
\end{lemma}

We are now ready to prove Theorem \ref{thm:noBiInfinitesAwayFromAxes}. 
%Recall the event $W_{N, \delta}$ that there exist points $u \in \partial^{N, \delta}$ and $v \in \hat{\partial}^{N, \delta}$ such that at least one geodesic from $u$ to $v$ goes through the origin. % and the event $U^{u,v}$ is that at least one geodesic from $u$ to $v$ uses the edge $(\zero,\ek_1)$.   

\begin{proof}[Proof of Theorem \ref{thm:noBiInfinitesAwayFromAxes}] 
%It suffices to bound the probability that at least one geodesic uses the origin and exits in the $\ek_1$ direction.  
%We will prove there exist positive constants $C(\delta, \delta_0, \fluctexp), N_0(\delta, \delta_0, \fluctexp)$ such that for all $N \geq N_0$
%\begin{align}
%	\P\left(\quad \bigcup_{u \in \partial^{N,\delta}, v \in \widehat{\partial}^{N,\delta}} U^{u, v}\right) \leq C N^{-a(\fluctexp) + 1/3} = C N^{-b(\fluctexp)}.
%\end{align}
Take $d_1 = 1$ and $d_2=N^{\frac13-\frac{\expnt}2}/18$.  Let
\begin{align*}
	\mathcal{O}^{N}=\partial^{N,\delta} \cap\left(\left\{\left(-N+i d_{1}(\lfloor N^{2/3}\rfloor-1),- N\right)\right\}_{i \in \mathbb{Z} \geq 0} \bigcup\left\{\left(-N,-N+j d_{1}(\lfloor N^{2/3}\rfloor-1) \right)\right\}_{j \in \mathbb{Z}_{\geq 0}}\right).
\end{align*}
Then we can decompose
\begin{align*}
	\bigcup_{u \in \partial^{N,\delta}, v \in \widehat{\partial}^{N,\delta}} U^{u, v} \subset \bigcup_{o \in \mathcal{O}^{N}} \bigcup_{u \in \mc{I}_{o,d}, v \in \hat{\partial}^{N,\delta} } U^{u,v}.
\end{align*}

Since  $|\mc{O}^N| \leq C d_1^{-1} N^{1-2/3} = C N^{1/3}$, for some positive finite constant $C$, a union bound and Lemma \ref{lem:boundOnUuv} give
\begin{align*}
	\P\Bigl(\bigcup_{u \in \partial^{N,\delta}, v \in \hat{\partial}^{N,\delta}} U^{u, v}\Bigr) \leq \sum_{o \in \mathcal{O}^{N}} \P\Bigl(\bigcup_{u \in \mathcal{I}_{o, d}, v \in \hat{\partial}^{N,\delta}} U^{u, v}\Bigr) \leq C_4 N^{-(\expntt- 1/3)}.
\end{align*}
The theorem is proved.
\end{proof}

	\appendix
	\section{Stationary boundary}
    \label{app:Buse}
	\subsection{General weight distribution}\label{Busemann general}
The next theorem provides the boundary weights $I_x^\xi$ and $J_x^\xi$ that are used throughout our proofs. It follows directly from Theorem 4.7 of \cite{Jan-Ras-20-aop}.
Note that when the weights are geometric, random variables, Theorem \ref{thm:coupledGeomLPP} below gives an alternate construction of these boundary weights, with some additional independence properties. The purpose of the theorem in this section is to give a construction that works for a general weight distribution. If the reader is only interested in the geometric weights setting, then Theorem \ref{thm:coupledGenericLPP} can be bypassed and Theorem \ref{thm:coupledGeomLPP} can be used instead.

Recall the shape function $\gamma$ defined in \eqref{shape}. The subadditivity \eqref{Gsubadd} and the limit \eqref{shapeThm} imply that $\gamma$ is a convex positively homogeneous function on $\RR_{\ge0}^2$. As such, we can define the right-gradient $\gpp(\xi+)$ via the limits
\[\ek_1\cdot\nabla\gpp(\xi+)=\lim_{\e\searrow0}\frac{\gpp(\xi+\e\ek_1)-\gpp(\xi)}{\e}
\quad\text{and}\quad
\ek_2\cdot\nabla\gpp(\xi+)=\lim_{\e\searrow0}\frac{\gpp(\xi)-\gpp(\xi-\e\ek_2)}{\e}\,.\]

Let $\sU_0$ be a countable dense subset of $\ri\sU$. Let 
$\sH_0=\bigl\{-\nabla\gpp(\xi+):\xi\in\sU_0\bigr\}$.
Let $\hat\Omega=\Omega\times\RR^{\ZZ^2\times\{1,2\}\times\sH_0}$
and equip it with the product topology and the Borel $\sigma$-algebra $\hat\sG$. 
%Denote the coordinate projections of an element $\hat\w\in\hat\Omega$ by $\{\w_x,I^\xi_{y+e_1},J^\xi_{y+e_2}:x,y\in\ZZ^2,\xi\in\sH_0\}$, 
%where $\w=(\w_x)_{x\in\Z^2}$ is the projection of $\hat\w$ onto $\Omega$. 
Let $\hat T=(\hat T_x)_{x\in\ZZ^2}$ be the natural group of shifts on $\hat\Omega$. For $A\subset\ZZ^2$ let $A^{\le}=\{x\in\ZZ^2:\exists y\in A\text{ with }x\le y\}$
and $A^>=\ZZ^2\setminus A^{\le}$.

%	\cite{Geo-Ras-Sep-17-ptrf-1}
%	
%	\note{C: This part needs to be re-written completely. The coupling is to use the backward Busemann processes on the axes. We can just take the Cadlag version to simplify things, since we only need the 2D marginals of the coupling. We need to include the queueing connections and to mention Chang's uniqueness argument to appeal to those.} 
%Almost surely, we have for each $z \in \bbZ^2$,
%\begin{align}
%    (I_x^q, J_y^q : x,y \in \bbZ^2, q \in (r,1) ) \circ T_z &= (I_{x+z}^q, J_{y+z}^q : x,y \in \bbZ^2, q \in (r,1) ), \label{eq:cov-1}
%\end{align}
%In words, \eqref{eq:cov-1} says that these stochastic processes are jointly shift-covariant under the group of shifts $T$. Invariance of $\bbP$ under $T$ then implies shift-stationarity; since $\bbP$ is a product measure, ergodicity follows. This model is called increment stationary due to this observation and the identities $I_y^q = G_{x,y}^q-G_{x,y-\ek_1}^q,$ and $J_y^q = G_{x,y}^q-G_{x,y-\ek_2}^q$, which are valid whenever $x \leq y-\ek_1-\ek_2$. These last identities, along with properties of the joint distribution of these processes, are included in \textcolor{red}{Proposition [number]}.
%
%For general $\omega$ weights, base vertex $u \in \ZZ^2$, and directions $\xi, \eta \in \mc{U}$, we can define the coupled stationary LPP processes $G_{u, \bigcdot}^{\xi \pm}$ and $G_{u, \bigcdot}^{\eta \pm}$ using the Busemann functions $B^{\xi \pm}$ and $B^{\eta \pm}$ as the increments.  

\begin{theorem}\label{thm:coupledGenericLPP}
Assume \eqref{stand:assumption}.
There exist a $\hat T$-invariant probability measure $\hat\P$ on $(\hat\Omega,\hat\sG)$ and random variables $(x,\xi,\hat\w)\in\ZZ^2\times\ri\sU\times\hat\Omega\mapsto(I^\xi_x,J^\xi_x)\in\RR^2$ such that the following properties hold.
\begin{enumerate} [label=\rm(\alph{*}), ref=\rm\alph{*}] \itemsep=2pt 
%\item\label{coordinates} For $\xi\in\sH_0$ and $x\in\ZZ^2$, $I^\xi_x(\hat\w)$ and $J^\xi_x(\hat\w)$ are, respectively, the $(x,1,\xi)$-th and $(x,2,\xi)$-th coordinates of $\hat\w\in\hat\Omega$.
%
\item\label{BusTh.a} $\P$ is the restriction of $\hat\P$ onto $\Omega$. 

\item\label{BusTh.b} For any $A\subset\ZZ^2$, the process
$\{\w_x,I^\xi_x,J^\xi_x:x\in A,\xi\in\ri\sU\}$ is 
%measurable with respect to the $\sigma$-algebra generated by
%$\{\w_x,I^\xi_x,J^\xi_x:x\in A^\le,\xi\in\sH_0\}$ and 
independent of $\{\w_x:x\in A^>\}$.

\item\label{BusTh.c} For each $\xi\in\ri\sU$ and $x\in\ZZ^2$, $I_x^\xi$ and $J_x^\xi$ are integrable and $\hat\E[(I^\xi_x,J^\xi_x)]=\nabla\gpp(\xi+)$. 

\item\label{BusTh.d} There exists an event $\hat\Omega_0$ such that $\hat\P(\hat\Omega_0)=1$ and the following all hold for $\hat\w\in\hat\Omega_0$:

\begin{enumerate} [label=\rm(\ref{BusTh.d}.\arabic{*}), ref=\rm\ref{BusTh.d}.\arabic{*}] \itemsep=2pt 
\item\label{BusTh.d1} For each $x,y\in\ZZ^2$ and $\xi\in\ri\sU$, $I_x^\xi(\hat T_y\hat\w)=I_{x+y}^\xi(\hat\w)$ and $J_x^\xi(\hat T_y\hat\w)=J_{x+y}^\xi(\hat\w)$.  

\item\label{BusTh.d2} For each $x\in\ZZ^2$ and $\xi,\zeta\in\ri\sU$ with $\xi_1\le\zeta_1$ we have $\w_x= I_x^\xi\wedge J_x^\xi$,
	\[\w_x \leq I_x^{\zeta} \leq I_x^{\xi},\quad\text{and}\quad\w_x \leq J_x^{\xi} \leq J_x^{\zeta}.\] 

\item\label{BusTh.d3} For $\xi\in\ri\sU$, $x=(x_1,x_2)\in\ZZ^2$, and $k\in\NN$ set $G^\xi_{x,x}=0$,
\begin{align}\label{Gxi1}
G^{\xi}_{x,x+k\ek_1}=\sum_{i=1}^k I^\xi_{x+i\ek_1}
\quad\text{and}\quad
G^{\xi}_{x,x+k\ek_2}=\sum_{i=1}^k J^\xi_{x+i\ek_2}.
\end{align}
For $y\in x+\N^2$ let
\begin{align}\label{Gxi2}
	G^{^\xi}_{x, y} =\max _{1 \leq k \leq y_1-x_1}\Bigl\{ \sum_{i=1}^{k} I^\xi_{x+i \ek_1} + G_{x+k \ek_1 + \ek_2, y} \Bigr\} \bigvee \max_{1 \leq \ell \leq y_2-x_2} \Bigl\{ \sum_{j=1}^{\ell} J^\xi_{x+j \ek_2}+G_{x+\ek_1+\ell \ek_2, y}\Bigr\}.
\end{align}
Then for all $x\le y\le z$ in $\ZZ^2$ and $\xi\in\ri\sU$ we have
    \begin{align}\label{add}
    G_{x,y}^\xi+G_{y,z}^\xi=G_{x,z}^\xi.
    \end{align}
    In particular, for any $\xi\in\ri\sU$ and $x\in\ZZ^2$
    \begin{align}\label{cocIJ}
    I^\xi_{x+\ek_1}+J^\xi_{x+\ek_1+\ek_2}=J^\xi_{x+\ek_2}+I^\xi_{x+\ek_1+\ek_2}.
    \end{align}
\end{enumerate}

\item\label{BusTh.d'} 
	For each $u\ge v$ in $\ZZ^2_{\ge0}$ 
	\begin{align*}%\label{statG}
	\{ (G_{u, v+x}^\xi - G_{u, v}^\xi : x \in \ZZ_{\geq 0}^2,\xi\in\ri\sU \}  \eqd \{ (G_{u, u+x}^\xi : x \in \ZZ_{\geq 0}^2,\xi\in\ri\sU \} .
	\end{align*}
%	
%\item\label{BusTh.e} Assume that
%\begin{align}\label{unbded}
%\P\{\w_x>c\}>0\quad\text{for all }c.
%\end{align}
%Fix an integer $L\ge1$. Suppose %$\big\{\bar\omega_x,I^\ell_x,J^\ell_x:x\in\ZZ^2,\ell\in\flint{1,L}\big\}$ 
%are real-valued random variables such that
%\begin{enumerate}[label=\rm(\roman{*}), ref=\rm\roman{*}] \itemsep=2pt 
%\item $\big\{\bar\omega_x:x\in\ZZ^2\big\}$ has distribution $\P$.
%\item The distribution of %$\big\{\bar\omega_{x+z},I^\ell_{x+z},J^\ell_{x+z}:x\in\ZZ^2,\ell\in\flint{1,L}\big\}$ is the same for all $z\in\ZZ^2$.
%\item The process $\big\{\bar\w_x,I_x^\ell,J_x^\ell,x\le\zero,\ell\in\flint{1,L}\big\}$ is independent of $\{\bar\w_x:x\not\le\zero\}$. 
%\item\label{BusTh.ass.iv} There exists $i\in\{1,2\}$ such that 
%the process $\big\{\bar\omega_x,I^\ell_x,J^\ell_x:x\in\ZZ^2\big\}$ is ergodic under the natural shift of the coordinates by $\ek_i$.
%\item There exist $\xi_1,\dots,\xi_L\in\ri\sU$ such that for each $\ell\in\flint{1,L}$, $(I_x^\ell,J_x^\ell)$ is integrable and 
%has mean $\nabla\gpp(\xi_\ell+)$.
%\item For all $x\in\ZZ^2$ and $\ell\in\flint{1,L}$, we have almost surely \[\bar\omega_x=I^\ell_x\wedge J^\ell_x\quad\text{and}\quad I^\ell_{x+\ek_1}+J^\ell_{x+\ek_1+\ek_2}=J^\ell_{x+\ek_2}+I^\ell_{x+\ek_1+\ek_2}.\]
%\end{enumerate}
%Then the distribution of $\big\{\bar\omega_x,I^\ell_x,J^\ell_x:x\in\ZZ^2,\ell\in\flint{1,L}\big\}$ is the same as that of $\big\{\omega_x,I^{\xi_\ell}_x,J^{\xi_\ell}_x:x\in\ZZ^2,\ell\in\flint{1,L}\big\}$ under $\hat\P$.
\end{enumerate}
\end{theorem}

\begin{proof}
Taking $\beta=\infty$ in Theorem 4.7 of \cite{Jan-Ras-20-aop} we get 
a process 
$B^{\infty,h(\xi)+}(x,y,\hat\w)$, $x,y\in\ZZ^2$, $\hat\w\in\hat\Omega$, and $\xi\in\ri\sU$.
For $\hat\w\in\hat\Omega$ let $\overline\w\in\hat\Omega$ 
be such that $\overline\w_x=\hat\w_{-x}$, for all $x\in\ZZ^2$. Set $I_x^\xi(\hat\w)=B^{\infty,h(\xi)+}(-x,-x+\ek_1,\overline\w)$ and $J_x^\xi(\hat\w)=B^{\infty,h(\xi)+}(-x,-x+\ek_2,\overline\w)$.

%Property \eqref{coordinates} is a direct consequence of the construction of the process $B^{\infty,h(\xi)+}(x,y)$. 
%See middle of page 794.
Properties (\ref{BusTh.a}-\ref{BusTh.c}) follow from \cite[Theorem 4.7(a-c)]{Jan-Ras-20-aop}. \eqref{BusTh.d1} comes from \cite[(4.4)]{Jan-Ras-20-aop} and \eqref{BusTh.d2} comes from \cite[(4.7-4.8)]{Jan-Ras-20-aop}.

It is immediate from the cocycle property (4.4) and the recovery property (4.7) in \cite[Theorem 4.7]{Jan-Ras-20-aop} that for any $x\le y$ in $\ZZ^2$, we have $\hat\P$-almost surely, for any $\xi\in\ri\sU$, $G^\xi_{x,y}(\w)=B^{\infty,h(\xi)+}(-y,-x,\overline\w)$. Then the additivity  \eqref{add} (and \eqref{cocIJ}) is exactly the cocycle property \cite[(4.4)]{Jan-Ras-20-aop}. 

Next, note that \eqref{add} implies  $G^\xi_{u,v+x}-G^\xi_{u,v}=G^\xi_{v,v+x}$. Then property \eqref{BusTh.d'} follows from the $\hat T$-invariance of $\hat\P$ and the shift-covariance property in part \eqref{BusTh.d1}. 
\end{proof}

	\subsection{Geometric weights}\label{Busemann geometric}
	When the weights are geometric the process in Theorem \ref{thm:coupledGenericLPP} has some independence features and explicit one-dimensional marginals. Recall the bijection \eqref{p->xi}. %In this section, we will use the superscript $p\in(r,1)$ in the random variables constructed in Theorem \ref{thm:coupledGenericLPP} to refer to the quantities that correspond to the direction $\xip(p)$. For example, we use $I^p$ to mean $I^{\xip(p)}$.
%	
%		\begin{theorem}\label{thm:coupledGeomLPP}  Fix $0 < r < 1$ and let the bulk weights $\{\w_x:x\in\ZZ^2\}$ be i.i.d.\ $\Geom(r)$ random variables. 
		%%We will index the random variables constructed in 
%		Then \eqref{stand:assumption} is satisfied and for each $r < q_1 < q_2 < 1$ and $u\in\ZZ^2$ we have these properties:
%\begin{enumerate} [label=\rm(\alph{*}), ref=\rm\alph{*}] \itemsep=2pt 
%	
%	\item\label{coupledGeomLPPa} For each $v\in u+\ZZ^2_{\ge0}$,
%	\[
%	\{ (G_{u, v+x}^{\xip(q_1)} - G_{u, v}^{\xip(q_1)}, G_{u, v+x}^{\xip(q_2)} - G_{u, v}^{\xip(q_2)}) : x \in \ZZ_{\geq 0}^2 \}  \eqd \{ (G_{u, u+x}^{\xip(q_1)}, G_{u, u+x}^{\xip(q_2)}) : x \in \ZZ_{\geq 0}^2 \}.
%	\]
%	
%	\item\label{coupledGeomLPPb} The vertical increments $\{J_{u + j\ek_2}^{\xip(q_1)} : j \leq 0 \}$ and $\{J_{u + j\ek_2}^{\xip(q_2)} : j \geq 1 \}$ are mutually independent. Similarly, the horizontal increments $\{I_{u + i\ek_1}^{\xip(q_2)} : i \leq 0 \}$ and $\{I_{u + i\ek_1}^{\xip(q_1)} : i \geq 1 \}$ are mutually independent.
%	
%	\item\label{coupledGeomLPPc} For each $q \in \{q_1, q_2\}$, the increment variables $\{I_{u + i\ek_1}^{\xip(q)}, J_{u + j\ek_2}^{\xip(q)} : i, j \geq 1 \}$ are mutually independent.  Also the increment variables $\{I_{u - i\ek_1}^{\xip(q)}, J_{u - j\ek_2}^{\xip(q)} : i\ge0, j \ge0\}$ are mutually independent.
%	
%	\item\label{coupledGeomLPPd} For each $i \geq 1$, $j \geq 1$, and $q \in \{q_1, q_2\}$, the increments have marginal distributions: $I_{u + i\ek_1}^{\xip(q)} \sim \Geom(q)$ and $J_{u + j\ek_2}^{\xip(q)} \sim \Geom(r/q)$. % 
%	
%	\end{enumerate}
%	\end{theorem}  
%
Let $\overline\Omega=\Omega\times\RR^{\ZZ^2\times\{1,2\}}\times\RR^{\ZZ^2\times\{1,2\}}$
and equip it with the product topology and the Borel $\sigma$-algebra $\overline\sG$. Let $(\w(\overline\omega),I_x^1(\overline\omega),J_x^1(\overline\omega),I_x^2(\overline\omega),J_x^2(\overline\omega))$, $x\in\ZZ^2$, denote the coordinate projections of an element $\overline\omega\in\overline\Omega$.
Let $\overline T=(\overline T_x)_{x\in\ZZ^2}$ be the natural group of shifts on $\overline\Omega$. 

		\begin{theorem}\label{thm:coupledGeomLPP}  Fix $0 < r < 1$ and let the bulk weights $\{\w_x:x\in\ZZ^2\}$ be i.i.d.\ $\Geom(r)$ random variables. Then \eqref{stand:assumption} is satisfied and for 
		each $r < q_1 < q_2 < 1$ 
		there exist a $\overline T$-invariant probability measure $\overline\P_{q_1,q_2}$ on $(\overline\Omega,\overline\sG)$ such that the following properties hold.
		
\begin{enumerate} [label=\rm(\alph{*}), ref=\rm\alph{*}] \itemsep=2pt 
	\item\label{coupledGeomLPP0} The properties in Theorem \ref{thm:coupledGenericLPP}{\rm(}\ref{BusTh.a}-\ref{BusTh.d'}{\rm)} all hold:
	%, with $\hat\P$, $\hat\E$, $\hat\Omega_0$, $\hat\w$, $\hat T$, and $\xi,\zeta\in\ri\sU$ replaced, respectively, by $\overline\P_{q_1,q_2}$, $\overline\E_{q_1,q_2}$, $\overline\Omega_0$, $\overline\w$, $\overline T$, and $\ell,\ell'\in\{1,2\}$. In property \eqref{BusTh.c}, $\nabla\gpp(\xi+)$ is replaced by $\nabla\gpp(\xip(q_\ell))$.
\begin{enumerate} [label=\rm(\ref{coupledGeomLPP0}.\roman{*}), ref=\rm\ref{coupledGeomLPP0}.\roman{*}] \itemsep=2pt 
\item\label{coupledGeomLPP0.i} $\P$ is the restriction of $\overline\P_{q_1,q_2}$ onto $\Omega$. 

\item\label{coupledGeomLPP0.ii} For any $A\subset\ZZ^2$, the process
$\{\w_x,I^1_x,J^1_x,I^2_x,J^2_x:x\in A\}$ is 
%measurable with respect to the $\sigma$-algebra generated by
%$\{\w_x,I^\xi_x,J^\xi_x:x\in A^\le,\xi\in\sH_0\}$ and 
independent of $\{\w_x:x\in A^>\}$.

\item\label{coupledGeomLPP0.iii} For each $\ell\in\{1,2\}$ and $x\in\ZZ^2$, $I_x^\ell$ and $J_x^\ell$ are integrable and 
\begin{align}\label{Geomean}
\hat\E[(I^\ell_x,J^\ell_x)]=\nabla\gpp(\xip(q_\ell))=\Bigl(\frac{q_\ell}{1-q_\ell},\frac{r}{q_\ell-r}\Bigr).
\end{align}

\item\label{coupledGeomLPP0.iv} There exists an event $\overline\Omega_0$ such that $\overline\P_{q_1,q_2}(\overline\Omega_0)=1$ and the following all hold for $\overline\w\in\overline\Omega_0$:

\begin{enumerate} [label=\rm(\ref{coupledGeomLPP0.iv}.\arabic{*}), ref=\rm\ref{coupledGeomLPP0.iv}.\arabic{*}] \itemsep=2pt 
\item\label{coupledGeomLPP0.iv.1} For each $x,y\in\ZZ^2$ and $\ell\in\{1,2\}$, $I_x^\ell(\overline T_y\overline\w)=I_{x+y}^\ell(\overline\w)$ and $J_x^\ell(\overline T_y\overline\w)=J_{x+y}^\ell(\overline\w)$.  

\item\label{coupledGeomLPP0.iv.2} For each $x\in\ZZ^2$ and $\ell\in\{1,2\}$ we have $\w_x= I_x^\ell\wedge J_x^\ell$,
	\[\w_x \leq I_x^2 \leq I_x^1,\quad\text{and}\quad\w_x \leq J_x^1 \leq J_x^2.\] 

\item\label{coupledGeomLPP0.iv.3} For $\ell\in\{1,2\}$, if we define $G^\ell_{x,y}$ as in {\rm(}\ref{Gxi1}-\ref{Gxi2}{\rm)}, with $\xi$ replaced by $\ell$, then 
 for all $x\le y\le z$ in $\ZZ^2$ we have
    \begin{align}\label{add2}
    G_{x,y}^\ell+G_{y,z}^\ell=G_{x,z}^\ell.
    \end{align}
    In particular, for any $x\in\ZZ^2$
    \begin{align}\label{cocIJ2}
    I^\ell_{x+\ek_1}+J^\ell_{x+\ek_1+\ek_2}=J^\ell_{x+\ek_2}+I^\ell_{x+\ek_1+\ek_2}.
    \end{align}
\end{enumerate}

\item\label{coupledGeomLPP0.v} 
	For each $u\ge v$ in $\ZZ^2_{\ge0}$ 
	\begin{align*}%\label{statG}
	\bigl\{ (G_{u, v+x}^\ell - G_{u, v}^\ell : x \in \ZZ_{\geq 0}^2,\ell\in\{1,2\} \bigr\}  \eqd \bigl\{ (G_{u, u+x}^\ell : x \in \ZZ_{\geq 0}^2,\ell\in\{1,2\} \bigr\}.
	\end{align*}
\end{enumerate}
\end{enumerate}

\noindent In addition, we have the following independence properties.

\begin{enumerate} [resume,label=\rm(\alph{*}), ref=\rm\alph{*}] \itemsep=2pt 
	\item\label{coupledGeomLPPb} The vertical increments $\{J_{u + j\ek_2}^1 : j \leq 0 \}$ and $\{J_{u + j\ek_2}^2 : j \geq 1 \}$ are mutually independent. Similarly, the horizontal increments $\{I_{u + i\ek_1}^2 : i \leq 0 \}$ and $\{I_{u + i\ek_1}^1 : i \geq 1 \}$ are mutually independent.
	
	\item\label{coupledGeomLPPc} For each $\ell \in \{1, 2\}$, the increment variables $\{I_{u + i\ek_1}^\ell, J_{u + j\ek_2}^\ell : i, j \geq 1 \}$ are mutually independent.  Also the increment variables $\{I_{u - i\ek_1}^\ell, J_{u - j\ek_2}^\ell : i\ge0, j \ge0\}$ are mutually independent.
	
	\item\label{coupledGeomLPPd} For each $i \geq 1$, $j \geq 1$, and $\ell \in \{1, 2\}$, the increments have marginal distributions: $I_{u + i\ek_1}^\ell \sim \Geom(q_\ell)$ and $J_{u + j\ek_2}^\ell \sim \Geom(r/q_\ell)$.  
	
	\end{enumerate}
	\end{theorem}  

\begin{remark}
The edge weights needed to define the stationary boundary models we used in Sections \ref{sec:geo_fluct} and \ref{sec:nonexist} came from
the process produced by Theorem \ref{thm:coupledGenericLPP}. 
Theorem \ref{thm:coupledGeomLPP} can be used just the same to produce these edge weights, since by  Theorem \ref{thm:coupledGeomLPP}\eqref{coupledGeomLPP0}, 
the process $\{\omega_x,I^1_x,J^1_x,I^2_x,J^2_x:x\in\ZZ^2\}$, under  $\overline\P_{q_1,q_2}$,
satisfies all the properties of the process
$\{\omega_x,I^{\xip(q_1)}_x,J^{\xip(q_1)}_x,I^{\xip(q_2)}_x,J^{\xip(q_2)}_x:x\in\ZZ^2\}$, under $\hat\P$. 
%and we only need to work with two directions at a time
The advantage of using the process from Theorem \ref{thm:coupledGeomLPP} is that in the case of geometric weights, one has the additional independence properties in Theorem \ref{thm:coupledGeomLPP}{\rm(}\ref{coupledGeomLPPb}-\ref{coupledGeomLPPd}{\rm)}.
These properties are used to verify that Assumptions \ref{exitPointTailBound} and \ref{ass:RW} hold when the weights are geometric random variables. 
In the rest of this appendix (specifically, in Corollary \ref{lem:RWIndependent}, Theorem \ref{geometricExitPointTailBound}, and Lemma \ref{lm:assRWholds} below), although we continue using the notation from Theorem \ref{thm:coupledGenericLPP}, we mean to use the process from 
Theorem \ref{thm:coupledGeomLPP}. We also remark that the proof of Theorem 5.6 in \cite{Fan-Sep-20} implies that the two processes actually have the same distribution, but we do not need this fact.
%\note{C: I trimmed this a lot. We don't need to go into this much detail for a fact that we do not need in this paper. It's enough to say that the proof works in our case as well.}  
%When the weights $\omega_x$ are geometrically distributed, the process defined in Theorem \ref{thm:coupledGenericLPP} does satisfy all the 
%extra independence properties in Theorem \ref{thm:coupledGeomLPP}{\rm(}\ref{coupledGeomLPPb}-\ref{coupledGeomLPPd}{\rm)}. %One thing to note is that the theorem in \cite{Fan-Sep-20}  assumes  the weights are exponentially distributed. However, this assumption plays no role in their proof. What matters there is that the weights are unbounded, i.e.\ that $\P\{\w_0>c\}>0$ for all $c$.
%\mathnote{Would be nice to point out where this assumption is used in the proof of that theorem.}
\end{remark}
	
%	\begin{theorem}\label{thm:coupledGeomLPP}  Fix $0 < r < 1$ and let the bulk weights be i.i.d.\ $\Geom(r)$ random variables. For $r < q_1 < q_2 < 1$ and $u \in \ZZ^2$, there exists a coupling of the boundary weights $\{ I_{u+i\ek_1}^{q_1}, I_{u+i\ek_1}^{q_2}, J_{u+j\ek_2}^{q_1}, J_{u+j\ek_2}^{q_2} : i, j \geq 1 \}$ such that the process $(G_{u, \bigcdot}^{q_1}, G_{u, \bigcdot}^{q_2})$ has these properties.
%	
%	\begin{enumerate}
%	
%	\item  For each $v \in u + \ZZ_{\geq 0}^2$, 
%	\[
%	\{ (G_{u, v+x}^{q_1} - G_{u, v}^{q_1}, G_{u, v+x}^{q_2} - G_{u, v}^{q_2}) : x \in \ZZ_{\geq 0}^2 \}  \eqd \{ (G_{u, u+x}^{q_1}, G_{u, u+x}^{q_2}) : x \in \ZZ_{\geq 0}^2 \}.
%	\]
%	
%	\item For each $v \in u + \NN^2$, the vertical increments $\{J_{v + j\ek_2}^{q_1} : u_2 - v_2 + 1 \leq j \leq 0 \}$ and $\{J_{v + j\ek_2}^{q_2} : j \geq 1 \}$ are mutually independent.
%	
%	\item For each $v \in u + \ZZ_{\geq 0}^2$ and $q \in \{q_1, q_2\}$, the increment variables $\{I_{v + i\ek_1}^q, J_{v + j\ek_2}^q : i, j \geq 1 \}$ are mutually independent.  Also the increment variables $\{I_{v - i\ek_1}^q, J_{v - j\ek_2}^q : 0 \leq i < v_1 - u_1, 0 \leq j < v_2 - u_2 \}$ are mutually independent.
%	
%	\item For each $v \in u + \ZZ_{\geq 0}^2$, $i \geq 1$, $j \geq 1$, and $q \in \{q_1, q_2\}$, the increments have marginal distributions $I_{v + i\ek_1}^q \sim \Geom\left(q \right)$ and $J_{v + j\ek_2}^q \sim \Geom \left( \frac{r}{q} \right)$.  
%	
%	\end{enumerate}
%	\end{theorem}  

	\begin{cor} \label{lem:RWIndependent} 
	Fix $0 < r < 1$ and let the bulk weights $\{\w_x:x\in\ZZ^2\}$ be i.i.d.\ $\Geom(r)$ random variables. Let $\xi_\star, \xi^\star, \eta_\star, \eta^\star \in \ri \sU$ be such that $\xi_\star \cdot \ek_1 < \xi^\star \cdot \ek_1$ and $\eta_\star \cdot \ek_1 > \eta^\star \cdot \ek_1$. The processes $\bigl\{S_{m}^{\xi^\star,\eta_\star}: m \in \flint{-N^{2 / 3},-1 }\bigr\}$ and $\bigl\{S_{n}^{\xi_\star,\eta^\star}: n \in \flint{ 1, N^{2 / 3} } \bigr\}$, as defined in \eqref{Sdef},  are independent.
    \end{cor}

	\begin{proof}
	%First, note that examining the construction in the proof of Theorem 4.7 of \cite{Jan-Ras-20-aop} 
	Examining the construction in the proof of the previous theorem one sees that the processes $\{J_x^{\xi_\star}, J_x^{\xi^\star} \}_{x \in \mc I}$ and $\{\widehat{J}_{x+\ek_1}^{\eta_\star}, \widehat{J}_{x+\ek_1}^{\eta^\star}\}_{x \in \mc I}$ can be constructed simultaneously. Then the independence of $\{\w_x:x\in\ZZ_{\ge1}\times\ZZ\}$ and 
	$\{\w_x:x\in\ZZ_{\le0}\times\ZZ\}$ implies that the joint distribution of the two process (that are now defined on a larger, product space) is in fact a product measure and the two processes are independent.

    Next, Theorem \ref{thm:coupledGeomLPP} says that $\{ J_{j\ek_2}^{\xi^\star} \}_{j \leq 0}$ is independent of $\{J_{j\ek_2}^{\xi_\star} \}_{j \geq 1}$ and that   
    %To apply Theorem \ref{thm:coupledGeomLPP} to the reversed process, the order of indexing is reversed so that 
    $\{ \widehat{J}_{\ek_1+(j-1)\ek_2}^{\eta_\star} \}_{j \leq 0}$ is independent of $\{\widehat{J}_{\ek_1+(j-1)\ek_2}^{\eta^\star} \}_{j \geq 1}$.  
    The claim follows from these independence properties.    %So all four processes involved in defining the random walks are mutually independent. 
    %(A,B) independent of (C,D), A independent of B, and C independent of D, implies all four are mutually independent (not only pairwise).
     \end{proof}

%\section{Queueing}\label{app:queueing}
The proof of Theorem \ref{thm:coupledGeomLPP} follows closely that of \cite[Theorem 3.1]{Bal-Bus-Sep-20}. It is based on a few results from queuing theory. The queuing-theoretic interpretation is not important for this paper; however, it gives some intuition behind the algebra that follows.
To this end, consider a \textit{queue} or \textit{service station} with a single \textit{server} and unbounded room for
customers waiting to be served. 
Index the bi-infinite sequence of customers by $j$.
The server serves one customer at a time. Once the service
of customer $j$ is complete, they leave the queue
and customer $j+1$ enters service if they were already waiting
in the queue. If the queue is empty after the departure of customer $j$, then the server
remains idle until customer $j + 1$ arrives. 
Let $\mb{s} = (s_j)_{j \in \ZZ}$ denote the service process, i.e.\ $s_j$ is the time it takes to service customer $j$. %These will be i.i.d.\ Geom($r$) random variables.
Let $\mb{a} = (a_j)_{j \in \ZZ}$ be the inter-arrival process, i.e.\ $a_j$ is the time elapsed between the arrivals of customers $j-1$ and $j$. Assume that
\begin{align}\label{a-cond}
\lim_{n \rightarrow -\infty} \sum_{i=n}^0 (s_i - a_{i+1}) = - \infty.
\end{align}
Let $\mb{G} = (G_j)_{j \in \ZZ}$ be a sequence of customer arrival times such that $a_j=G_j-G_{j-1}$.  
Define the sequence $\tilde{\mb{G}}=(\tilde G_j)_{j\in\ZZ}$ 
by 
\begin{align}\label{D-def}
\tilde{G}_j = \sup_{k \leq j} \Bigl\{ G_k + \sum_{i=k}^j s_i  \Bigr\}.
\end{align}
Condition \eqref{a-cond} guaranties that the supremum is achieved and that $\tilde G_j$ is a finite real number.
The recurrence relation
\begin{align}\label{Dep-ind}
\tilde G_j=(\tilde G_{j-1}+s_j)\vee(G_j+s_j)
\end{align}
provides a natural interpretation of $\tilde G_j$ as the time customer $j$ leaves the service station. 
%Precisely, if customer $j-1$ departs before customer $j$ arrives, then service starts immediately when customer $j$ arrives and the customer departs at time $G_j+s_j$. If, on the other hand, the customer arrives before customer $j-1$ departs, then service starts only when the latter leaves the system and customer $j$ departs at time $\tilde G_{j-1}+s_j$.
%Iterating \eqref{Dep-ind} gives that for
%$\ell<j$ with $\abs{\ell}$ large enough we have
%\[\tilde G_j=\Bigl(\tilde G_{\ell-1}+\sum_{i=\ell}^j s_i\Bigr)\vee
%\max_{\ell\le k\le j}\Bigl(G_k+\sum_{i=k}^j s_i\Bigr)
%=\Bigl(\tilde G_{\ell-1}+\sum_{i=\ell}^j s_i\Bigr)\vee\tilde G_j.\]
%Thus, for $k<j$ at which the maximum is achieved, $\tild G_{k-1}\le G_k$ and thus the queue is empty when customer $k$ arrives. 
%If we take $k$ to be the largest such index (below $j$) then the queue will be constantly serving customers from the time customer $k$ arrives to the time when customer $j$ departs. Thus, $\Gtil_j=G_k+s_k+\dotsc+s_j$.

It is noteworthy that \eqref{D-def} is not the only solution to \eqref{Dep-ind}. For example, the sequence that is identically equal to $\infty$ is another solution. 
However, adapting the proof of \cite[Lemma 4.3]{Jan-Ras-20-jsp} to the current setting shows that under the assumptions that $(s_j)$ is i.i.d., $(a_j)$ is ergodic and independent of $(s_j)$, and the mean of $a_j$ is strictly larger than the mean of $s_j$, \eqref{D-def} is the unique stationary almost surely finite solution to \eqref{Dep-ind}.

Define the inter-departure process $\mb{d} = (d_j)_{j \in \ZZ} = D(\mb{a}, \mb{s})$ by $d_j = \tilde{G}_j - \tilde{G}_{j-1}$.  Define the sojourn process $\mb{t} = (t_j)_{j \in \ZZ} = S(\mb{a}, \mb{s})$ by $t_j = \tilde{G}_j - G_j$. Define the dual service times $\check{\mb{s}} = (\check{s}_j)_{j \in \ZZ} = R(\mb{a}, \mb{s})$ by $\check{s}_j = a_j \wedge t_{j-1}$. These definitions do not depend on the particular sequence $\mb G$ which was selected. 
%The relation \ref{Dep-ind} implies that 
%\begin{align}\label{recIJ}
%t_j\wedge d_j=s_j\quad\text{for all }j\in\ZZ.
%\end{align}

Note that 
    \begin{align}\label{cocyclequeue}
    t_j+a_j=\tilde G_j-G_{j-1}=t_{j-1}+d_j.
    \end{align}
Also, \eqref{Dep-ind} implies
\begin{align}\label{s<d}
s_j\le d_j\quad\text{for all }j\in\ZZ.
\end{align}
Subtracting $G_j$ from both sides in 
\eqref{D-def} and expanding $G_k-G_j=-\sum_{i=k+1}^j a_i$ shows that 
\begin{align}\label{t1>t2}
\text{the sojourn times $t_j$ are non-increasing functions of the inter-arrival times $a_j$.}
\end{align}

The following is Lemma A.1 from \cite{Bal-Bus-Sep-20}. 

\begin{lemma} \label{lem:queueingIdentity}  The following holds for any $\mb{a}, \mb{b}, \mb{s}$ for which all the involved departure times are defined:
\[
D(D(\mb{b}, \mb{a}), \mb{s}) = D(D(\mb{b}, R(\mb{a}, \mb{s})), D(\mb{a}, \mb{s})).
\]
\end{lemma}

For horizontal edge weight $I$, vertical edge weight $J$, and vertex weight $\omega$, define  
\[
	I^{\prime} =\omega+(I-J)^{+},\quad 
	J^{\prime} =\omega+(I-J)^{-},\quad\text{and}\quad
	\omega^{\prime} =I \wedge J.
\]
%Note that
%\begin{align}\label{recIJ}
%I^\prime\wedge J^\prime=\w.
%\end{align}
The next lemma can be proved for example using Laplace transforms. It is essentially a consequence of the memoryless property of the Geometric distribution.

\begin{lemma} \label{lem:geometricInduction}  
Fix $0 < r < 1$ and $r \leq q \leq 1$. Let $\omega \sim \Geom(r)$, $I \sim \Geom(q)$, and $J \sim \Geom\left(\frac{r}{q}\right)$ be independent.  Then the following hold.
\begin{enumerate} [label=\rm(\alph{*}), ref=\rm\alph{*}] \itemsep=2pt 
\item $I-J$ and $I \wedge J$ are independent.  
\item  The distribution of $(I-J)^+$ is the same as that of the product of a $\Ber\bigl(\frac{q-r}{q(1-r)}\bigr)$ and an independent $\Geom(q)$ random variables.  
\item  The triple $(I^\prime, J^\prime, \omega^\prime)$ has the same distribution as $(I, J, \omega)$. 
\end{enumerate}
\end{lemma}

Take  $0<\sigma<\alpha_1 < \alpha_2$ in $(0,1)$.
Let $\mb{b}^i$ be an i.i.d.\ sequence of $\Geom(\alpha_i)$ random variables for $i \in \{1, 2\}$ and let $\mb{s}$ be an i.i.d.\ sequence of $\Geom(\sigma)$ random variables, which are all mutually independent.   Define the arrival sequences $(\mb{a}^1, \mb{a}^2) = (\mb{b}^1, D(\mb{b}^2, \mb{b}^1))$.  Define $\mb{d}^k = D(\mb{a}^k, \mb{s})$, $\mb{t}^k = S(\mb{a}^k, \mb{s})$, and $\check{\mb{s}}^k = R(\mb{a}^k, \mb{s})$ for $k \in \{1, 2\}$.  

\begin{lemma}\label{lem:geomQueueing}
The following statements are true.
\begin{enumerate} [label=\rm(\alph{*}), ref=\rm\alph{*}] \itemsep=2pt 
		\item\label{geomQueue1} Marginally, $\mb{a}^2$ is a sequence of i.i.d.\ $\Geom(\alpha_2)$ random variables.  
		\item\label{geomQueue2} For each $k \in \{1, 2\}$ and $m \in \ZZ$, the random variables $\{d_j^k\}_{j \leq m}$, $t_m^k$, and $\{\check{s}_j^k\}_{j \leq m}$ are mutually independent.  Their marginal distributions are $d_j^k \sim \Geom(\alpha_k)$, $t_m^k \sim \Geom\bigl(\frac{\sigma}{\alpha_k}\bigr)$, and $\check{s}_j^k \sim \Geom(\sigma)$.
		\item\label{geomQueue3} For each $k \in \{1, 2\}$, the sequences $\mb{d}^k$ and $\check{\mb{s}}^k$ are mutually independent.  Their marginal distributions are $d_j^k \sim \Geom(\alpha_k)$ and  $\check{s}_j^k \sim \Geom(\sigma)$.  
		\item\label{geomQueue4} $(\mb{d}^1, \mb{d}^2) \eqd (\mb{a}^1, \mb{a}^2)$.
		\item\label{geomQueue5} For each $m \in \ZZ$, the random variables $\{a_i^2\}_{i \leq m}$ and $\{a_j^1\}_{j \geq m+1}$ are mutually independent.
	\end{enumerate}
\end{lemma}

The proof of the first three claims follows from Lemma B.2 of \cite{Fan-Sep-20} by replacing the exponential version of the induction with the geometric version in Lemma \ref{lem:geometricInduction}.  The proof of the last two claims follows from Lemma A.2 of \cite{Bal-Bus-Sep-20} with the same replacements. 
Note that \eqref{s<d} and \eqref{t1>t2} imply
 \begin{align}\label{mono-at}
a_j^1\le a_j^2\quad\text{and}\quad
t_j^1\ge t_j^2\quad\text{for all }j\in\ZZ.
 \end{align}

\begin{proof}[Proof of Theorem \ref{thm:coupledGeomLPP}]
 Fix $u\in\ZZ^2$. We start by constructing a joint LPP process $(L_x^{1}, L_x^{2})_{x \in u + \ZZ_{\geq 0} \times \ZZ}$.  In the bulk, we have the i.i.d.\ $\Geom(r)$ weights $\{\omega_x : x_1 > u_1\}$.  For $\ell \in \{1, 2\}$, let $\mb{Y}^{\ell} = \{Y_j^{\ell}\}_{j \in \ZZ}$ be a sequence of i.i.d.\ $\Geom(r/q_\ell)$ random variables such that $\{\mb{Y}^{1},\mb{Y}^{2},\w\}$ are mutually independent.
 Note that $q_1<q_2$ implies that \eqref{a-cond} holds almost surely with $\mb{s}=\mb{Y}^{1}$ and $\mb{a}=\mb{Y}^{2}$. For $\ell \in \{1, 2\}$, define $\mb{J}^{\ell} = \{J_{u + j\ek_2}^{\ell}\}_{j \in \ZZ}$ by $(\mb{J}^{1}, \mb{J}^{2}) = (\mb{Y}^{1}, D(\mb{Y}^{2}, \mb{Y}^{1}))$.  By Lemma \ref{lem:geomQueueing}\eqref{geomQueue1}, marginally $\{J_{u+j\ek_2}^{\ell}\}_{j \in \ZZ}$ are i.i.d.\ $\Geom(r/q_\ell)$.  

For $\ell \in \{1, 2\}$, define the LPP values on this vertical axis by
\begin{align}
	L_{u}^{\ell}=0\quad\text{and}\quad L_{u+j \ek_{2}}^{\ell}-L_{u+(j-1) \ek_{2}}^{\ell}=J_{u+j \ek_{2}}^{\ell} \quad \text { for } j \in \mathbb{Z}.
\end{align}
Note that this means $L_{u+j\ek_2}^\ell$ is negative for $j<0$.  Now, we define the LPP values for $x \in u + \ZZ_{>0}\times \ZZ$:
\begin{align}
	L_{x}^{\ell}=\sup _{j: j \leq x_{2}-u_{2}}\left\{L_{u+j \ek_{2}}^{\ell}+G_{u+\ek_{1}+j \ek_{2}, x}\right\}, \quad I_{x}^{\ell}=L_{x}^{\ell}-L_{x-\ek_{1}}^{\ell}, \quad \text { and } \quad J_{x}^{\ell}=L_{x}^{\ell}-L_{x-\ek_{2}}^{\ell}.  \label{bulkLx}
\end{align}
The supremum is achieved at a finite $j$ because the boundary variables $J^{\ell}$ stochastically dominate the bulk weights $\omega$, as we show next.
Note that one has 
\begin{align}\label{recIJ}
I_x^\ell\wedge J_x^\ell=\w_x\quad\text{for all $x\in u+\ZZ_{>0}\times\ZZ$ and $\ell\in\{1,2\}$.}
\end{align}

For $k\ge0$ and $\ell  \in \{1, 2\}$ let $\mb{J}^{\ell, k} = \{J_j^{\ell, k}\}_{j \in \ZZ} = \{J_{u + k\ek_1 + j\ek_2}^{\ell}\}_{j \in \ZZ}$ and $\mb{s}^k = \{s_j^k\}_{j \in \ZZ} = \{\omega_{u + k\ek_1 + j\ek_2} \}_{j \in \ZZ}$.  Then $\mb{J}^{\ell, 0}$ is the original boundary sequence on the vertical axis.  In the notation of Lemma \ref{lem:geomQueueing}, with $\sigma = r$, $\alpha_1 = q_1$, and $\alpha_2 = q_2$, setting $\mb{b}^\ell=\mb{Y}^{\ell}$ gives $(\mb{a}^1, \mb{a}^2)=(\mb{J}^{q_1}, \mb{J}^{q_2})$.
Then, for any $\ell\in\{1,2\}$, \eqref{a-cond} is satisfied, $G_j^{\ell}=L_{u+j\ek_2}^\ell$, $j\in\ZZ$, is a sequence of arrival times, and $\tilde{G}_j^\ell=L^\ell_{u+\ek_1+j\ek_2}$, $j\in\ZZ$, is the corresponding sequence of departure times. Consequently, $\mb{J}^{\ell,1}=D(\mb{J}^{\ell,0},\mb{s}^1)$. Lemma \ref{lem:geomQueueing}\eqref{geomQueue4} then implies $(\mb{J}^{1,1},\mb{J}^{2,1})\eqd(\mb{J}^{1},\mb{J}^{2})$. Repeating this inductively gives that $\mb{J}^{\ell, k+1} = D(\mb{J}^{\ell, k}, \mb{s}^{k+1})$ and $(\mb{J}^{1, k}, \mb{J}^{2, k}) \overset{d}{=} (\mb{J}^{1}, \mb{J}^{2})$ for all $k \geq 0$.  This and the first inequality in
\eqref{mono-at} imply 
\begin{align}\label{monoJ}
J^1_x\le J^2_x\quad\text{ for all $x\in u+\ZZ_+\times\ZZ$}.
\end{align}

Furthermore, Lemma \ref{lem:geomQueueing}\eqref{geomQueue5} implies that for any $x \in u + \ZZ_{\geq 0} \times \ZZ$, 
\begin{align}
	\left\{J_{x+j \ek_{2}}^{2}: j \leq 0\right\} \quad \text { and } \quad\left\{J_{x+j \ek_{2}}^{1}: j \geq 1\right\} \quad \text { are mutually independent. } \label{vertLineInd}
\end{align}

The definition \eqref{bulkLx} satisfies a semi-group property:  For each $k \geq 0$, the values $L_x^{\ell}$ for $x$ such that $x_1 > u_1 + k + 1$ satisfy
\begin{align}
	L_{x}^{\ell}=\sup _{j: j \leq x_{2}-u_{2}}\left\{L_{u+k \ek_{1}+j \ek_{2}}^{\ell}+G_{u+(k+1) \ek_{1}+j \ek_{2}, x}\right\}.
\end{align}
%To get to $x$, we first must reach the vertical line $u+j\ek_2$ where $j \in \ZZ$.  As before, we can travel below $x$ but never above $x$.  There will be a point where from the shifted vertical axis, we move directly to the right, since $x$ is further right of the line, and the sup is attained at a finite value.  At this point, we moved from $u$ to $u+k\ek_1+j\ek_2$ in the $L_x^{q_i}$ sense.  Then from this point, we can now just travel to $x$ using the bulk weights.  
%
This and the distributional equality $(\mb{J}^{1, k}, \mb{J}^{2, k}) \overset{d}{=} (\mb{J}^{1}, \mb{J}^{2})$
imply that for any $z \in \ZZ_{\geq 0} \times \ZZ$,
\begin{align}
	\left\{I_{z+x+\ek_{1}}^{1}, I_{z+x+\ek_{1}}^{2}, J_{z+x}^{1}, J_{z+x}^{2}: x \in u+\mathbb{Z}_{\geq 0} \times \mathbb{Z}\right\} \stackrel{d}{=}\left\{I_{x+\ek_{1}}^{1}, I_{x+\ek_{1}}^{2}, J_{x}^{1}, J_{x}^{2}: x \in u+\mathbb{Z}_{\geq 0} \times \mathbb{Z}\right\}. \label{transInv} 
\end{align}
The index on the $I$ increments requires $x+\ek_1$ because the increments are not defined on the boundary, where $x_1 = u_1$.  

Next, we claim that for $\ell \in \{1, 2\}$, and for any $u\in\ZZ^2$,
\begin{align}
	\begin{aligned}
	&\left\{I_{u+i \ek_{1}}^{\ell}, J_{u+j \ek_{2}}^{\ell}: i, j \in \mathbb{Z}_{>0}\right\} \text { are mutually independent with marginal distributions }  \\ 
	& I_{u+i \ek_{1}}^{\ell} \sim \Geom(q_\ell) \text { and } J_{u+j \ek_{2}}^{\ell} \sim \Geom(r/q_\ell). 
	\end{aligned}\label{forwardLind}
\end{align}

 We have already shown that $\mb{J}^{\ell}$ are i.i.d.\ $\Geom(r/q_\ell)$ random variables.  Also notice that $\{I_{u+i \ek_{1}}^{\ell} : i \geq 1\}$ are a function of only $\{J_{u+j \ek_{2}}^{\ell}, \omega_{u+i\ek_1+j\ek_2}: i \geq 1, j \leq 0\}$ which are independent of $\{J_{u+j\ek_2}^\ell : j \geq 1\}$.  What remains to prove is that the horizontal increments are i.i.d.\ and to determine their maginal distribution. For this, we prove the following claim inductively in $n \geq 1$:
\begin{align}
	\begin{aligned}&\left\{I_{u+i \ek_{1}}^{\ell}, J_{u+n \ek_{1}+j \ek_{2}}^{\ell}: 1 \leq i \leq n, j \leq 0\right\} \text { are mutually independent with } \\ &\text { marginal distributions } I_{u+i \ek_{1}}^{\ell} \sim \Geom(q_\ell) \quad \text { and } \quad J_{u+n \ek_{1}+j \ek_{2}}^{\ell} \sim \Geom(r/q_\ell).\end{aligned} \label{reverseLind}
\end{align}
This and the fact that $I_{u+(n+1)\ek_1}^\ell$ is a function of $\{J^\ell_{u+n\ek_1+j\ek_2},\omega_{u+(n+1,j)}:j\le0\}$ imply the mutual independence of the horizontal increments.

We now prove \eqref{forwardLind}. For the base case $n=1$, consider inter-arrival times $\{a_j = J_{u+j\ek_2}^\ell : j \leq 0\}$ and service times $\{s_j = \omega_{u + \ek_1 + j\ek_2} : j \leq 0\}$.  The inter-departure times are $\{d_j = J_{u+\ek_1 + j\ek_2}^\ell : j \leq 0\}$.  The sojourn time is $t_0 = I_{u+\ek_1}^\ell$.  Lemma \ref{lem:geomQueueing}\eqref{geomQueue2} then gives the above claim for $n=1$.  

For the inductive step, assume the claim holds for a fixed $n\ge1$.  Then use inter-arrival times $\{a_j = J_{u+n\ek_1 + j\ek_2}^\ell : j \leq 0\}$ and service times $\{s_j = \omega_{u+(n+1)\ek_1 + j\ek_2} : j \leq 0\}$ which are independent of $\{I_{u+i\ek_1}^\ell : 1 \leq i \leq n \}$ by the inductive hypothesis.  Then compute the corresponding inter-departure times $\{d_j = J_{u+(n+1)\ek_1 + j\ek_2} : j \leq 0\}$ and the sojourn time $t_n = I_{u+(n+1)\ek_1}^\ell$.  Lemma \ref{lem:geomQueueing}\eqref{geomQueue2} again gives the validity of the claim for $n+1$, completing the proof of the claim \eqref{reverseLind}.  

Combining \eqref{cocyclequeue} with observation that $I^\ell_x$ are sojourn times gives
    \begin{align}\label{cocycleIJ}
    I^\ell_{x+\ek_1}+J^\ell_{x+\ek_1+\ek_2}=J^\ell_{x+\ek_2}+I^\ell_{x+\ek_1+\ek_2}\quad\text{ for all $x\in u+\ZZ_{\ge0}\times\ZZ$ and $\ell\in\{1,2\}$.}
    \end{align}
And with the second inequality in \eqref{mono-at} we get 
\begin{align}\label{monoI}
I^1_x\ge I^2_x
\quad\text{for all }x\in\ZZ_{>0}\times\ZZ.
\end{align}

Lastly, observe that $\{L_x^\ell : x \in u + \ZZ_{\geq 0}^2\}$ are last passage times with boundary weights $\{I^\ell_{u+i\ek_1},J^\ell_{u+j\ek_2}:i,j\in\ZZ_{>0}\}$ and bulk weights $\w_x$, $x\in u+\N^2$.  Indeed, if we denote by $G^\ell_{u,x}$ the passage time from $u$ to $x\in u+\N^2$ with these boundary and bulk weights, then as in \eqref{GSW}
\begin{align*} 
&G_{u, x}^{\ell} = \max _{1 \leq k \leq x_{1}-u_{1}}\Bigl\{\sum_{i=1}^{k} I_{u+i \ek_{1}}^{\ell}+G_{u+k \ek_{1}+\ek_{2}, x}\Bigr\} \bigvee \max_{1 \leq m \leq x_{2}-u_{2}}\Bigl\{\sum_{j=1}^{\ell} J_{u+j \ek_{2}}^{\ell}+G_{u+\ek_1+m\ek_2, x}\Bigr\} \\ 
	&= \max _{1 \leq k \leq x_{1}-u_{1}}\left\{L_{u+k \ek_{1}}^{\ell}+G_{u+k \ek_{1}+\ek_{2}, x}\right\} \bigvee \max _{1 \leq m \leq x_{2}-u_{2}}\left\{L_{u+m \ek_{2}}^{\ell}+G_{u+\ek_{1}+m \ek_{2}, x}\right\} \\
	&= \sup _{j \leq 0}\Bigl\{L_{u+j \ek_{2}}^{\ell}+\max _{1 \leq k \leq x_{1}-u_{1}}\left[G_{u+\ek_{1}+j \ek_{2}, u+k \ek_{1}}+G_{u+k \ek_{1}+\ek_{2}, x}\right]\Bigr\}  \bigvee \max_{1 \leq m \leq x_{2}-u_{2}}\left\{L_{u+m \ek_{2}}^{\ell}+G_{u+\ek_{1}+\ell \ek_{2}, x}\right\} \\
	&=\sup _{j: j \leq x_{2}-u_{2}}\left\{L_{u+j \ek_{2}}^{\ell}+G_{u+\ek_{1}+j \ek_{2}, x}\right\}=L_{x}^{\ell}. 
\end{align*}

By \eqref{recIJ}, the weights $\w_x$ can be recovered from the edge weights $I^\ell_x$ and $J_x^\ell$. Then, due to \eqref{transInv} we can extend the process $\bigl\{\w_x,I_{x+\ek_{1}}^{1}, I_{x+\ek_{1}}^{2}, J_{x}^{1}, J_{x}^{2}: x \in u+\mathbb{Z}_{\geq 0} \times \mathbb{Z}\bigr\}$ to a stationary process on the whole lattice. This produces a $\overline T$-invariant probability measure $\overline\P_{q_1,q_2}$
on $(\overline\Omega,\overline\sG)$ whose marginal on $\Omega$ is exactly $\P$. 
We now verify that all the claims in the theorem hold for this choice of measure.

Property \eqref{coupledGeomLPP0.i} holds by construction and the independence property \eqref{coupledGeomLPP0.ii} follows from  the definition \eqref{bulkLx}.
Recall that an $\overline\w\in\overline\Omega$ has coordinate projections
$\bigl\{\w_x,I_x^{1}, I_x^{2}, J_{x}^{1}, J_{x}^{2}: x \in \ZZ^2\bigr\}$. 
Thus, the shift-covariance in \eqref{coupledGeomLPP0.iv.1}  holds trivially.
%since the weights are simply the coordinate projections. 
The recovery and monotonicity properties in \eqref{coupledGeomLPP0.iv.2} follow from \eqref{recIJ}, \eqref{monoJ}, and \eqref{monoI}.  
The additivity property \eqref{cocIJ2} is given in \eqref{cocycleIJ} and \eqref{add2} follows from that. Then, as it was the case for Theorem \ref{thm:coupledGenericLPP}\eqref{BusTh.d'}, property \eqref{coupledGeomLPP0.v} 
follows from \eqref{add2} and the shift-invariance of $\overline\P_{q_1,q_2}$.

Observe that  $(I^\ell_x,J^\ell_x)$ has mean $\bigl(q_\ell/(1-q_\ell),r/(q_\ell-r)\bigr)$.
A direct computation using the explicit formulas \eqref{p->xi} and \eqref{eq:shape} shows that this is equal to $\nabla\gpp(\xip(q_\ell))$.
%r/(1-r)+\sqrt{r}/(1-r)(\sqrt{x_2/x_1},\sqrt{x_1/x_2})
%r/(1-r)+\sqrt{r}/(1-r)((q-r)/(1-q)/\sqrt{r},(1-q)\sqrt{r}/(q-r))
%(r+(q-r)/(1-q),r+r(1-q)/(q-r))/(1-r)
%((r-qr+q-r)/(1-q),r(q-r+1-q)/(q-r))/(1-r)
%(q/(1-q),r/(q-r)) 
This completes the proof of part \eqref{coupledGeomLPP0} of the theorem.
Part \eqref{coupledGeomLPPb} follows from \eqref{vertLineInd} and parts \eqref{coupledGeomLPPc} and \eqref{coupledGeomLPPd} from \eqref{forwardLind} and \eqref{reverseLind}.  
%
%
%	Because this construction in the Geometric case is stationary and ergodic, the uniqueness of the Busemann process shows that the general construction and this Geometric construction are almost surely equivalent.  We will continue to use the $I$ and $J$ notation to refer to either construction.
\end{proof}

\section{Verifying Assumption \ref{exitPointTailBound} for the geometric LPP}\label{sec:pfExit}
This appendix is dedicated to the proof of an exponential tail bound for the location of exit points. It can be read independently of the rest of the paper. We assume throughout the section that $\w_0\sim\Geom(r)$ for a given $r\in(0,1)$. 

For $\delta\in(0,1)$ recall the definition of the cone 
    \[S_\delta = \{x \in \RR_{>0}^2 : x\cdot \ek_1 \geq \delta x \cdot \ek_2 \text{ and } x \cdot \ek_2 \geq \delta x\cdot \ek_1  \}.\] 
%The following estimate is one of the main tools in the proof of Theorem \ref{thm:noBiInfinites}.
%\note{C: Need to rewrite this in terms of the stationary parameters.}
%The proof can be read independently of the rest of the results and is left to  Appendix \ref{sec:pfExit}. 

\begin{theorem}\label{geometricExitPointTailBound}  
Assume $\w_0\sim\Geom(r)$ for some $r\in(0,1)$.
For any $\delta \in (0,r)$ and $\kappa \geq 0$ there exist positive finite constants $c_0 = c_0(\delta, r)$, $N_0 = N_0(\delta, r,\kappa)$, and $s_0 = s_0(\delta,r,\kappa)$ such that 
\[
\P\bigl\{\abs{Z^{\xi, \ek_1}(m,n)}\vee\abs{Z^{\xi, \ek_2}(m,n)} \geq s (m + n)^{2/3}\bigr\} 
\leq \exp\{-c_0 s^3\}
\]
for all $(m, n) \in S_\delta \cap \ZZ_{\geq N_0}^2$, $s \geq s_0,$ and $\xi\in\ri\sU$ such that $\xi_1 \in (\delta, 1-\delta)$ and $\abs{\xi_1 - \frac{m}{m+n}} \leq \kappa(m+n)^{-1/3}$.
\end{theorem}

%As we mentioned above, Theorem \ref{thm:coupledGenericLPP} gives us a family of random variables $\{I_x^\xi,J_y^\xi:x,y\in\bbZ^2\}$, for each $\xi\in\ri\sU$.  
%In fact, the theorem gives a whole process, indexed by $\xi\in\ri\sU$. We will not need this in the present paper. We will work with one direction at a time, except for this section where 
%we will need pairs of these boundary weights, i.e.\ the family of random variables
%$\{I_x^\xi,J_y^\zeta:x,y\in\bbZ^2\}$ for a fixed pair of directions $\xi,\zeta\in\ri\sU$.
%We thus define    
%    \begin{align}\label{Gxizeta}
%    G_{x,y}^{\xi,\zeta}(\w)=G_{x,y}^{\rm{SW}}(\w,I^\xi(\w),J^\zeta(\w)).
%    \end{align}
%$\Exit_{x,y}^{\xi,\zeta}$ and $Z_{x,y}^{\xi,\zeta,\ek_k}$ 
%are defined similarly. 

For $p,q\in(r,1)$ consider random variables $\{I^p_{i\ek_1},I^q_{i\ek_1},J^p_{i \ek_2},J^q_{i \ek_2}:i,i\in\NN\}$ 
that are mutually independent and independent of the weights $\w$ and such that 
the $I^p$ variables are $\Geom(p)$, the $I^q$ variables are $\Geom(q)$, 
the $J^p$ variables are $\Geom(r/p)$, and the $J^q$ variables are $\Geom(r/q)$. Note that this parametrization in terms of $p$ and $q$ does not agree with the parametrization of the $I$ and $J$ random variables elsewhere in the paper. This abuse of notation is to simplify the formulas in this section; we will also abuse notation and continue to write $\P$ and $\E$ for the probability and expectation on the larger probability space on which this collection of random variables is defined.

We will write 
    \begin{align}\label{Gpq}
    G_x^{p,q}=G_x^{\rm{SW}}(\w,I^p,J^q),\quad
    G_x^p=G_x^{\rm{SW}}(\w,I^p,J^p),\quad\text{and}\quad
    G_x^q=G_x^{\rm{SW}}(\w,I^q,J^q).
    \end{align}
The quantities $\Exit_x^{p,q}$, $Z_x^{p,q,\ek_k}$, $\Exit_x^p$, $Z_x^{p,\ek_k}$,
$\Exit_x^q$, and $Z_x^{q,\ek_k}$ are defined similarly. 

From \eqref{IJw-indep} (which follows from Theorem \ref{thm:coupledGeomLPP}\eqref{coupledGeomLPPc}) we see that $\{G_x^p:x\in\bbZ^2_{\ge0}\}$ has the same distribution as $\{G_x^{\xip(p)}:x\in\bbZ^2_{\ge0}\}$. The same, of course, holds when $p$ 
is replaced by $q$.

One of our major uses of the exact solvability of the model comes through an exact formula for a particular log moment generating function of the increment stationary passage time.

\begin{propo}\label{prop:logMGF}  Let $m, n \in \ZZ_{\geq 0}^2$, and $p, q \in (r,1)$.  Then
\[
\log \E\Bigl[\exp\Bigl\{ \log \bigl(\frac{q}{p}\bigr) G^{p, q}(m, n)  \Bigr\} \Bigr] = m \log\bigl( \frac{1-p}{1-q} \bigr) + n \log \Bigl( \frac{1-\frac{r}{q}}{1-\frac{r}{p}} \Bigr).
\]
\end{propo}

\begin{proof}  
Start by writing
\begin{align}
&\log \E\Bigl[\exp\Bigl\{ \log \bigl(\frac{q}{p}\bigr) G^{p, q}(m, n)  \Bigr\} \Bigr] 
= \log \E\Bigl[ \prod_{i=1}^m (q/p)^{I^{p}_{(i, 0)}} \,e^{\log(q/p) (G^{p, q}(m, n) - G^{p, q}(m, 0))}\Bigr]\notag\\
&\qquad= m \log \bigl(\frac{1-p}{1-q}\bigr)
+\log \E\Bigl[ \prod_{i=1}^m \bigl(\frac{1-q}{1-p}\cdot(q/p)^{I^{p}_{(i, 0)}}\bigr) e^{\log(q/p) (G^{p, q}(m, n) - G^{p, q}(m, 0))}\Bigr].\label{foo0001}
\end{align}
Next, note that $\E[\frac{1-q}{1-p}(q/p)^{I^p_{i\ek_1}}]=1$ and for any $n\in\bbZ_{\ge0}$
    \[\E\Bigl[\frac{1-q}{1-p}(q/p)^{I^p_{i\ek_1}}\one\{I^p_{i\ek_1}=n\}\Bigr]=q^n(1-q).\] 
This means that the product inside the expectation on the right-hand side in \eqref{foo0001} is a Radon-Nikodym derivative and using it to change the measure $\P$ switches the distribution of the boundary $I^p$ weights to have the same distribution as the $I^q$ weights. Consequently, \eqref{foo0001} is equal to 
\begin{align*}
&m \log \bigl(\frac{1-p}{1-q}\bigr)
+\log \E\Bigl[  e^{\log(q/p) (G^q(m, n) - G^q(m, 0))}\Bigr]\\
&\qquad=m \log \bigl(\frac{1-p}{1-q}\bigr)
+\log \E\Bigl[  e^{\log(q/p) (G^{\xip(q)}(m, n) - G^{\xip(q)}(m, 0))}\Bigr]\\
&\qquad=m \log \bigl(\frac{1-p}{1-q}\bigr)
+\log \E\Bigl[  e^{\log(q/p) G^{\xip(q)}(0, n)}\Bigr].
\end{align*}
For the last equality we used the additivity \eqref{Gadd} and the shift-invariance \eqref{Gxi-stat}. Now, simply compute
\begin{align*}
m \log \bigl(\frac{1-p}{1-q}\bigr)
+\log \E\Bigl[  e^{\log(q/p) G^{\xip(q)}(0, n)}\Bigr]
%= m \log \left(\frac{1-p}{1-q}\right) + \log \E\left[  e^{\log(q/p) G^{q}(0,n)}\right] \\
&= m \log \bigl(\frac{1-p}{1-q}\bigr) + \log \E\Bigl[ e^{\log(q/p) J_{(0,1)}^q}\Bigr]^n \\
%&= m \log \left(\frac{1-p}{1-q}\right) + n \log \frac{1-\frac{r}{q}}{1-\frac{r}{q} e^{\log(q/p)}} \\
&= m \log \bigl(\frac{1-p}{1-q}\bigr) + n \log \Bigl(\frac{1-\frac{r}{q}}{1-\frac{r}{p}}\Bigr). 
\qedhere
\end{align*}
\end{proof}

We prove Theorem \ref{geometricExitPointTailBound} after a series of calculus lemmas.
The following lemma is immediate from the definitions. Recall (\ref{xi->p}-\ref{eq:statshape}).

\begin{lemma}\label{lem:StationaryShapeProperties}  Fix $a, b > 0$.  The function $p \mapsto M^p(a, b)$ is
continuous and strictly convex on $(r, 1)$, 
decreasing on $(r, \pxi(a, b)]$ with range $[\gpp(a, b), \infty)$, and
increasing on $[\pxi(a, b), 1)$ with range $[\gpp(a, b), \infty)$.
% \begin{enumerate} [label=\rm(\roman{*}), ref=\rm(\roman{*})]
%	\item continuous and strictly convex on $(r, 1)$,
%	\item decreasing on $(r, \pxi(a, b)]$ with image $[\gpp(a, b), \infty)$, and
%	\item increasing on $[\pxi(a, b), 1)$ with image $[\gpp(a, b), \infty)$.
%\end{enumerate}
\end{lemma}
%
%\begin{proof}  The second partial derivative of 
%$M^p(a, b) = \frac{ap}{1-p} +  \frac{br}{p-r}$ 
%%$= -a+\frac{a}{1-p} +  \frac{br}{p-r}$ 
%with respect to $p$ is always positive:
%\begin{align*}
%\frac{\partial}{\partial p} M^p(a, b) &= \frac{a}{(1-p)^2} - \frac{br}{(p-r)^2} \\
%\frac{\partial^2}{\partial p^2} M^p(x, y) &= \frac{2a}{(1-p)^3} + \frac{2br}{(p-r)^3} > 0.
%\end{align*}
%Also, $M^p$ is always positive and $\pxi(a,b)$ is exactly the unique minimizer of $M^p(a, b)$.  
%%Setting the first derivative to 0 gives $p=\frac{ar+\sqrt{abr}}{a+\sqrt{abr}}$.
%%Setting this equal to the formula for $\pxi(a,b)$ and cross multiplying checks that these
%%are actually equal!!
%\end{proof}

Consequently, for each $\lambda \in (1, 1/r)$, there exists a unique pair $\pxi_-^\lambda(a,b)\in(r,\pxi(a,b))$ and $\pxi_+^\lambda(a,b)\in(\pxi(a,b),1)$ such that $\pxi_+^\lambda(a,b)=\lambda\pxi_-^\lambda(a,b)$ and $M^{\pxi_-^\lambda}(a,b) = M^{\pxi_+^\lambda}(a, b)$.   
Precisely, using a little bit of calculus, we get that if $a \neq rb$, then
%\begin{equation}
\begin{align}\label{pxi-}
\pxi_-^\lambda(a,b) = \frac{r(\lambda + 1)(a-b) + \sqrt{r^2(\lambda + 1)^2 (a-b)^2 - 4r\lambda (ra-b)(a-rb)}}{2 \lambda (a-rb)} %\label{ZetaMinus}
\end{align}
%\end{equation}
%Move the terms that have a to one side and the ones with b to the other. It makes it easier to solve for $\pxi_-^\lambda$.
%Note that as a function of $\lambda$ the thing under the square root is a convex quadratic with a minimum at $\lambda=\frac{2(ar-b)(a-br)}{(a-b)^2r}-1\in<1$ and its value at $\lambda=1$ is positive. (Both these claims boil down to $r^2-2r+1>0$.) So the thing under the square root is positive for $\lambda\in(1,1/r)$.
%If $b=ar$ then the solutions are $s=0$ and $s=\frac{(\lambda+1)r}{\lambda(r+1)}$ and we want the positive one. This does correspond to the given formula.
%The product of the two solutions is $\frac{(ar-b)r}{\lambda(a-br)}<1$.
%If $b\in(ar,a/r)$, then the product is negatice and so the smallest of the two solutions is negative and we want the positive one and again this corresponds to the formula.
%If $b<ar$ or $b>a/r$ then the product is positive and both solutions are positive. 
%If $b<ar$ then the product is $<1$.  (This boils down to $b(\lambda-1)<a(\lambda-r^2)/r$ which holds because $b<ar<a/r$ and $\lambda>1>r^2$.) So the smaller of the two solutions must be less than one and we want the larger one, which again corresponds to the formula.
%When $b>a/r$ the product is $>r^2$. (This boils down to $b(1-\lambda r^2)>ar(r-\lambda)$.) So the larger of the two solutions is $>r$ and we want the smaller one. Again, the formula is correct in this case too.
and if $a = rb$, then 
%\begin{equation}
\[\pxi_-^\lambda(a,b) = \frac{r+1}{\lambda + 1}.\] %  \label{ZetaMinusSp}
%\end{equation}
%%In both cases, $\pxi_+^\lambda(a,b)=\lambda\pxi_-^\lambda(a,b)$.
%
This extends continuously to $\lambda=1$ and $\lambda=1/r$ with $\pxi_{\pm}^1(a, b) = \pxi(a, b)$, $\pxi_-^{1/r}(a,b) = r$, and $\pxi_+^{1/r}(a, b) = 1$. 
% (which could be extended to all $\lambda \geq \frac{1}{r}$ if desired).  

%\begin{remark}
%Part of the computation involved in getting \eqref{pxi-} is to show that the quantity under the square root is minimized at a $\lambda<1$ and that it is positive at $\lambda=1$. As a consequence, the quantity is positive and increasing in $\lambda\in(1,1/r)$.
%\end{remark}
%\smallskip

For $\xi\in \RR^2_{>0}$ and $p, q \in (r, 1)$, define 
\[L^{p, q}(\xi) = L^{p,q}(\xi_1,\xi_2)=\xi_1 \log \bigl(\frac{1-p}{1-q}\bigr) + \xi_2 \log \Bigl(\frac{1-\frac{r}{q}}{1-\frac{r}{p}}\Bigr).\]    
%
%
% Could add the vertical part but also symmetry would work instead of actually going through the whole argument again.
%
Then for $\xi\in\RR_{>0}$ and $q \in [r, 1)$ set
\begin{align}\label{Lq}
\mc{L}^{\lambda, q}(\xi) = \inf_{q<s<1/\lambda} L^{s, \lambda s}(\xi)
\end{align}
when $\lambda\in[1,1/q)$
and $\mc{L}^{\lambda,q}(\xi)=\infty$ when $\lambda\ge1/q$.
%
% \textbf{Caution: } Slight differences from original paper: here $v \geq u$ rather than the other way around.  Also, the horizontal and vertical roles are backwards here (horizontal uses the $q$)
%
%with the convention that $\displaystyle\inf_\varnothing = \infty$.  
In the special case where $q = r$ we abbreviate 
$\mc{L}^{\lambda}(\xi) = \mc{L}^{\lambda, r}(\xi) $.
%\inf\limits_{\substack{s, t \in (r, 1) \\ t = r\lambda}} L^{s, t}(\xi).%\]

\begin{lemma}\label{lem:restrictedLogMGFProperty}  
Let $\xi\in\RR_{>0}$, $q \in [r, 1)$, and $\lambda \in (1, 1/q)$.
Then the infimum in \eqref{Lq} is uniquely achieved 
at $s = \max \{ q, \pxi_-^\lambda(\xi) \}$. 
%If $\lambda=1$ then $L^{s, s}(\xi)=0$ for all $s\in(1,1/\lambda)$ and 
%$\mc{L}^{1, q}(\xi) = 0$.
\end{lemma}

\begin{proof}
%Suppose $\lambda \in \left(1, \frac{1}{r}\right)$. Then 
A direct computation gives  
\begin{align*}
\frac{\partial}{\partial s}  L^{s, \lambda s}(\xi) 
%&= -\frac{\xi_1}{1-s}+\frac{\xi_1\lambda}{1-\lambda s}+\frac{\xi_2\lambda}{\lambda s-r}-\frac{\xi_2}{s-r}.
%&= \frac1s\left(-\frac{\xi_1 s}{1-s}+\frac{\xi_1\lambda s}{1-\lambda s}+\frac{\xi_2 r}{\lambda s-r}-\frac{\xi_2 r}{s-r}\right)\\
&= \frac{1}{s} \left(M^{\lambda s}(\xi) - M^{s}(\xi)\right)\quad\text{for }s\in(r,1/\lambda).
\end{align*}
By Lemma \ref{lem:StationaryShapeProperties}, 
this is a continuous strictly increasing function of $s$
with range $\RR$. %To see that it is strictly increasing, note that strict convexity of $s \mapsto M^s(\xi)$ implies that the slope of the line segment from $s$ to $\lambda s$ strictly increases with $s$.
It is equal to zero at $s = \pxi_-^\lambda(\xi)$. 
%, by the definition of $\pxi_-^\lambda$.  
If $\pxi_-^\lambda(\xi)\in(q,1/\lambda)$, then the unique infinimum in \eqref{Lq} is attained at $\pxi_-^\lambda(\xi)$.  If $\pxi_-^\lambda(\xi)\in(r,q]$, then $L^{s,\lambda s}(\xi)$ strictly increases on $(q,1/\lambda)$ and is thus minimized at $s=q$.   
\end{proof}

For $k,\ell,m,n\in\ZZ$ with $m>k$ and $n\ge \ell$ 
let $G_{(k,\ell)}^q(m,n|1,0)$ be the last-passage time for paths which start at $(k,\ell)$, immediately take an $\ek_1$ step, and then go to $(m,n)$, while collecting the weights $\{I^q_{k+j\ek_1,\ell}:j\in\NN\}$ on the south boundary. Precisely,
    \[G^q_{(k,\ell)}(m,n|1,0)=\max_{1 \leq j \leq m-k}\Bigl\{ \sum_{i=1}^{j} I^q_{k+i \ek_1,\ell} + G_{(k+j) \ek_1 + (\ell + 1)\ek_2, m\ek_1+n\ek_2} \Bigr\}.\]
When $(k,\ell)=\zero$ we omit it from the index.

\begin{lemma}\label{lem:restrictedLogMGFInequality} 
Let $m, n \in \NN$, 
$q \in [r, 1)$, and $\lambda\ge1$.  Then
\[
\log \E\bigl[e^{\log(\lambda) G^q (m, n|1,0)}\bigr] \leq \mc{L}^{\lambda, q}(m,n).
\]
\end{lemma}
 
\begin{proof}  The case $\lambda\ge1/q$ is trivial because the right-hand side is infinite.
When $\lambda=1$ we have $L^{s, s}(m, n) = 0$ for all $s\in(r,1)$ and the claim is again trivial. Therefore, assume $\lambda \in (1, 1/q)$.

Using
    \[G^{q,\lambda q}(m,n)\ge I^q_{\ek_1}+G^q(m,n|1,0)\ge G^q(m,n|1,0)\]
and Proposition \ref{prop:logMGF} we see that
\begin{align}
\log \E\bigl[ e^{\log(\lambda) G^q(m, n|1,0)} \bigr] 
\leq \log \E\bigl[ e^{\log(\lambda) G^{q, \lambda q}(m, n)} \bigr]
= L^{q, \lambda q}(m,n) \label{Lem4:3:eq:1}.
\end{align}

Geometric random variables are stochastically increasing in the parameter. Therefore, 
%If q<p then P(I^q\ge n)=q^n<p^n=P(I^p\ge n)
%OR: If $q_1 < q_2$, then if $U %\sim \text{Unif}(0,1)$,
%\[
%\left\lfloor \frac{\log(U)}{\log(q_1)} \right\rfloor \leq 	\left\lfloor %\frac{\log(U)}{\log(q_2)} \right\rfloor.
%\]
%Here the quantity on the left has distribution $\Geom(q_1)$ and on the right has distribution $\Geom(q_2)$.  
if $q \leq \pxi_-^\lambda(m,n)$, 
\begin{align*}
\log \E\bigl[ e^{\log(\lambda) G^q(m, n|1,0)} \bigr] 
\leq \log \E\bigl[ e^{\log(\lambda) G^{\pxi_-^\lambda}(m, n)|1,0)} \bigr]
\leq L^{\pxi_-^\lambda, \lambda\pxi_-^\lambda}(m,n)
\end{align*}
where the last inequality follows from applying \eqref{Lem4:3:eq:1} with $\pxi_-^\lambda(m,n)$ in place of $q$.

We have thus shown that 
\begin{align*}
\log \E\bigl[ e^{\log(\lambda) G^q(m, n|1,0)} \bigr] 
%&\leq L^{q, \lambda q}(m,n) \mb{1}_{\{q > \zeta_-^\lambda \}} +  L^{\zeta_-^\lambda, \zeta_+^\lambda}(m,n) \mb{1}_{\{q \leq \zeta_-^\lambda \}} \\
\le L^{\max\{q, \pxi_-^\lambda\}, \lambda\max\{q, \pxi_-^\lambda\}}(m,n)
=\mc{L}^{\lambda, q}(m,n),
\end{align*}
where the  equality holds by Lemma  \ref{lem:restrictedLogMGFProperty}.  \end{proof}

%The next lemma follows immediately from Taylor's Remainder Theorem.

\begin{lemma}\label{lem:logMGFExpansion} %Fix $\delta\in(0,1)$ and 
For all $a, b >0$, $\e\in(0,\min\bigl(r,1-s,(1-r)/2\bigr)$, $s\in(r,1)$, and 
$\lambda \in \bigl[\max((r+\e)/s,1), (1-\e)/s\bigr]$
%\note{C: This interval is empty for certain choices of the parameters.}
%\note{S: I think this should be for $s \in (r, 1)$ and $\e \in (0, \min(r, 1-s, (1-r)/2)$ }
\begin{align*}
& \Bigl|\, L^{s, \lambda s}(a,b) - (\lambda - 1)\Bigl( \frac{as}{1-s} + \frac{br}{s-r} \Bigr) - \frac{1}{2} (\lambda -1)^2 \Bigl(\frac{a s^2}{(1-s)^2} - \frac{b r(2s-r)}{(s-r)^2}\Bigr) \Bigr| \leq 2\e^{-3}(a+b)(\lambda-1)^3.
\end{align*}
\end{lemma}

\begin{proof}  
Fix $a$, $b$, $\e$, and $s$ as in the claim and perform a Taylor expansion of $\lambda\mapsto L^{s, \lambda s}(a,b)$, defined on $(r/s,1/s)$, at $\lambda=1$. 
For the error term write
\[
\Bigl|\,\frac{\partial^3}{\partial \lambda^3} L^{s, \lambda s}(a,b) \Bigr|
= \Bigl|\,\frac{2 b r (3 \lambda^2 s^2 - 3 \lambda r s + r^2)}{\lambda^3 (\lambda s - r)^3} + \frac{2 a s^3}{(1 - \lambda s)^3}\Bigr|
\]
and use $\lambda\ge1$, $\lambda s-r\ge\e$, $\lambda s\le 1$, $r\le 1$, $s\le1$, and $1-\lambda s\ge\e$
to bound the above by $8\e^{-3}(a+b)$.  
\end{proof}

\begin{lemma}\label{lem:parameterExpansion}  Fix  $\delta\in(0,r)$. 
Let $C_0 = C_0(\delta, r)$ be given by
\[C_0 = \max\Bigl\{r+1,
 \frac{(\delta^{-1}+1)[(1+r)^2\delta+2r^2+2]}{8(1-r)^2\delta}+\frac{r(r+1)(\delta^{-1}+1)}{4(1-r)\sqrt{r\delta}}+\frac{r\delta^{-1}+1}{2(1-r)\sqrt{r\delta}} \Bigr\}.\]
Then for any 
$\lambda \in (1, 1/r)$ and $a,b>0$ such that $(a,b)\in S_\delta$
\[
\pxi_-^\lambda(a,b) - \pxi(a,b) \geq - C_0 (\lambda -1).
\]
\end{lemma}

\begin{proof} Fix positive $a$ and $b$ with $(a,b)\in S_\delta$. Let 
%\[
%f(\lambda) = r(a-b) \frac{1}{\lambda} + \sqrt{r^2(a-b)^2 + \left(-2r^2 a^2 - 2r^2 b^2 + 4abr(r^2 - r + 1)\right) \frac{1}{\lambda} + r^2(a-b)^2 \frac{1}{\lambda^2}}.
%\]  
\[
f(\lambda) = r(a-b) \frac{1}{\lambda} +
\frac1\lambda\sqrt{r^2(\lambda + 1)^2 (a-b)^2 - 4r\lambda (ra-b)(a-rb)}\,.
\]  
%Then $f(1) = r(a-b) + 2(1-r)\sqrt{rab}$.
%We have
Then
%\begin{align*}
%f'(c) = -\frac{r(a-b)}{c^2} + \frac{\bigl((-2r^3 - 2r)ab + r^2(a+b)^2\bigr)c^{-1} - r^2(a-b)^2c^{-2}}{\sqrt{r\bigl[r(c^2+1)(a-b)^2 + 2c\bigl(2(r^2+1)ab - r(a+b)^2\bigr)\bigr]}}\,.
%&f'(\lambda) 
%= -\frac{r(a-b)}{\lambda^2} + \frac{-2r^2(\lambda+1)(a-b)^2+4r\lambda(ra-b)(a-rb)}{2\lambda^2\sqrt{r^2(\lambda + 1)^2 (a-b)^2 - 4r\lambda (ra-b)(a-rb)}}\\
%&= -\frac{r(t-1)}{\lambda^2}\cdot b + \frac{\bigl(2r^2t^2+2r^2-4tr(r^2-r+1)\bigr)\lambda-2r^2(t-1)^2}{2\lambda^2\sqrt{r^2\lambda^2t^2-2r^2\lambda^2t-2r^2\lambda t^2+r^2\lambda^2+r^2t^2+4r(r^2-r+1)\lambda t-2r^2\lambda-2r^2t+r^2}}\cdot b,
%\end{align*}
\begin{align*}
    f'(\lambda) 
&= \frac{-r(a-b)}{\lambda^2} + \frac{-2r^2(\lambda+1)(a-b)^2+4r\lambda(ra-b)(a-rb)}{2 \lambda^2 \sqrt{r^2(\lambda + 1)^2 (a-b)^2 - 4r\lambda (ra-b)(a-rb)}} \\
%&= \frac{-2 r(a-b) \sqrt{r^2(\lambda + 1)^2 (a-b)^2 - 4r\lambda (ra-b)(a-rb)} -2r^2(\lambda+1)(a-b)^2+4r\lambda(ra-b)(a-rb)}{2 \lambda^2  \sqrt{r^2(\lambda + 1)^2 (a-b)^2 - 4r\lambda (ra-b)(a-rb)}} \\
%&= \frac{r(b-a) \left[ 2\sqrt{r^2(\lambda + 1)^2 (a-b)^2 - 4r\lambda (ra-b)(a-rb)} -2r(\lambda+1)(b-a) \right] +4r\lambda(ra-b)(a-rb)}{ 2\lambda^2 \sqrt{r^2(\lambda + 1)^2 (a-b)^2 - 4r\lambda (ra-b)(a-rb)}}  \\
%&= \frac{b^2 r(1-t) \left[ 2\sqrt{r^2(\lambda + 1)^2 (t-1)^2 - 4r\lambda (rt-1)(t-r)} -2r(\lambda+1)(1-t) \right] + 4b^2 r \lambda(rt-1)(t-r)}{ 2 b \lambda^2 \sqrt{r^2(\lambda + 1)^2 (t-1)^2 - 4r\lambda (rt-1)(t-r)}}  \\
&= \frac{r(1-t) \left[ 2\sqrt{r^2(\lambda + 1)^2 (t-1)^2 - 4r\lambda (rt-1)(t-r)} -2r(\lambda+1)(1-t) \right] + 4 r \lambda(rt-1)(t-r)}{ 2  \lambda^2 \sqrt{r^2(\lambda + 1)^2 (t-1)^2 - 4r\lambda (rt-1)(t-r)}} \cdot b \\
\end{align*}
where $t = a/b$. Let 
\[g_\lambda(t) = 2\sqrt{r^2(\lambda + 1)^2 (t-1)^2 - 4r\lambda (rt-1)(t-r)} -2r(\lambda+1)(1-t).\] 
Then
\begin{align*}
    g_\lambda'(t) &= 2r \Bigl( \frac{r(\lambda+1)^2(t-1) -2r\lambda(t-r)-2\lambda(rt-1)}{\sqrt{  r^2 (\lambda+1)^2  (t-1)^2 - 4 r \lambda (rt-1)(t-r)  }} + \lambda + 1   \Bigr).
\end{align*}
The quadratic equation in $t$ inside the radical is minimized at 
\begin{align*}
    t = \frac{ (\lambda+1)^2 r - 2\lambda (r^2+ 1)}{(\lambda- 1)^2 r } = 1 - \frac{2\lambda (1-r)^2}{r(\lambda-1)^2} \leq 1 - \frac{2 r^{-1} (1-r)^2}{r(r^{-1}-1)^2} = -1
\end{align*}
%(The derivative with respect to $\lambda$ is always positive, so $\lambda = r^{-1}$ will maximize the quantity.)  
Since this value for $t$ is negative, the quadratic is smallest at $t = \delta$.  With $t=\delta$, the quadratic as a function of $\lambda$ is minimized at 
\[
\lambda = \frac{2(r\delta-1)(\delta-r)}{r(\delta-1)^2}-1=1-\frac{2(1-r)^2\delta}{r(1-\delta)^2}\,.\]
Since this is strictly below $1$, the minimum over the interval $[1,1/r)$ is achieved at $\lambda=1$. The resulting minimum is thus
    \begin{align}\label{radicalLB}
    4r^2   (1-\delta)^2 - 4 r  (r\delta-1)(\delta-r)=4\delta r(1-r)^2 >0.
    \end{align}
This yields
\begin{align*}
    \abs{g_\lambda'(t)} 
    &\le
    2r \Bigl( \frac{r(r^{-1}+1)^2\delta^{-1} +2r+2r^{-1}}{\sqrt{  4\delta r(1-r)^2 }} + r^{-1} + 1   \Bigr)\\
    &=
    \frac{(1+r)^2\delta^{-1} +2r^2+2}{(1-r)\sqrt{\delta r}} + 2(1 + r) 
    = C(\delta,r) = C
\end{align*}
 for all $t\in[\delta,1/\delta]$ and $\lambda\in(1,1/r)$.

Since $g_\lambda(r) = 0$ the Mean Value Theorem implies that
$g_\lambda(t)=g_\lambda'(s)(t-r)$ for some $s$ between $t$ and $r$. In particular,
since $\delta<r$, $s\in[\delta,1/\delta]$, and $\abs{g_\lambda(r)}\le C\abs{t-r}$.
Returning to $f'(\lambda)$ we get
\begin{align*}
\abs{f'(\lambda)} &= \Bigl|\,\frac{r(1-t) g_\lambda(t) + 4 r \lambda(rt-1)(t-r)}{ 2  \lambda^2 \sqrt{r^2(\lambda + 1)^2 (t-1)^2 - 4r\lambda (rt-1)(t-r)}} \cdot b\, \Bigr|\\
&\leq \frac{Cr(\delta^{-1}+1)  + 4 (r\delta^{-1}+1)}{ 2  \lambda^2 \sqrt{r^2(\lambda + 1)^2 (t-1)^2 - 4r\lambda (rt-1)(t-r)}} \cdot \abs{t-r}b \\
&\leq \frac{Cr(\delta^{-1}+1)  + 4 (r\delta^{-1}+1)}{ 4(1-r)\sqrt{\delta r}} \cdot \abs{a-rb}\le 2C_0\abs{a-rb},
\end{align*}
where in the second-to-last inequality we used $\lambda\ge1$ and the lower bound \eqref{radicalLB} on the expression under the radical.

Now, if $a \neq rb$, then
\begin{align*}
\pxi_-^\lambda(a,b) - \pxi(a,b) 
%&= \frac{r(a-b)\lambda + \lambda f(\lambda)}{2(a-rb)\lambda} - \frac{r(a+b) + (r+1)\sqrt{rab}}{a+rb+2\sqrt{rab}} \\
%&= \frac{r(a-b) +  f(\lambda)}{2(a-rb)} - \frac{r(a+b) + (r+1)\sqrt{rab}}{a+rb+2\sqrt{rab}} \\
&= \frac{(a+rb+2\sqrt{rab}) (r(a-b) +  f(\lambda)) - 2(a-rb) (r(a+b) + (r+1)\sqrt{rab}) }{2(a-rb)(a+rb+2\sqrt{rab})}\\
&= \frac{f(\lambda)-f(1)}{2(a-rb)}\,.
\end{align*}
By the Mean Value Theorem, $f(\lambda) = f(1) + f'(c) (\lambda - 1)$ for some $c \in (1, \lambda)$. 
In particular, $c\in(1,1/r)$.
Therefore, $\abs{f(\lambda)-f(1)} \le 2C_0\abs{a-rb}(\lambda - 1)$ and
\[
\pxi_-^\lambda(a,b) - \pxi(a,b) \geq -C_0(\lambda - 1).
\]

If, on the other hand, $a = rb$, then 
\begin{align*}
\pxi_-^\lambda(a,b) - \pxi(a,b) 
&= \frac{r+1}{\lambda + 1} - \frac{r(a+b) + (r+1)\sqrt{rab}}{a + rb + 2\sqrt{rab}} \\
%&= \frac{r+1}{\lambda + 1} - \frac{r(rb+b) + (r+1)\sqrt{rrbb}}{rb + rb + 2\sqrt{rrbb}} \\
%&= \frac{r+1}{\lambda + 1} - \frac{r(r+1)b + (r+1) rb}{4rb} \\
%&= (r+1) \left(\frac{1}{\lambda + 1} - \frac{1}{2} \right) \\
%&= \frac{(r+1)(2-\lambda-1)}{2(\lambda + 1)} \\
&= -\frac{(r+1)(\lambda - 1)}{2(\lambda + 1)} 
\geq -(r+1)(\lambda - 1)\ge-C_0(\lambda-1)
\end{align*}
and the claim holds again.    	
\end{proof}

\begin{lemma}\label{lem:parameterDirectionEquivalence}  Let $0<\delta<1$.   
Let
    \[C_1=C_1(\delta,r)=\frac{(1+\delta)^2r(1-r)}{2\delta^2 \sqrt{r} (1+\sqrt{r})^2}\quad\text{and}\quad
    C_2=C_2(r)=\frac{2(r+1)^2}{r(1-r)}\,.\]
Then for all $\xi,\zeta\in S_\delta$ we have
		\begin{align}\label{dpxi}
			\abs{\pxi(\xi) - \pxi(\zeta)} \leq C_1\Bigl|\,\frac{\xi_1}{\abs{\xi}_1}-\frac{\zeta_1}{\abs{\zeta}_1}\Bigr|.
		\end{align}
And for all $\xi\in\RR^2_{>0}$ and $q\in(r,1)$ we have
        \begin{align}\label{dxip}
			\Bigl|\,\xip(q)\cdot \ek_1 - \frac{\xi_1}{\abs{\xi}_1}\,\Bigr| \leq C_2 \abs{q-\pxi(\xi)}.
		\end{align}
	\end{lemma}

\begin{proof}
    %We have
	%\[\pxi(t, 1-t) = \frac{r + (r+1)\sqrt{rt(1-t)}}{t+r(1-t)+2\sqrt{rt(1-t)}}\,.\]
	Note that $(t,1-t)\in S_\delta$ if and only if $t\in(\frac{\delta}{1+\delta},\frac1{1+\delta})$. For such $t$, 
	\[
	  %-\frac{r(1-r)}{2\sqrt{rt(1-t)} \left(t+r(1-t)+2\sqrt{rt(1-t)}\right)} \geq 
	  \frac{d}{dt}\pxi(t, 1-t) = -\frac{r(1-r)}{2\sqrt{rt(1-t)} \left(\sqrt{t}+\sqrt{r(1-t)}\right)^2} \geq 
	  -\frac{(1+\delta)^2r(1-r)}{2\delta^2 \sqrt{r} (1+\sqrt{r})^2}= -C_1(\delta,r). 
	\]
	\eqref{dpxi} follows from this bound and the fact that 
	$\pxi(\xi)=\pxi(c\xi)$ for any $\xi\in\RR^2_{>0}$ and $c>0$.
%	
%	Then by the Mean Value Theorem, 
%	\[
%	\abs{\pxi(\xi) - \pxi(\zeta)}= \abs{\pxi(\xi/\abs{\xi}_1) - \pxi(\zeta/\abs{\zeta}_1)}\leq C_1\left|\frac{\xi_1}{\abs{\xi}_1}-\frac{\zeta_1}{\abs{\zeta}_1}\right|.
%	\]

	For the second claim, differentiate $\xip(q)\cdot \ek_1$ to get 
	%$\xi_1(q) = \frac{r(1-q)^2}{q^2(r+1) - 4qr + r(r+1)}$ and
	\[
	 -\frac{2r(1-r)(1-q)(q-r)}{\bigl((q^2(r+1) - 4qr + r(r+1)\bigr)^2} 
	 \geq -\frac{2r(1-r)^3 }{\bigl((q^2(r+1) - 4qr + r(r+1)\bigr)^2}
	%$a = \frac{2r}{r+1}$ minimizes the denominator.  
	\ge-\frac{2(r+1)^2}{r(1-r)}=-C_2(r).
	\]
    \eqref{dxip} follows from this bound and the fact that $\xip(\pxi(\xi_1/\abs{\xi}_1))=\xi_1/\abs{\xi}_1$ for all $\xi\in\RR^2_{>0}$.
\end{proof}

The following estimates are immediate from \eqref{xi->p} and \eqref{eq:shape}. 

\begin{lemma}\label{lem:shapeFunctionBounds}  For $x\in \RR_{\ge0}^2$, 
\[
\frac{r}{1-r}|x|_1 \leq \gpp(x) \leq \frac{r+\sqrt{r}}{1-r} |x|_1.
\]
\end{lemma}
%lower bound is trivial. for upper bound use 2\sqrt{x_1x_2}\le x_1+x_2
\begin{lemma}\label{lem:parameterBounds}  
For any $\delta\in(0,1)$ and $x\in S_\delta$ we have 
    \[r+\frac{(1-r)\delta\sqrt{r\delta}}{(1+\sqrt{r})^2}\le\pxi(x)\le1-\frac{(1-\sqrt{r})\delta}{1+\sqrt{r}}\,.\]
\end{lemma}

\begin{proof}
Let $t=x_2/x_1$. Then 
    \[\pxi(x)
    =\frac{r(1+t)+(r+1)\sqrt{rt}}{1+rt+2\sqrt{rt}}
    =r+\frac{(1-r)(rt+\sqrt{rt})}{1+rt+2\sqrt{rt}}
\ge r+\frac{(1-r)(r\delta+\sqrt{r\delta})}{(1+\sqrt{r/\delta})^2}
\ge r+\frac{(1-r)\delta\sqrt{r\delta}}{(1+\sqrt{r})^2}\,.\]
Similarly,
\[\pxi(x)=1-\frac{(1-r)(1+\sqrt{rt})}{(1+\sqrt{rt})^2}\le1-\frac{(1-r)(1+\sqrt{r\delta})\delta}{(\sqrt\delta+\sqrt{r})^2}\le1-\frac{(1-r)\delta}{(1+\sqrt{r})^2}\,.\qedhere\]
\end{proof}

We are now ready to prove Theorem \ref{geometricExitPointTailBound}.

\begin{proof}[Proof of Theorem \ref{geometricExitPointTailBound}]  
Fix $\delta\in(0,r)$, $\kappa>0$, and
$\e\in(0,\frac\delta{2(1+\delta^{-1})})$. 
Take integers $(m,n)\in S_\delta$ and $k\le\e(m+n)$ and write
%notice that %$\gpp(m-k,n) = \frac{1}{1-r} (r(m+n-k) + 2\sqrt{rn(m-k)})$, so that
%\[
%\gpp(m-k,n)+m-k = \frac{1}{1-r} (m + rn - k + 2 \sqrt{r n (m-k)}).
%\]
%use Lemma \ref{lem:shapeFunctionBounds} and note that $(m+n)$ and $\sqrt{mn}$ have the same order (up to a constant factor depending only on $\delta$) inside of the cone $S_\delta$.  Below, all constants $C$ are positive and only depend on $r, \delta$.  
\begin{align}
&\pxi(m-k,n) - \pxi(m,n) =  \frac{r(m+n-k) + (r+1)\sqrt{rn(m-k)}}{m-k + rn + 2\sqrt{rn(m-k)}} - \frac{r(m+n) + (r+1) \sqrt{rmn}}{m + rn + 2 \sqrt{rmn}}\notag \\
&\qquad= \frac{(1-r)\left(rnk + \sqrt{rn} (rn-m) (\sqrt{m}-\sqrt{m-k})  + k \sqrt{rmn}\right) }{(m - k + rn + 2\sqrt{rn(m-k)})( m + rn + 2\sqrt{rmn})} \notag\\
&\qquad= \frac{rnk + \sqrt{rn} (rn-m) (\sqrt{m}-\sqrt{m-k})  + k \sqrt{rmn}  }{(\gpp(m-k,n) + m-k)(m+rn+2\sqrt{rmn})}\,.\label{extra}
\end{align}
%Note that if $rn -m <0$, then since 
%$2k\le2\e(m+n)\le m+rn$ 
%%2\e\le1 and 2\e\le\delta\le r
%we have $m^2+2m^2-4mk\ge r^2n^2-2rmn$ and therefore 
%$2\sqrt{m(m-k)}\ge m-rn$ and consequently \note{C: I'm probably missing something easy, but I couldn't follow this.}
% \[(m-rn)(\sqrt{m}-\sqrt{m-k})\le\frac{k(m-nr)}{2\sqrt{m-k}}\le k\sqrt m\,.\]
If $m-rn>0$, dropping the $\sqrt{m-k}$ from the denominator and using $rn>0,$ we have
 \[(m-rn)(\sqrt{m}-\sqrt{m-k})=\frac{k(m-rn)}{\sqrt{m}+\sqrt{m-k}}\le k\sqrt m\,.\]
The same inequality holds trivially if $m-rn\le0$.
In the next computation use the above inequality to bound the numerator of 
\eqref{extra}, then bound the denominator using the upper bound from Lemma \ref{lem:shapeFunctionBounds} and the facts that $2\sqrt{mn}\le m+n$ and $(m,n)\in S_\delta$:
\begin{align}\label{fifi1}
\pxi(m-k,n) - \pxi(m,n) 
&\geq \frac{(1-\sqrt r)rnk}{2(1+\sqrt r)(m+n)^2}\ge  \frac{(1-\sqrt r)rk}{2(1+\sqrt r)(\delta^{-1}+1)(m+n)}
%denominator \le (\frac{r+\sqrt r}{1-r}(m-k+n)+m-k)(m+n+2\sqrt{mn})
%\le(\frac{r+\sqrt r}{1-r}(m+n)+m+n)(m+n+m+n)
%= \frac{C_3(\delta, r) k}{m+n}\,.
=\frac{a_0(\delta,r)k}{m+n}\,.
%=\frac{a_0(\delta,r)s}{m+n}\,.
\end{align}
%\firasnote{There is also an upper bound that is commented out since it does not seem to be needed.}
%For the upper bound note that 
%    \[\frac{n}{m-k}\le\frac{n}{m-\e(n+m)}\le\frac1{\delta-\e(1+\delta^{-1})}\le\frac2\delta\,.\]
%Use this, the lower bound in Lemma \ref{lem:shapeFunctionBounds}, and that $(m,n)\in S_\delta$ to write
%\begin{align*}
%\pxi(m-k,n) - \pxi(m,n)
%&\leq \frac{rnk + rn\sqrt{rn} (\sqrt{m}-\sqrt{m-k})  + k \sqrt{rmn}  }{C(r)n^2} 
%%denominator \ge (C(r)(m-k+n)+m-k)(\sqrt{m}+\sqrt{rn})^2\ge rC(r)n^2
%\leq \frac{rnk + rn\sqrt{rn}\cdot\frac{k}{2\sqrt{m-k}}  + nk \sqrt{r\delta^{-1}}  }{C(r)n^2} \\
%&\leq \frac{(\delta^{-1}+1)\bigl(r+r\sqrt{r}\sqrt{2\delta^{-1}}/2+\sqrt{r\delta^{-1}}\bigr)k}{C(r)(m+n)}\,. 
%%=\frac{C_4(\delta, r) k}{m+n}\,.
%\end{align*}
%The upshot is that
%\begin{align}\label{fifi1}
%%\pxi(m-k,n) - \pxi(m,n) \in \left[ \frac{a_0 s}{(m+n)^{1/3}}\,,\, \frac{A_0 s}{(m+n)^{1/3}} \right]
%\pxi(m-k,n) - \pxi(m,n) \in \left[ \frac{a_0 k}{m+n}\,,\, \frac{A_0 k}{m+n} \right]
%\end{align}
% for some finite positive $a_0=a_0(\delta, r)$ and $A_0=A_0(\delta, r)$ and all integers $(m,n)\in S_\delta$ and $k\le\e(m+n)$.

Next, take
$\xi\in\ri\sU$ with $\xi_1\in(\delta,1-\delta)$ and such that $\abs{\xi_1-\frac{m}{m+n}}\le\kappa(m+n)^{-1/3}$. 
Abbreviate $q = \pxi(\xi)$.
Note that $\xi\in S_\delta$ and therefore 
%$$\xi\in S_{\delta/(1-\delta)}\subset S_\delta$.
Lemma \ref{lem:parameterDirectionEquivalence} implies that
    \begin{align}\label{fifi2}
    \abs{q - \pxi(m,n)} \leq C_1\kappa(m+n)^{-1/3}.
    \end{align}

Let 
$s_0=s_0(\delta,r,\kappa)=\max(1,16C_1\kappa/a_0)$. %,C_1\kappa/A_0)$.
%Then $a_0 s - C_1\kappa > \frac{1}{2} a_0 s$ and $A_0 s + C_1\kappa < 2 A_0 s$, for all $s \geq s_0$. 
Let
    \[\epsilon=\epsilon(\delta,r)=\min\Bigl(\frac{r}2,\frac{(1-r)\delta\sqrt{r\delta}}{(1+\sqrt{r})^2}\,,\,
\frac{(1-\sqrt{r})^2\delta^2}{2(1+\sqrt{r})^2}\,\Bigr).\]
%the r/2 is so that \epsilon<r to apply Lm \ref{lem:logMGFExpansion} the other necessary bounds on \epsilon are verified below. Namely, that r+\epsilon\le1-\epsilon 
%and \epsilon<1-\pxi, which implies it is <1-q and <1-\pxi_-
Take $\eta$ so that 
\begin{align}\label{eta}
\begin{split}
0<\eta<\min\Bigl\{-\frac{\log r}\e\,,\,
&\frac1{2\e}\log\Bigl(1+\frac{(1-\sqrt{r})\delta}{1+\sqrt{r}}\Bigr)\,,\,
1\,,\,\frac{a_0}{32(1+C_0(\delta/2,r))}\,,\\
&\qquad\qquad\e^{-1}\log\Bigl(1+\frac\epsilon{\sqrt2}\Bigr)\,,\,\frac{\e\delta^2 a_0^2}{100\bigl(4\epsilon^2+4a_0/5+16\epsilon)}\Bigr\}\,. %,A_0)$.
\end{split}
\end{align}
%and
%\begin{align}\label{eta2}
%4(\epsilon^2+a_0/5)\eta+16\epsilon\eta^3\le \e\delta ^2a_0^2/100.
%\end{align}
Let $N_0=(s_0/\e)^3$ and take $(m,n)\in S_\delta\cap\bbZ^2_{\ge N_0}$. 
%This guarantees that the interval from which we will now pick $s$ is not empty.
Take $s\in[s_0,\e(m+n)^{1/3}]$ and set 
\begin{align}\label{lambda}
\lambda = \exp\{ \eta s (m+n)^{-1/3} \}\le e^{\eta\e}\in(1,1/r).
\end{align}
Then by Lemma \ref{lem:parameterBounds} and the choice of $\eta$ 
%in \eqref{eta}
    \begin{align}\label{lam-q}
    r+\epsilon\le r+\frac{(1-r)\delta\sqrt{r\delta}}{(1+\sqrt{r})^2}\le q\le\lambda q\le\lambda^2 q\le e^{2\eta\e}\Bigl(1-\frac{(1-\sqrt{r})\delta}{1+\sqrt{r}}\Bigr)
    \le 
    %\Bigl(1+\frac{(1-\sqrt{r})\delta}{2(1+\sqrt{r})}\Bigr)\Bigl(1-\frac{(1-\sqrt{r})\delta}{2(1+\sqrt{r})}\Bigr)=
    1-\frac{(1-\sqrt{r})^2\delta^2}{4(1+\sqrt{r})^2}\le 1-\epsilon.
    \end{align}

Since $\eta s(m+n)^{-1/3}\le \eta\e\le1$  and $e^x - 1 \leq 2x$ for $ x \in [0,1]$,
\begin{align}\label{fifi3}
%\abs{\pxi(m,n) - \lambda q}
%&\leq \abs{\pxi(m,n) - q} + (\lambda-1)q
%\leq \frac{C_1\kappa}{(m+n)^{1/3}} +  (e^{\eta s (m+n)^{-1/3}} - 1)q 
%\leq \frac{C_1\kappa+2\eta s}{(m+n)^{1/3}}\,.
\lambda q-q\le\lambda-1\le\frac{2\eta s}{(m+n)^{1/3}}\,.
\end{align}

The choices of $\e$ and $k$ and that $(m,n)\in S_\delta$ imply that 
$(m-k,n)\in S_{\delta/2}$.
%m-k\le m\le n/\delta\le 2n/\delta and m-k\ge \delta n-\e(n+m)\ge\delta n-\e n-\e/\delta n=(\delta-\e-\e/\delta)n\ge n\delta/2
Then Lemma \ref{lem:parameterExpansion} implies
\begin{align}\label{fifi4}
\pxi_-^\lambda(m-k,n) - \pxi(m-k,n) \geq -C_0(\delta/2,r) (\lambda - 1)\ge-\frac{2C_0\eta s}{(m+n)^{1/3}}\,.
\end{align}

Take $k=\lceil s(m+n)^{2/3}\rceil-1$ and 
abbreviate $\pxi_\pm=\pxi_\pm^\lambda(m-k,n)$ and $\pxi=\pxi(m-k,n)$.
Note that 
\begin{align}\label{k>s_0/2}
s_0/2\le 2^{2/3}s_0-1 \le s(m+n)^{2/3}-1\le k\le s(m+n)^{2/3}\le\e(m+n).
\end{align}
%used that s_0\ge1\ge1/(2^{2/3}-1/2)
Putting this, (\ref{fifi1}-\ref{eta}), %\eqref{fifi2}, \eqref{eta}, 
and (\ref{fifi3}-\ref{k>s_0/2}) %, \eqref{fifi4}, and \eqref{k>s_0/2} 
together we get
\begin{align}
    \pxi-q
    \ge\pxi_-^\lambda-q
    \ge\pxi_-^\lambda-\lambda q
    &\ge\frac{(a_0-2\eta-2C_0\eta)s-C_1\kappa}{(m+n)^{1/3}}-\frac{a_0}{m+n}\notag\\
%    &\ge\frac{(a_0-a_0/16)s-a_0s_0/16}{(m+n)^{1/3}}-\frac{a_0s_0}{(m+n)^{2/3}(m+n)^{1/3}}\\
%    &\ge\frac{(a_0-a_0/16)s-a_0s/16}{(m+n)^{1/3}}-\frac{a_0s}{2^{2/3}(m+n)^{1/3}}\\
    &\ge\frac{7a_0 s}{8(m+n)^{1/3}}-\frac{a_0s}{2^{2/3}(m+n)^{1/3}}\notag\\
    &\ge\frac{a_0s}{5(m+n)^{1/3}}>0.\label{p-q}
\end{align}
Thus, by Lemma \ref{lem:restrictedLogMGFProperty}, $\mc{L}^{\lambda, \lambda q}(m-k, n) = L^{\pxi_-^\lambda, \pxi_+^\lambda}(m-k,n)$. %=  \mc{L}^{\lambda}(m-k,n)$.
%
%From \eqref{fifi1}, \eqref{fifi2}, \eqref{k>s_0/2}, and the choice of $s_0$ we get
%\begin{align*}
%\pxi - q 
%%&= (\pxi(m-k,n) - \pxi(m,n)) + (\pxi(m,n) - q) 
%\geq \frac{a_0 s-C_1\kappa}{(m+n)^{1/3}}-\frac{a_0}{m+n} 
%\geq \frac{7a_0 s}{8(m+n)^{1/3}}-\frac{a_0s}{2^{2/3}(m+n)^{1/3}} 
%\geq \frac{a_0 s}{5(m+n)^{1/3}}\,.
%\end{align*}
%This, \eqref{fifi4}, and the choice of $\eta$ in \eqref{eta} give
%%and  $e^x - 1 \geq x$ give
%\begin{align*}
%\pxi_-^\lambda- q 
%%&= \pxi_-^\lambda(m-k,n) - \pxi(m-k,n) + \pxi(m-k,n) - q \\
%%&\geq  -C_0(\delta/2,r)(\lambda - 1)  + \frac{a_0 s}{5(m+n)^{1/3}} \\
%\geq  -\frac{2C_0\eta s}{(m+n)^{1/3}} + \frac{a_0 s}{5(m+n)^{1/3}} 
%\geq \frac{a_0s}{10(m+n)^{1/3}}\,.
%\end{align*}
Also, Lemma \ref{lem:parameterBounds} and the choice of $\eta$ in \eqref{eta} imply
    \begin{align*}
    r+\epsilon\le r+\frac{(1-r)\delta\sqrt{r\delta}}{(1+\sqrt{r})^2}\le  q\le \pxi_-^\lambda\le\lambda\pxi_-^\lambda
    \le\lambda\pxi\le e^{\eta\e}\Bigl(1-\frac{(1-\sqrt{r})\delta}{2(1+\sqrt{r})}\Bigr)
    \le 
    %e^{2\eta\e}\Bigl(1-\frac{(1-\sqrt{r})\delta}{2(1+\sqrt{r})}\Bigr)
    1-\frac{(1-\sqrt{r})^2\delta^2}{4(1+\sqrt{r})^2}\le 1-\epsilon.
    \end{align*}
In particular, $\pxi-r\ge q-r\ge\e$ and $1-\pxi\ge1-\lambda\pxi\ge\e$.
Using this, $(m-k,n)\in S_{\delta/2}$, $(m,n)\in S_\delta$, and the identity $a(\pxi(a,b)-r)^2 - rb(1-\pxi(a,b))^2 = 0$, we get
\begin{align}
&(m-k)(\pxi_-^\lambda-r)(q-r) - nr(1-\pxi_-^\lambda)(1-q)\notag\\
&\qquad\leq (m-k)(\pxi - r)(q-r) - nr (1-\pxi)(1-q) \notag \\
%&= (m-k)(\pxi - r)(\pxi - r + q - \pxi) - nr (1-\pxi)(1-\pxi + \pxi-q)\notag\\
&\qquad= (m-k)(\pxi-r)^2 + (m-k)(\pxi-r)(q-\pxi) - nr(1-\pxi)^2 - nr(1-\pxi)(\pxi-q)\notag \\
%&=  (m-k)(\pxi-r)(q-\pxi) - nr(1-\pxi)(\pxi-q) \notag\\
&\qquad= -(\pxi - q)((m-k)(\pxi-r) + nr(1-\pxi))\notag \\
%&\leq -\min\Bigl(\frac{(1-r)\delta^3\sqrt{r\delta}}{4\sqrt2(1+\sqrt{r})^2},\frac{r(1-\sqrt r)\delta}{2(1+\sqrt r)}\Bigr)a_0s(m+n)^{2/3}\notag\\
%m-k\ge \delta n/2\ge m\delta^2/2 and r\ge r^2\ge\delta^2\ge\delta^2/2
&\qquad\leq -\e\delta ^2(\pxi-q)(m+n)/2.\label{fofo3}
\end{align}

%This implies 
%\begin{align*}
%\pxi(m-k,n) - \lambda q 
%&= (\pxi(m-k,n) - \pxi(m,n)) + (\pxi(m,n) - \lambda q)\\
%&\in \left[ \frac{a_0 s - C_1\kappa - 2 \eta s}{(m+n)^{1/3}}, \frac{A_0 s + C_1\kappa + 2 \eta s}{(m+n)^{1/3}} \right]\\
%&\subset \left[ \frac{a_0 s - 4 \eta s}{2(m+n)^{1/3}}, \frac{2A_0 s + 2 \eta s}{(m+n)^{1/3}} \right]
%\subset \left[ \frac{a_0 s}{4(m+n)^{1/3}}, \frac{4A_0 s }{(m+n)^{1/3}} \right].
%\end{align*}
%
%Thus,
%\[
%\pxi(m-k,n) - \lambda q - 3C_0 \log \lambda 
%= \pxi(m-k,n) - \lambda q - 3C_0 \eta s (m+n)^{-1/3} \geq \frac{a_0 s}{4(m+n)^{1/3}} - \frac{3C_0 \eta s}{(m+n)^{1/3}} > 0
%\]
%for $\eta$ small enough.  Therefore, $\pxi(m-k,n) - \lambda q > 3C_0 \log \lambda$.  
%Since $\lambda - 1 \leq 2 \log \lambda$ for $\lambda \in [1,2]$,
%\begin{align*}
%\tilde{\zeta_-^\lambda} - \lambda q &= (\tilde{\zeta_-^\lambda} - \tilde{\zeta}) + (\tilde{\zeta} - \lambda q) \\
%&\geq -C_0 (\lambda - 1) + 3C_0 \log \lambda \\
%&\geq -2C_0 \log \lambda + 3C_0 \log \lambda = C_0 \log \lambda > 0
%\end{align*}
%for $\eta$ small enough.  

%Therefore, there exists constants $c = c(\delta) > 0$ and $N_0 = N_0(\delta, K)$ such that whenever $(x, y) \in S_{\frac{\delta}{2}} \cap \RR_{\geq N_0}^2$, 
%\[
%\text{if } |q - \zeta(x, y)| \leq \frac{K}{(x+y)^{1/3}} \text{ then } q \in (r+c, 1-c).
%\]
%(This follows because $q$ and $\zeta$ can be made arbitrarily close by taking $N_0$ very large relative to the choices of $\delta, K$.)

%Stochastic monotonicity justified earlier (but commented out)
We have now collected all the necessary pieces to be able to bound the 
probability of interest. The first line below uses the stochastic monotonicty of geometric random variables in their inverse mean parameter and the monotonicity of the exit points in the boundary weights.  The second line uses that, on the event in the indicator function the value inside the exponent is 0.  The third line drops the indicator function and uses the Cauchy-Schwartz inequality.  The fourth line uses independence and shift-invariance.  Write
\begin{align*}
&\P\{Z^{q, \ek_1}(m,n) > k\}^2 \leq \P\{Z^{\lambda q, q, \ek_1}(m,n) > k\}^2 \\
&= \E\Bigl[\one\{Z^{\lambda q, q, \ek_1}(m,n) > k\} \exp\Bigl\{ \frac{\log(\lambda)}{2} \bigl(G^{\lambda q}(k,0) + G_{(k,0)}^{\lambda q}(m,n|1,0) - G^{\lambda q, q}(m,n)\bigr) \Bigr\}\Bigr]^2 \\
&\leq \E\bigl[ \exp\bigl\{ \log(\lambda) \bigl(G^{\lambda q} (k,0) + G_{(k,0)}^{\lambda q}(m,n|1,0) \bigr) \bigr\} \bigr] \E\bigl[ \exp\bigl\{ -\log(\lambda) G^{\lambda q, q}(m,n) \bigr\} \bigr] \\
%&= \E\left[ \exp\left\{ \log(\lambda) G^{\lambda q} (k,0) \right\}\right] \E\left[ \exp\left\{ \log(\lambda) G_{(k,0)}^{\lambda q}(m,n|1,0) ) \right\} \right] \E\left[ \exp\left\{ -\log(\lambda) G^{\lambda q, q}(m,n) \right\} \right] \\
&= \E\bigl[ \exp\bigl\{ \log(\lambda) G^{\lambda q} (k,0) \bigr\}\bigr] \E\bigl[ \exp\bigl\{ \log(\lambda) G^{\lambda q}(m-k,n|1,0)  \bigr\} \bigr] \E\bigl[ \exp\bigl\{ -\log(\lambda) G^{\lambda q, q}(m,n) \bigr\} \bigr].
\end{align*}
Bound the third expectation on the last line using Proposition \ref{prop:logMGF}, the second expectation using Lemma \ref{lem:restrictedLogMGFInequality}, and compute the first expectation explicitly using the moment generating function of the Geometric distribution, to get:
%which is $M_{I^q}(t) = \frac{1-q}{1-qe^t}$, 
%since $G^{\lambda q}(k,0)$ is an i.i.d.\ sum.  So we continue after taking the log,
\begin{align*}
2 \log \P\{Z^{q, \ek_1}(m,n) > k\}
&\leq k \log\Bigl( \frac{1-\lambda q}{1-\lambda^2 q}\Bigr) + \mc{L}^{\lambda, \lambda q}(m-k,n) + L^{\lambda q, q}(m,n)\\ %\label{Thm2-5-log-ineq}\\
&= k \log\Bigl( \frac{1-\lambda q}{1-\lambda^2 q}\Bigr) + L^{\pxi_-^\lambda, \pxi_+^\lambda}(m-k,n) - L^{q, \lambda q}(m,n)\\ %\label{Thm2-5-log-ineq}
%\end{align*}
%Notice here the last term is subtracted because we swapped the order of $q$ and $\lambda q$.
%
%%Abbreviate $\tilde{\gpp}=\gpp(m-k,n)$ and similarly for $\tilde{\pxi}$, and $\tilde{\pxi}_-^\lambda$.
%%Next, we want to show that $\lambda q \leq \pxi_-^\lambda(m-k,n)$, so that $\mc{L}^{\lambda, \lambda q}(m-k,n) = \mc{L}^\lambda(m-k,n)$ by Lemma \ref{lem:restrictedLogMGFProperty}.  
%Continue the above bounds with
%\begin{align*}
%2 \log \P(Z^{q, \ek_1}(m,n) > k) 
&= k \log\Bigl( \frac{1-\lambda q}{1-\lambda^2 q}\Bigr) + L^{\pxi_-^\lambda, \pxi_+^\lambda}(m-k,n) - k \log\Bigl(\frac{1-q}{1-\lambda q}\Bigr) \\
%&= -k \log \left(\frac{(1-\lambda^2 q)(1-q)}{(1-\lambda q)^2}\right) + \mc{L}^{\lambda}(m-k,n) - L^{q, \lambda q}(m-k,n) \nonumber\\
%&= -k \log \left( \frac{1 - q - \lambda^2 q + \lambda^2 q^2}{1 - 2\lambda q + \lambda^2 q^2}\right) + \mc{L}^{\lambda}(m-k,n) - L^{q, \lambda q}(m-k,n) \nonumber\\
%&= -k \log \left( 1 - \frac{q + \lambda^2 q -2 \lambda q}{1 - 2\lambda q + \lambda^2 q^2}\right) + \mc{L}^{\lambda}(m-k,n) - L^{q, \lambda q}(m-k,n) \nonumber\\
&= -k \log \Bigl( 1 - \frac{(\lambda - 1)^2q}{(1-\lambda q)^2}\Bigr) + L^{\pxi_-^\lambda, \pxi_+^\lambda}(m-k,n) - L^{q, \lambda q}(m-k,n).
\end{align*}
Next, use \eqref{lambda}, \eqref{lam-q},  and
\eqref{eta} to deduce
    \[\frac{(\lambda-1)^2q}{(1-\lambda q)^2}\le\epsilon^{-2}(e^{\eta\e}-1)^2\le1/2.\]
Use this, the fact that $-\log(1-t) \leq 2t$ for $t \in \left[0, \frac{1}{2}\right]$, \eqref{lam-q}, \eqref{k>s_0/2}, and \eqref{fifi3} to continue with the bound
\begin{align}
2 \log \P(Z^{q, \ek_1}(m,n) > k) 
&\leq 2 k \frac{(\lambda-1)^2 q}{(1- \lambda q)^2} + L^{\pxi_-^\lambda, \pxi_+^\lambda}(m-k,n) - L^{q, \lambda q}(m-k,n) \nonumber\\
&\leq 2 \epsilon^{-2}k (\lambda-1)^2 + L^{\pxi_-^\lambda, \pxi_+^\lambda}(m-k,n) - L^{q, \lambda q}(m-k,n) \nonumber\\
&\le 4\epsilon^{-2}\eta^2 s^3 + L^{\pxi_-^\lambda,\pxi^+_\lambda}(m-k,n) - L^{q, \lambda q}(m-k,n). \label{logProbEta2}
\end{align}
Using Lemma \ref{lem:logMGFExpansion}, \eqref{fofo3}, \eqref{p-q}, and \eqref{fifi3}  we get
\begin{align*}
&L^{\pxi_-,\pxi_+}(m-k,n) - L^{q, \lambda q}(m-k,n) \nonumber\\
&\qquad\leq (\lambda -1)\Bigl(\frac{(m-k) \pxi_-}{1-\pxi_-} - \frac{(m-k)q}{1-q} + \frac{nr}{\pxi_- - r} - \frac{nr}{q-r} \Bigr) \nonumber\\
&\qquad\qquad + \frac{1}{2} (\lambda-1)^2 \Bigl( \frac{(m-k)\pxi_-^2}{(1-\pxi_-)^2} - \frac{(m-k) q^2}{(1-q)^2} - \frac{nr(2\pxi_--r)}{(\pxi_--r)^2} + \frac{nr(2q-r)}{(q-r)^2}  \Bigr) \nonumber\\
%&\qquad + \frac{1}{3} (\lambda-1)^3 \left(\frac{(m-k)(\tilde{\zeta_-^\lambda})^3}{(1-\tilde{\zeta_-^\lambda})^3} - \frac{(m-k)q^3}{(1-q)^3} + \frac{nr(3\tilde{\zeta_-^\lambda} - 3r\tilde{\zeta_-^\lambda} + r^2)}{(\tilde{\zeta_-^\lambda}-r)^3} - \frac{nr(3q-3rq+r^2)}{(q-r)^3}\right) \nonumber\\
%&\qquad + C(\delta, r) (m+n) (\lambda-1)^4 \nonumber\\
&\qquad\qquad + 2\epsilon^{-3} (m+n) (\lambda-1)^3 \nonumber\\
&\qquad= (\lambda-1)(\pxi_--q) \frac{(m-k)(\pxi_--r)(q-r) - nr(1-\pxi_-)(1-q)} {(1-\pxi_-)(1-q)(\pxi_--r)(q-r)} \\ %\label{diffOfLCalc}\\
&\qquad\qquad + \frac{1}{2}(\lambda-1)^2 (\pxi_--q)\Bigl( \frac{(m-k)(\pxi_-+q-2\pxi_- q)}{(1-\pxi_-)^2(1-q)^2} + \frac{nr(-\pxi_- r + 2 \pxi_- q - rq)}{(\pxi_--r)^2 (q-r)^2} \Bigr) \nonumber\\
&\qquad\qquad + 2\epsilon^{-3} (\lambda-1)^3(m+n) \nonumber\\ 
&\qquad\leq -\frac{\e\delta ^2}{2\epsilon^4}(\lambda - 1)(\pxi_- -q)(\pxi-q)(m+n)+ \epsilon^{-4} (\lambda - 1)^2 (\pxi_- -q)(m+n)+2\epsilon^{-3} (\lambda-1)^3(m+n)\\
%\pxi_-\le1, q\le1, r\le1, 1-q and 1-\pxi_-\ge\epsilon, and \pxi_--r \ge q-r \ge\epsilon
&\qquad\le-\Bigl(\frac{\e\delta ^2a_0^2\eta}{50\epsilon^4}-\frac{4\eta^2a_0}{5\epsilon^4}- \frac{16\eta^3}{\epsilon^3}\Bigr)s^3
%used e^x-1\ge x to bound \lambda-1 from below
\le-\Bigl(\frac{\e\delta ^2a_0^2\eta}{50\epsilon^4}-\frac{4\eta^2a_0}{5\epsilon^4}- \frac{16\eta^2}{\epsilon^3}\Bigr)s^3.
\end{align*}

Setting $c_1=c_1(\delta,r)=\epsilon^{-4}\e\delta ^2a_0^2\eta/200$ and 
using \eqref{logProbEta2} and the choice of $\eta$ in \eqref{eta}
we get 
\begin{align*}
\P\{Z^{q, \ek_1}(m,n) \geq s (m+n)^{2/3}\} \leq e^{-c_1 s^3}
\end{align*}
for $s \in [s_0, \e (m+n)^{1/3}]$.
When $s \geq (m+n)^{1/3}$, the above probability is $0$ and the bound holds trivially.  When $s \in [\e (m+n)^{1/3}, (m+n)^{1/3}]$ we have
    \[\P\{Z^{q, \ek_1}(m,n) \geq s (m+n)^{2/3}\} \leq
    \P\{Z^{q, \ek_1}(m,n) \geq \e(m+n)\} \leq
    e^{-c_1\e(m+n)^{1/3}}\le e^{-c_1\e s^3}.\]
The claim of the theorem is thus proved for the case of $Z^{q,\ek_1}$.
The equivalent bound for vertical exit points follows by symmetry. 
\end{proof}

    \section{Verifying Assumption \ref{ass:RW} for the geometric LPP}\label{sec:RW}
%\mathnote{edit this text, since the lemma changed a bit}
%The non-trivial part of our next result can be viewed as saying that if the increments of a centered random walk are perturbed in an independent manner, causing a small negative, then the probability that the perturbed walk remains negative for its first $n$ steps can bounded by the size of the perturbation, provided that one has uniform control over higher moments of the walk and the drift does not converge to zero too quickly. As will be seen from the proof, neither our moment assumption nor our assumption about the speed of the drift going to zero is sharp. These assumptions are used in a union bound to compare the probability that the walk remains negative on a fixed time interval to the probability that the walk remains negative forever, which we show is always bounded by the drift, assuming only a lower bound on the variance and an upper bound on the third absolute central moment.
We first prove a bound on the probability that an i.i.d.~random walk with non-positive drift remains non-positive for its first $n$ steps, given some control over the step's higher moments. 
%The bound is useful when the mean step of the walk is a decaying function of $n$.

%Given $p>2$, $D>0$, define the sequence
%    \begin{align}\label{tndef}
%    t_n=\Bigl(\frac{(2/e)^p p^{p+1}D^2\sqrt{2\pi}}{(p-2)e^{9D^4+\sqrt{2}}}\Bigr)^{1/(p+1)}n^{-\frac{p-2}{2(p+1)}}\,.
%    \end{align}

	\begin{lemma}\label{lem:rwBound}
	Let $\{X_i, i \in \bbN\}$ be i.i.d.\ random variables.  
	Suppose $\mu=E[X_1] \le0$ and $\sqrt{\text{Var}(X_1)}\ge\e$ and $E[|X_1- \mu|^p]\le D$
	for some $p\ge3$ and $\e,D\in (0,\infty)$.
	Call $S_k = \sum_{i=1}^k X_i$. There exists a finite $C=C(p,D,\e)>0$ such that for all 
	$n \in \bbN$,
	\begin{align}
		&P(S_1 \leq 0, S_1 \leq 0, \ldots, S_n \leq 0) \leq C \bigl(n^{-\frac{p-2}{2(p+1)}}\vee\abs{\mu}\bigr)
		\quad\text{and}\label{eq:rwbd1} \\
		&P(S_1 \geq 0, S_2 \geq 0, \ldots, S_n \geq 0) \leq \frac{C}{\sqrt{n}}. \label{eq:rwbd2}
	\end{align}
	\end{lemma}
	\begin{proof}
Since $\mu\le0$ the probability in \eqref{eq:rwbd2} is bounded above by
the probability of $\{S_1\ge\mu,S_2\ge2\mu,\dotsc,S_n\ge n\mu\}$.
%To see \eqref{eq:rwbd2}, couple $S_{k,n}$ to the centered walk $S_{k,n} - \mu_n k$, which stochastically dominates it. 
The bound \eqref{eq:rwbd2} then follows from Theorem 5.1.7 in \cite{Law-Lim-10}.

Let $n \in \bbN$ be sufficiently large that $t_n=n^{-\frac{p-2}{2(p+1)}} < \epsilon$ and let $\nu_n=(\mu+t_n)^+$. Note that $\nu_n-\mu=t_n\vee\abs{\mu}>0$. 
Let $\overline S_{k,n}=S_k-k\nu_n$.  Then 
    \[P(S_1\leq 0, S_2 \leq 0, \ldots, S_n \leq 0) \le 
    P(\overline S_{1,n}\leq 0, \overline S_{2,n} \leq 0,\ldots, \overline S_{n,n} \leq 0). \]
For $k\in \bbN$ define $p_{k,n} = P(\overline S_{1,n} \leq 0, \overline S_{2,n} \leq 0, \ldots, \overline S_{k-1,n}\leq 0, \overline S_{k,n} >0)$ and $\tau_n = \inf\{k : \overline S_{k,n} > 0\}$. For $s \in [0,1],$ set
	\begin{align*}
	    p_n(s) &=\sum_{k=1}^\infty s^k p_{k,n} 
	\end{align*}
	and observe that $P(\tau_n=\infty) = 1-p_n(1)$.
	By the Sparre-Andersen Theorem, Theorem XII.7.1 in \cite{Fel-71}, for $s\in [0,1),$
	\begin{align}
	\log \frac{1}{1-p_n(s)} &= \sum_{k=1}^\infty \frac{s^k}{k} P(\overline S_{k,n}>0).\label{eq:feller}
	\end{align}
    Denote by $\Phi(x)$ and $\phi(x)$ the cumulative distribution function and probability density function of a standard Normal random variable. Recall that $p \geq 3$. By the Berry-Esseen theorem \cite[Theorem 3.4.9]{Dur-10}, for all $x \in \bbR$ and $k\in \bbN$,
	\begin{align*}
    \Bigl|P\Bigl\{\frac{\overline S_{k,n}-\mu k+\nu_n k}{\e\sqrt{k}} \leq x\Big\} - \Phi(x) \Bigr| \leq \frac{3D^{3/p}}{\e^3\sqrt{k}}
	\end{align*}
	%Berry-Esseen says that the normalized CDF minus the normal CDF is bounded by \frac{3E[|X_{1,n} - \mu_n|^3]}{\e^3 \sqrt{n}}.
		Set $m_n = \frac{\e\sqrt{2\pi}}{2(\nu_n-\mu)}=\frac{\e\sqrt{2\pi}}{2(t_n\vee\abs{\mu})} > 0$.
		Since $\phi$ is decreasing on $[0,\infty)$ with $\phi(0) = 1/\sqrt{2\pi}$,
\begin{align*}
    1-\Phi\Bigl((\nu_n-\mu)\frac{\sqrt{k}}{\e}\Bigr) &= \frac{1}{2} - \int_0^{(\nu_n-\mu) \frac{\sqrt{k}}{\e}} \phi(x)dx \geq \frac{1}{2} - \frac{(\nu_n-\mu) \sqrt{k}}{\e\sqrt{2\pi}}  = \frac{1}{2} - \frac{\sqrt{k}}{2m_n}\,.
\end{align*}
Then we have
	\begin{align*}
	        \sum_{k=1}^\infty \frac{1}{k} P(\overline S_{k,n}>0) &\geq \sum_{1\le k<m_n^2} \frac{1}{k} P(\overline S_{k,n}>0) \geq \sum_{1\le k<m_n^2} \frac{1}{k} \Bigl(1-\Phi\Bigl(-(\mu-\nu_n)\frac{\sqrt{k}}{\e}\Bigr) -  
	        \frac{3D^{3/p}}{\e^3\sqrt{k}}\Bigr) \\
	        &\geq \sum_{1\le k<m_n^2} \frac{1}{2k} - \frac{1}{2m_n } \sum_{1\le k<m_n^2}\frac{1}{\sqrt{k}} - \frac{3D^{3/p}}{\e^3} \sum_{k=1}^\infty \frac{1}{k^{3/2}} \\
	        & \geq \log m_n - 1 - \frac{9D^{3/p}}{\e^3},
	\end{align*}
	where the empty sum is, as usual, equal to $0$.
	When $m_n>1$, each of the three bounds in the last line comes from integral comparison. The bounds are trivial when $0<m_n<1$.
	%\sum_{1\le k<a}\frac1{2k} \ge \int_1^a \frac1{2x}\,dx
	%=(\log a)/2.  The first inequlaity holds by integral comparison when a>1. It holds trivially when 0<a\le 1.
	%
	%\sum_{1\le k<a} \frac1{\sqrt k} \le 1+\sum_{k=2}^{\fl{a}} \frac1{\sqrt k} \le 1+\int_1^a\frac1{\sqrt{x}}\,dx = 1+2(\sqrt a-1)=2\sqrt a-1\le 2\sqrt a.
	%This works for all a>1. 
	%It also works trivially when 0<a\le1, since the sums are empty but the rhs is >0. 
	%
	It then follows from \eqref{eq:feller} that
	\begin{align*}
	    P(\tau_n = \infty) &\leq \frac{2e^{\frac{9D^{3/p}}{\e^3}+1}}{\e\sqrt{2\pi}}(t_n\vee\abs{\mu}).
	\end{align*}
 	On the other hand, by the Markov and Burkholder-Davis-Gundy \cite[Theorem 4.2.12]{Bic-02} inequalities followed by integral comparison,
	\begin{align*}
	    P(\exists k \geq n : \overline S_{k,n} > 0) &\leq \sum_{k=n}^\infty P(\overline S_{k,n} > 0) \leq \sum_{k=n}^{\infty} P(|S_k-k\mu|^p > |k(\mu-\nu_n)|^p) \\
	    &\leq (2/e)^{p/2}p^p D \sum_{k=n}^\infty \frac{1}{k^{p/2} |\mu-\nu_n|^p} \leq \frac{2 (2/e)^{p/2}p^p D}{p-2} n^{1-p/2} (t_n\vee\abs{\mu})^{-p}.	
	   \end{align*}
    Combining these results, we have
    \begin{align*}
        P\left(S_1 \leq 0, \dots, S_n \leq 0\right) &\leq P(\tau_n = \infty) + P(\exists k \geq n \text{ such that } \overline S_{k,n} > 0) \\
        &\leq \frac{2e^{\frac{9D^{3/p}}{\e^3}+1}}{\e \sqrt{2\pi}}(t_n\vee\abs{\mu})+\frac{2 (2/e)^{p/2}p^p D}{p-2} n^{1-p/2} (t_n\vee\abs{\mu})^{-p}.
    \end{align*}
    Note that $n^{1-p/2}t_n^{-p}=t_n$ and, if $t_n\le\abs{\mu}$, then we have 
    $n^{1-p/2}\abs{\mu}^{-p}=t_n^{1+p}\abs{\mu}^{-p}\le\abs{\mu}$; in this case, we have $n^{1-p/2}(t_n \vee |\mu|)^{-p} = |\mu|= t_n \vee |\mu|$. On the other hand, if $t_n \geq |\mu|$, then $t_n = n^{1-p/2}t_n^{-p} \leq n^{1-p/2}|\mu|^{-p}$ and, consequently, $n^{1-p/2}(t_n \vee |\mu|)^{-p}= (n^{1-p/2}t_n^{-p}) \wedge (n^{1-p/2}|\mu|^{-p}) = t_n= t_n \vee |\mu|$. Bound \eqref{eq:rwbd1} follows.
        	\end{proof}

\begin{lemma}\label{lm:assRWholds}
If $\omega_0 \sim \Geom(r)$, then Assumption \ref{ass:RW} holds for any  $\expnt\in(1/3,2/3)$.
\end{lemma}

\begin{proof}
        Fix $\expnt\in(1/3,2/3)$.
		The steps of the random walk $S_k^{\xi_\star,\eta^\star}$ for $k\in\flint{0, N^{2/3} }$ are i.i.d.\ differences of independent geometric random variables with parameters $r/\pxi(\xi_\star)$ and $r/\pxi(\eta^\star)$. %(the parameters of the $J$ weights).  
		%Since a larger $\ek_1$-coordinate of $\xi$ corresponds to stochastically larger $J$-weights, the steps of $S^{\xi_\star,\eta^\star}$ have a negative mean.  	
		%By Lemma \ref{lem:parameterDirectionEquivalence}, the parameters differ by at most $C_9 N^{-\expnt/2}$.  %$C_9 N^{-a(\fluctexp)/2}$. 
		%On the other hand, under 
		Under the conditions of Assumption \ref{ass:RW} on $\xi_\star$ and $\eta^\star$, 
		Lemma \ref{lem:parameterDirectionEquivalence} implies that 
		    \[-C_1r^{-1} N^{-\expnt/2}\le\mu=\E[S_1^{\xi_\star,\eta^\star}]=\frac{\pxi(\xi_\star)}{r}-\frac{\pxi(\eta^\star)}{r}\le0.\]
		%The steps of $S^{\xi_\star,\eta^\star}$ have mean at least $-C_{10} N^{-\expnt/2}$,
		%$-C_{10} N^{-a(\fluctexp)/2}$, 
		%since the parameters are bounded away from $r$ and $1$ by constants depending on $\delta$.  
		%Since we are interested in an upper bound on the probability $S^{\xi_\star,\eta^\star}$ remains non-positive through time $N^{2/3}$, it suffices to consider a random walk with steps that have mean  $-C_1 N^{-\expnt/2}$.
		%$-C_9 N^{-a(\fluctexp)/2}$.
		%S\le 0 implies S-a\le -a\le0  so a random walk with a more negative
		%drift gives an upper bound. 
		
		%var of increment = (r/\pxi(\xi_*))/(1-r/\pxi(\xi_*))^2+(r/\pxi(\eta^*))/(1-r/\pxi(\eta_*))^2\ge \e(\delta)  since \xi_* and \eta^* have coordinates in (\delta,1-\delta) and therefore their \pxi stays away from r.
		%p-norm is bounded by the p-norm of J plus the p-norm of J' plus the two means.
		%All these are bounded by the norms and means of a geometric(r/
        
        Since $\pxi(\xi_\star)$ and $\pxi(\eta^\star)$ are both above $r$, the variance of 
        $S_1^{\xi_\star,\eta^\star}$ is bounded below by $\e=2r/(1-r)^2$.
        %twice the variance of a $\Geom(r)$ random variable.
        Since $\xi_\star\cdot\ek_1$ and $\eta^\star\cdot\ek_1$ are assumed to be in $(\delta,1-\delta)$, $\pxi(\xi_\star)$ and $\pxi(\eta^\star)$ are bounded away from $r$, uniformly in $N$, and thus for any $p\ge1$ there exists a finite constant 
        $D=D(\delta,p)$ such that
        %the $J$ weights are stochastically dominated by a $\Geom(r/(r+\kappa_\delta))$ random variable
        $\E[\abs{S_1^{\xi_\star,\eta^\star}-\mu}^p]\le D(\delta,p)$ for all $N\in\bbN$.
        Take $p\ge3$ large enough so that $\frac{p-2}{3(p+1)}>a_0/2$. 
		The conditions of Lemma \ref{lem:rwBound} are satisfied. If we take 
		$n = \fl{N^{2/3}}$, then \eqref{eq:rwbd1} gives, for $N$ large enough,
		\[
		\P\big\{S_1^{\xi_\star,\eta^\star} \leq 0, S_2^{\xi_\star,\eta^\star}\leq 0, \ldots, S^{\xi_\star,\eta^\star}_{\fl{N^{2/3}}} \leq 0\big\} \leq C \bigl((C_1 r^{-1}N^{-\expnt/2})\vee (\fl{N^{2/3}})^{-\frac{p-2}{2(p+1)}}\bigr)\le CC_1r^{-1}N^{-\expnt/2}.
		\]
		Repeating this same argument for $S_k^{\xi^\star,\eta_\star}$, $k\in\flint{-N^{2/3},0}$, yields
		\[
		\P(S_{-1}^{\xi^\star,\eta_\star} \leq 0, S_{-2}^{\xi^\star,\eta_\star} \leq 0, \ldots, S_{-N^{2/3}}^{\xi^\star,\eta_\star} \leq 0) \leq CC_1r^{-1} N^{-\expnt/2}.
		\]
		Bound \eqref{bd:RW} follows from the independence proved in Corollary \ref{lem:RWIndependent}. 
		The lemma is proved.
\end{proof}

\bibliography{firasbib2010}

\def\bysame{\leavevmode ---------\thinspace}
\makeatletter\if@francais\providecommand{\og}{<<~}\providecommand{\fg}{~>>}
\else\gdef\og{``}\gdef\fg{''}\fi\makeatother
\def\cdrandname{\&}
\providecommand\cdrnumero{no.~}
\providecommand{\cdredsname}{eds.}
\providecommand{\cdredname}{ed.}
\providecommand{\cdrchapname}{chap.}
\providecommand{\cdrmastersthesisname}{Memoir}
\providecommand{\cdrphdthesisname}{PhD Thesis}
\begin{thebibliography}{10}

\bibitem{Ahl-Hof-16-}
{\scshape D.~Ahlberg {\normalfont \cdrandname}~C.~Hoffman}, {\og Random
  coalescing geodesics in first-passage percolation.\fg},  (2016), Preprint
  (\href{https://arxiv.org/abs/1609.02447}{\tt arXiv 1609.02447}).

\bibitem{Ale-20-}
{\scshape K.~Alexander}, {\og {Geodesics, bigeodesics, and coalescence in first
  passage percolation in general dimension}\fg}, \emph{arXiv e-prints} (2020),
  article no.~arXiv:2001.08736, \url{https://arxiv.org/abs/2001.08736},
  \href{https://arxiv.org/abs/2001.08736}{\tt arXiv 2001.08736}.

\bibitem{Auf-Dam-Han-17}
{\scshape A.~Auffinger, J.~Hanson {\normalfont \cdrandname}~M.~Damron},
  \emph{50 years of first passage percolation}, University Lecture Series,
  vol.~68, American Mathematical Society, Providence, RI, 2017, 161~pages.

\bibitem{Bak-Cat-Kha-14}
{\scshape Y.~Bakhtin, E.~Cator {\normalfont \cdrandname}~K.~Khanin}, {\og
  Space-time stationary solutions for the {B}urgers equation\fg}, \emph{J.
  Amer. Math. Soc.} \textbf{27} (2014), \cdrnumero 1, p.~193-238.

\bibitem{Bak-Kha-18}
{\scshape Y.~Bakhtin {\normalfont \cdrandname}~K.~Khanin}, {\og On global
  solutions of the random {H}amilton-{J}acobi equations and the {KPZ}
  problem\fg}, \emph{Nonlinearity} \textbf{31} (2018), \cdrnumero 4,
  p.~R93-R121.

\bibitem{Bak-Li-18}
{\scshape Y.~Bakhtin {\normalfont \cdrandname}~L.~Li}, {\og Zero temperature
  limit for directed polymers and inviscid limit for stationary solutions of
  stochastic {B}urgers equation\fg}, \emph{J. Stat. Phys.} \textbf{172} (2018),
  \cdrnumero 5, p.~1358-1397.

\bibitem{Bak-Li-19}
\bysame , {\og Thermodynamic limit for directed polymers and stationary
  solutions of the {B}urgers equation\fg}, \emph{Comm. Pure Appl. Math.}
  \textbf{72} (2019), \cdrnumero 3, p.~536-619.

\bibitem{Bal-Bus-Sep-20}
{\scshape M.~Bal\'azs, O.~Busani {\normalfont \cdrandname}~T.~Sepp\"al\"ainen},
  {\og Non-existence of bi-infinite geodesics in the exponential corner growth
  model\fg}, \emph{Forum Math. Sigma} \textbf{8} (2020), \cdrnumero 46
  (\href{https://arxiv.org/abs/1909.06883}{\tt arXiv 1909.06883}).

\bibitem{Bas-Hof-Sly-18-}
{\scshape R.~Basu, C.~Hoffman {\normalfont \cdrandname}~A.~Sly}, {\og
  Nonexistence of Bigeodesics in Integrable Models of Last Passage
  Percolation\fg},  (2018), \url{https://arxiv.org/abs/1811.04908}, Preprint
  (\href{https://arxiv.org/abs/1811.04908}{\tt arXiv 1811.04908}).

\bibitem{Bich02}
{\scshape K.~Bichteler}, \emph{Stochastic Integration with Jumps}, Encyclopedia
  of Mathematics and its Applications, Cambridge University Press, 2002.

\bibitem{Bri-Dam-Han-20-}
{\scshape G.~Brito, M.~Damron {\normalfont \cdrandname}~J.~Hanson}, {\og
  {Absence of backward infinite paths for first-passage percolation in
  arbitrary dimension}\fg}, \emph{arXiv e-prints} (2020),  article
  no.~arXiv:2003.03367, \url{https://arxiv.org/abs/2003.03367},
  \href{https://arxiv.org/abs/2003.03367}{\tt arXiv 2003.03367}.

\bibitem{Cat-Pim-12}
{\scshape E.~Cator {\normalfont \cdrandname}~L.~P.~R. Pimentel}, {\og Busemann
  functions and equilibrium measures in last passage percolation models\fg},
  \emph{Probab. Theory Related Fields} \textbf{154} (2012), \cdrnumero 1-2,
  p.~89-125.

\bibitem{Cat-Pim-13}
\bysame , {\og Busemann functions and the speed of a second class particle in
  the rarefaction fan\fg}, \emph{Ann. Probab.} \textbf{41} (2013), \cdrnumero
  4, p.~2401-2425.

\bibitem{Cor-12}
{\scshape I.~Corwin}, {\og The {K}ardar-{P}arisi-{Z}hang equation and
  universality class\fg}, \emph{Random Matrices Theory Appl.} \textbf{1}
  (2012), \cdrnumero 1, p.~1130001, 76.

\bibitem{Cor-16}
\bysame , {\og Kardar-{P}arisi-{Z}hang universality\fg}, \emph{Notices Amer.
  Math. Soc.} \textbf{63} (2016), \cdrnumero 3, p.~230-239.

\bibitem{Dam-Han-14}
{\scshape M.~Damron {\normalfont \cdrandname}~J.~Hanson}, {\og Busemann
  functions and infinite geodesics in two-dimensional first-passage
  percolation\fg}, \emph{Comm. Math. Phys.} \textbf{325} (2014), \cdrnumero 3,
  p.~917-963.

\bibitem{Dam-Han-17}
\bysame , {\og Bigeodesics in {F}irst-{P}assage {P}ercolation\fg}, \emph{Comm.
  Math. Phys.} \textbf{349} (2017), \cdrnumero 2, p.~753-776.

\bibitem{Dau-Ort-Vir-19-}
{\scshape D.~Dauvergne, J.~Ortmann {\normalfont \cdrandname}~B.~Vir{\'a}g},
  {\og The {D}irected {L}andscape\fg}, \emph{Acta Math.} (To appear.),
  \href{https://arxiv.org/abs/1812.00309}{\tt arXiv 1812.00309}.

\bibitem{Dau-Vir-21-}
{\scshape D.~Dauvergne {\normalfont \cdrandname}~B.~Vir{\'a}g}, {\og The
  scaling limit of the longest increasing subsequence\fg},  (2021),
  \href{https://arxiv.org/abs/2104.08210}{\tt arXiv 2104.08210}.

\bibitem{Dur-10}
{\scshape R.~Durrett}, \emph{Probability: theory and examples}, fourth
  \cdredname, Cambridge Series in Statistical and Probabilistic Mathematics,
  vol.~31, Cambridge University Press, Cambridge, 2010, x+428~pages.

\bibitem{Emr-Jan-Sep-20-}
{\scshape E.~Emrah, C.~Janjigian {\normalfont \cdrandname}~T.~Sepp\"al\"ainen},
  {\og {Right-tail moderate deviations in the exponential last-passage
  percolation}\fg}, \emph{arXiv e-prints} (2020),  article
  no.~arXiv:2004.04285, \url{https://arxiv.org/abs/2004.04285},
  \href{https://arxiv.org/abs/2004.04285}{\tt arXiv 2004.04285}.

\bibitem{Emr-Jan-Sep-21-}
\bysame , {\og {Optimal-order exit point bounds in exponential last-passage
  percolation via the coupling technique}\fg}, \emph{arXiv e-prints} (2021),
  article no.~arXiv:2105.09402, \url{https://arxiv.org/abs/2105.09402},
  \href{https://arxiv.org/abs/2105.09402}{\tt arXiv 2105.09402}.

\bibitem{Emr-Jan-Xie-21-}
{\scshape E.~Emrah, C.~Janjigian {\normalfont \cdrandname}~Y.~Xie}, {\og
  {Moderate deviation and exit point estimates for integrable directed polymer
  models}\fg}, \emph{Preprint}.

\bibitem{Fan-Sep-20-}
{\scshape W.-T.~L. Fan {\normalfont \cdrandname}~T.~Sepp{{\"a}}l{{\"a}}inen},
  {\og Joint distribution of {B}usemann functions in the exactly solvable
  corner growth model.\fg}, \emph{Prob. Math. Phys.} (2020), To appear
  (\href{https://arxiv.org/abs/1808.09069}{\tt arXiv 1808.09069}).

\bibitem{feller2}
{\scshape W.~{Feller}}, \emph{{An Introduction to Probability Theory and Its
  Applications}}, second \cdredname, vol.~II, John Wiley \& Sons Inc., 1971.

\bibitem{Geo-Ras-Sep-17-ptrf-2}
{\scshape N.~Georgiou, F.~Rassoul-Agha {\normalfont
  \cdrandname}~T.~Sepp{\"a}l{\"a}inen}, {\og Geodesics and the competition
  interface for the corner growth model\fg}, \emph{Probab. Theory Related
  Fields} \textbf{169} (2017), \cdrnumero 1-2, p.~223-255.

\bibitem{Geo-Ras-Sep-17-ptrf-1}
\bysame , {\og Stationary cocycles and {B}usemann functions for the corner
  growth model\fg}, \emph{Probab. Theory Related Fields} \textbf{169} (2017),
  \cdrnumero 1-2, p.~177-222.

\bibitem{Hal-Tak-15}
{\scshape T.~Halpin-Healy {\normalfont \cdrandname}~K.~A. Takeuchi}, {\og A
  {KPZ} cocktail---shaken, not stirred \dots toasting 30 years of kinetically
  roughened surfaces\fg}, \emph{J. Stat. Phys.} \textbf{160} (2015), \cdrnumero
  4, p.~794-814.

\bibitem{Hof-05}
{\scshape C.~Hoffman}, {\og Coexistence for {R}ichardson type competing spatial
  growth models\fg}, \emph{Ann. Appl. Probab.} \textbf{15} (2005), \cdrnumero
  1B, p.~739-747.

\bibitem{Hof-08}
\bysame , {\og Geodesics in first passage percolation\fg}, \emph{Ann. Appl.
  Probab.} \textbf{18} (2008), \cdrnumero 5, p.~1944-1969.

\bibitem{How-New-97}
{\scshape C.~D. Howard {\normalfont \cdrandname}~C.~M. Newman}, {\og Euclidean
  models of first-passage percolation\fg}, \emph{Probab. Theory Related Fields}
  \textbf{108} (1997), \cdrnumero 2, p.~153-170.

\bibitem{How-New-01}
\bysame , {\og Geodesics and spanning trees for {E}uclidean first-passage
  percolation\fg}, \emph{Ann. Probab.} \textbf{29} (2001), \cdrnumero 2,
  p.~577-623.

\bibitem{Jan-Nur-Ras-20-}
{\scshape C.~Janjigian, S.~Nurbavliyev {\normalfont
  \cdrandname}~F.~Rassoul-Agha}, {\og A shape theorem and a variational formula
  for the quenched {L}yapunov exponent of random walk in a random
  potential.\fg},  (2020), Preprint
  (\href{https://arxiv.org/abs/2006.10871}{\tt arXiv 2006.10871}).

\bibitem{Jan-Ras-20-aop}
{\scshape C.~Janjigian {\normalfont \cdrandname}~F.~Rassoul-Agha}, {\og
  Busemann functions and {G}ibbs measures in directed polymer models on
  {$\mathbb Z^2$}\fg}, \emph{Ann. Probab.} \textbf{48} (2020), \cdrnumero 2,
  p.~778-816.

\bibitem{Jan-Ras-20-jsp}
\bysame , {\og Uniqueness and {E}rgodicity of {S}tationary {D}irected
  {P}olymers on {$\mathbb{Z}^2$}\fg}, \emph{J. Stat. Phys.} \textbf{179}
  (2020), \cdrnumero 3, p.~672-689.

\bibitem{Jan-Ras-Sep-21-}
{\scshape C.~Janjigian, F.~Rassoul-Agha {\normalfont
  \cdrandname}~T.~Sepp\"al\"ainen}, {\og Geometry of geodesics through
  {B}usemann measures in directed last-passage percolation\fg}, \emph{J. Eur.
  Math. Soc. (JEMS)} (2021), To appear
  (\href{https://arxiv.org/abs/1908.09040}{\tt arXiv 1908.09040}).

\bibitem{kes-86-stflour}
{\scshape H.~Kesten}, {\og Aspects of first passage percolation\fg}, in
  \emph{\'{E}cole d'{\'e}t{\'e} de probabilit{\'e}s de {S}aint-{F}lour,
  {XIV}---1984}, Lecture Notes in Math., vol. 1180, Springer, Berlin, 1986,
  p.~125-264.

\bibitem{randomWalkBook}
{\scshape G.~F. {Lawler} {\normalfont \cdrandname}~V.~{Limic}}, \emph{Random
  Walk: A Modern Introduction}, Cambridge Studies in Advanced Mathematics,
  Cambridge University Press, 2010.

\bibitem{Lic-New-96}
{\scshape C.~Licea {\normalfont \cdrandname}~C.~M. Newman}, {\og Geodesics in
  two-dimensional first-passage percolation\fg}, \emph{Ann. Probab.}
  \textbf{24} (1996), \cdrnumero 1, p.~399-410.

\bibitem{Mar-04}
{\scshape J.~B. Martin}, {\og Limiting shape for directed percolation
  models\fg}, \emph{Ann. Probab.} \textbf{32} (2004), \cdrnumero 4,
  p.~2908-2937.

\bibitem{New-95}
{\scshape C.~M. Newman}, {\og A surface view of first-passage percolation\fg},
  in \emph{Proceedings of the {I}nternational {C}ongress of {M}athematicians,
  {V}ol.\ 1, 2 ({Z}{\"u}rich, 1994)} (Basel), Birkh{\"a}user, 1995,
  p.~1017-1023.

\bibitem{Qua-12}
{\scshape J.~Quastel}, {\og Introduction to {KPZ}\fg}, in \emph{Current
  developments in mathematics, 2011}, Int. Press, Somerville, MA, 2012,
  p.~125-194.

\bibitem{Qua-Spo-15}
{\scshape J.~Quastel {\normalfont \cdrandname}~H.~Spohn}, {\og The
  one-dimensional {KPZ} equation and its universality class\fg}, \emph{J. Stat.
  Phys.} \textbf{160} (2015), \cdrnumero 4, p.~965-984.

\bibitem{RGMpt5}
{\scshape T.~{Sepp{\"a}l{\"a}inen}}, {\og {The corner growth model with
  exponential weights}\fg}, \emph{Proc. Sympos. Appl. Math.} \textbf{75}
  (2018).

\end{thebibliography}
\end{document}